\documentclass{article}
\usepackage[dvipsnames]{xcolor}
\usepackage{gastex}
\usepackage{amsmath}
\usepackage{amssymb}
\usepackage{upref}
\usepackage{multirow}
\usepackage{lineno}
    \linenumbers
\usepackage{makeidx}
\makeindex
\usepackage{theorem}
\newtheorem{theorem}{Theorem}
\newtheorem{lemma}[theorem]{Lemma}
\newtheorem{proposition}[theorem]{Proposition}
\newtheorem{corollary}[theorem]{Corollary}
{\theorembodyfont{\rmfamily}%
  \newtheorem{example}[theorem]{Example}
   }
\newenvironment{proof}{\noindent\textit{Proof.}}
{\QED\vskip\theorempostskipamount} 
\newenvironment{proofof}[1]{\noindent\textit{Proof
    \protect{#1}.}}
                       {\QED\vskip\theorempostskipamount}
\def\petitcarre{\vrule height4pt width 4pt depth0pt}
\def\QED{\relax\ifmmode\eqno{\hbox{\petitcarre}}\else{%
  \unskip\nobreak\hfil\penalty50\hskip2em\hbox{}\nobreak\hfil
  \petitcarre
  \parfillskip=0pt \finalhyphendemerits=0\par\smallskip}
  \fi}
\newcommand\A{\mathcal{A}}
\newcommand\B{\mathcal{B}}

\newcommand\D{\mathcal{D}}
\def\H{\mathcal{H}}

\newcommand\LL{\mathcal{L}}
\newcommand\RR{\mathcal{R}}

\newcommand{\N}{\mathbb{N}}
\newcommand{\Z}{\mathbb{Z}}
\newcommand{\R}{\mathbb{R}}
\let\pars\delta 
\let\proba\pi

\def\u(#1){\underline{#1}\,}
\DeclareMathOperator{\Card}{Card}
\DeclareMathOperator{\rep}{rep}

\DeclareMathOperator{\Pal}{Pal}

\DeclareMathOperator{\End}{End}
\DeclareMathOperator{\Fact}{Fact}
\DeclareMathOperator{\rank}{rank}
\def\Im{\text{\upshape{Im}}}
\definecolor{ivoire}{rgb}{0.99,0.99,0.8}
\usepackage{calc}
\newcounter{hours}\newcounter{minutes}
\newcommand\computetime{\setcounter{hours}{\time/60}%
  \setcounter{minutes}{\time-\value{hours}*60}%
  \thehours\,h\,\theminutes}
\newcommand\dateandtime{\today\quad\computetime}
\usepackage[hypertex,hyperindex,pagebackref,final]{hyperref}
\numberwithin{theorem}{subsection}
\numberwithin{equation}{section}
\numberwithin{figure}{section}
\numberwithin{table}{section}
\title{Bifix codes and Sturmian words}
\author{Jean Berstel$^1$, Clelia De Felice$^2$, Dominique Perrin$^1$,
  \\
Christophe  Reutenauer$^3$,
Giuseppina Rindone$^1$\\\\
$^1$Universit\'e Paris Est, $^2$Universit\`a degli Studi di Salerno,\\
$^3$ Universit\'e du Qu\'ebec \`a Montr\'eal}
\date{\dateandtime}
\begin{document}
\makeatletter
\def\@listI{%
  \leftmargin\leftmargini
  \setlength{\parsep}{0pt plus 1pt minus 1pt}
  \setlength{\topsep}{2pt plus 1pt minus 1pt}
  \setlength{\itemsep}{0pt}
}
\let\@listi\@listI
\@listi
\def\@listii {%
  \leftmargin\leftmarginii
  \labelwidth\leftmarginii
  \advance\labelwidth-\labelsep
  \setlength{\topsep}{0pt plus 1pt minus 1pt}
}
\def\@listiii{%
  \leftmargin\leftmarginiii
  \labelwidth\leftmarginiii
  \advance\labelwidth-\labelsep
  \setlength{\topsep}{0pt plus 1pt minus 1pt}
  \setlength{\parsep}{0pt} 
  \setlength{\partopsep}{1pt plus 0pt minus 1pt}
}
\makeatother
\maketitle

\begin{abstract} We prove new results concerning the relation between
  bifix codes, episturmian words and subgroups of free groups.  We
  study bifix codes in factorial sets of words. We generalize most
  properties of ordinary maximal bifix codes to bifix codes maximal in
  a recurrent set $F$ of words ($F$-maximal bifix codes).  In the case
  of bifix codes contained in Sturmian sets of words, we obtain
  several new results. Let $F$ be a Sturmian set of words, defined as
  the set of factors of a strict episturmian word.  Our results
  express the fact that an $F$-maximal bifix code of degree $d$
  behaves just as the set of words of $F$ of length $d$. An
  $F$-maximal bifix code of degree $d$ in a Sturmian set of words on
  an alphabet with $k$ letters has $(k-1)d+1$ elements.  This
  generalizes the fact that a Sturmian set contains $(k-1)d+1$ words
  of length $d$. Moreover, given an infinite word $x$, if there is a
  finite maximal bifix code $X$ of degree $d$ such that $x$ has at
  most $d$ factors of length $d$ in $X$, then $x$ is ultimately
  periodic. Our main result states that any $F$-maximal bifix code of
  degree $d$ on the alphabet $A$ is the basis of a subgroup of index
  $d$ of the free group on~$A$.
\end{abstract}

\tableofcontents
\section{Introduction}

This paper studies a new relation between three objects previously
unrelated altogether: bifix codes, epiturmian words and subgroups of
free groups. 

We first give some elements on the background of the first two.  The
study of bifix codes goes back to founding papers by
Sch\"utzenberger~\cite{Schutzenberger1956} and by Gilbert and
Moore~\cite{GilbertMoore1959}. These papers already contain
significant results. The first systematic study is in the papers of
Sch\"utzenberger \cite{Schutzenberger1961c,Schutzenberger1961b}. The
general idea is that the submonoids generated by bifix codes are an
adequate generalization of the subgroups of a group. This is
illustrated by the striking fact that, under a mild restriction, the
average length of a maximal bifix code with respect to a Bernoulli
distribution on the alphabet is an integer. Thus, in some sense a
maximal bifix code behaves as the uniform code formed of all the words
of a given length. The theory of bifix codes was developed in a
considerable way by C\'esari. He proved that all the finite maximal
bifix codes may be obtained by internal transformations from uniform
codes \cite{Cesari1972}. He also defined the notion of derived code
which allows to build maximal bifix codes by increasing
degrees~\cite{Cesari1979}.

Sturmian words are infinite words over a binary alphabet that have
exactly $n+1$ factors of length $n$ for each $n\ge 0$.  Their origin
can be traced back to the astronomer J. Bernoulli III.  Their first
in-depth study is by Morse and Hedlund~\cite{MorseHedlund1940}. Many
combinatorial properties were described in the paper by Coven and
Hedlund~\cite{CovenHedlund1973}.  
Note that, although Sturmian words appear first in the work of
Morse and Hedlund, their finitary version, Christoffel and standard
words, appear much before in the work of
Christoffel~\cite{Christoffel1875} and, apparently independently, in
the work of Markoff~\cite{Markoff1879,Markoff1880}; the latter
constructed the famous Markoff numbers by using them. The Markoff
theory (which was designed to study minima's of quadratic forms) was
revisited often by mathematicians, notably by
Frobenius~\cite{Frobenius1913}, Dickson~\cite{Dickson1930},
H. Cohn~\cite{Cohn1972}, Cusick and Flahive~\cite{CusickFlahive1989} and
Bombieri~\cite{Bombieri2007}. There, the connection with the free
group on two generators was established. Other connection of
Christoffel words with the free group may be found in
Osborne and Zieschang~\cite{OsborneZieschang1981} and
Kassel and Reutenauer~\cite{KasselReutenauer2007}. Moreover, the Sturmian
morphisms (substitutions that preserve Sturmian words) are the
positive endomorphisms of the free group on two generators, see
Wen and Wen~\cite{WenWen1994}, Mignosi and S\'e\'ebold~\cite{MignosiSeebold1993}.  Thus
Sturmian words are closely related to the free group. This connection is
one of the main points of the present paper.

Sturmian words were generalized to arbitrary
alphabets.  Following an initial work by Arnoux and Rauzy
\cite{ArnouxRauzy1991} and developing ideas of De
Luca~\cite{deLuca1997}, Droubay, Justin and Pirillo introduced
in~\cite{DroubayJustinPirillo2001} the notion of episturmian words
which generalizes Sturmian words to arbitrary finite alphabets.

In this paper, we consider the extension of the results known for
bifix codes maximal in the free monoid to bifix codes maximal in more
restricted sets of words, and in particular the sets of factors of
episturmian words.

We extend most properties of ordinary maximal bifix codes to bifix
codes that are maximal in a recurrent set $F$ of words ($F$-maximal bifix
codes). We show in particular that the average length of a finite
$F$-maximal bifix code of degree $d$ in a recurrent set $F$ with
respect to an invariant probability distribution on $F$ is equal to
$d$ (Corollary~\ref{corollaryAveragelength}).

Our main objective is the case of the set of factors of an episturmian
word.  We actually work with the set of factors of a strict
episturmian word, called simply a Sturmian set. The number of factors
of length $d$ of a strict episturmian word over an alphabet of $k$
letters is known to be $(k-1)d+1$. Our main result is that a maximal
bifix code of degree $d$ in a Sturmian set over an alphabet of $k$
letters is always a basis of a subgroup of index $d$ of the free group
(Theorem~\ref{theoremGroups}).  In particular, it has $(k-1)d+1$
elements (Theorem~\ref{theoremBifixd+1}). Since the set of all words of
length $d$ is a maximal bifix code of degree $d$, this yields a strong
generalization of the previous property. In particular, every finite
maximal bifix code of degree $d$ over a two letter alphabet contains
exactly $d+1$ factors of any Sturmian word.

Finally, bifix codes $X$ contained
in restricted sets of words are used to study the groups in the
syntactic monoid of the submonoid $X^*$ (Theorem~\ref{newTheorem}).
This aspect was first considered by Sch\"utzenberger in
\cite{Schutzenberger1979}.  He has studied the conditions under which
parameters linked with the syntactic monoid $M$ of a finitely
generated submonoid $X^*$ of a free monoid $A^*$ can be bounded in
terms of $\Card(X)$ only. One of his results is that, apart from a
special case where the group is cyclic, the cardinality of a group
contained in $M$ is such a parameter. In~\cite{Schutzenberger1979},
Sch\"utzenberger conjectured a refinement of his result which was
subsequently proved by C\'esari. This study led to the Critical
Factorization Theorem that we will meet again here
(Theorem~\ref{theoremCriticalFactorization}).

The extension of the results concerning codes in free monoids to codes
in a restricted set of words has already been considered by several
authors. However, most of them have focused on general codes rather
than on the particular class of bifix codes.  In~\cite{Reutenauer1986}
the notion of codes of paths in a graph has been introduced. Such
paths can also be viewed as words in a restricted set. The notion of a
bifix code of paths has been studied in~\cite{DeFelice1988} where the
internal transformation is generalized.  In~\cite{Restivo1990}, the
notion of code in a factorial set of words was introduced. The
definition of a code $X$ in a factorial set $F$ requires that the set
$X^*$ of all concatenations of words in $X$ is included in $F$. This
approach was pushed further in~\cite{HongShin2009}.  A more general
notion was considered in~\cite{BealPerrin2005}. It only requires that
$X\subset F$ and that no word of $F$ has two distinct factorizations
but not necessarily that $X^*\subset F$. The connection with
unambiguous automata was considered later in~\cite{BealPerrin2009}.
Codes in Sturmian sets have been studied before
in~\cite{CarpideLuca2005}.  Finally, prefix codes $X$ contained in
restricted sets of words are used in~\cite{PerrinRindone2003} to study
the groups in the syntactic monoid of the submonoid $X^*$.

Our paper is organized as follows. 

In a first section (Section~\ref{sectionFactorialSets}), we recall some
definitions
concerning prefix-closed, factorial, recurrent and uniformly recurrent
sets, in relation with infinite words.
We also introduce probability distributions on these sets.

In Section~\ref{sectionPrefixCodes}, we introduce prefix codes in
factorial sets, especially maximal ones. We introduce some basic
notions on automata. We define the average
length with respect to a probability distribution on the factorial
set.

In Section~\ref{sectionBifixCodes}, we develop the theory of maximal
bifix codes in recurrent sets. We generalize most of the properties
known in the classical case. In particular, we show that the notion of
degree and that of derived code can be defined
(Theorem~\ref{thmDerived}). We show that, for a uniformly recurrent set
$F$, any $F$-thin bifix code contained in $F$ is finite
(Theorem~\ref{theoremCompletion}).  In the case of Sturmian
sets, we prove our main results.  First, a bifix code of degree $d$
maximal in a Sturmian set on a $k$-letter alphabet has $(k-1)d+1$
elements (Theorem~\ref{theoremBifixd+1}).  Next, given an infinite
word $x$, if there is a finite maximal bifix code $X$ of degree $d$
such that $x$ has at most $d$ factors of length $d$ in $X$, then $x$
is ultimately periodic (Corollary~\ref{corollaryPeriod}). The proof uses
the Critical Factorization Theorem
(see e.g.~\cite{Lothaire1983,CrochemorePerrin1991}).

Section~\ref{sectionBasis} presents our results concerning free
groups.  Our main result (Theorem~\ref{theoremGroups}) in this area
states that for a Sturmian set $F$, a bifix code $X\subset F$ is
a finite and $F$-maximal bifix code of $F$-degree $d$ if and only if
it is a basis of a subgroup of index $d$ of the free group on $A$.  We
finally present in Section~\ref{sectionSyntacticGroups} a consequence
of Theorem~\ref{theoremGroups} concerning syntactic groups.
We show that any
transitive
permutation group of degree $d$ which can be generated by $k$ elements
is a syntactic group of a bifix code with $(k-1)d+1$ elements
(Theorem~\ref{newTheorem}).

Many results of this paper are extensions or generalizations of
results contained in~\cite{BerstelPerrinReutenauer2009}. We always
give the reference of the corresponding result
in~\cite{BerstelPerrinReutenauer2009}. The proofs sometimes consist in
the verification that the proof of the book still holds in the more
general setting, and sometimes require new and more involved
developments. In order to make the paper self contained, and to avoid
repetitive references to the book, we have tried to always give
complete proofs.

\section{Factorial sets}\label{sectionFactorialSets}

In this section, we introduce the basic notions of prefix-closed,
factorial, recurrent and uniformly recurrent sets. These form a
descending hierarchy. These notions are closely related with the
analogous notions for infinite words which are defined in
Section~\ref{subsectionRecurrentWords}.  In
Section~\ref{subsectionProbas}, we introduce probability distributions
on factorial sets.

We use the standard terminology and notation on words, in particular
concerning prefixes, suffixes and factors (see~\cite{Lothaire1983} for
example). Let $A$ be a finite alphabet. All words considered below are
supposed to be on the alphabet $A$. We denote by $1$ the empty word.
We denote by $A^*$ the set of all words on $A$ and by $A^+$ the set of
nonempty words.

The \emph{reversal}\index{reversal}\index{word!reversal} of a word
$w=a_1a_2\cdots a_n$, where $a_1,a_2,\ldots,a_n$ are letters, is the
word $\widetilde w=a_n\cdots a_2a_1$. In particular, the reversal of
the empty word is the empty word.  A set $X$ of words is \emph{closed
  under reversal}\index{closed under reversal, set} if it contains the
reversals of its elements.

Given a set $X$ of words, we
define, for a word $u$, the set $u^{-1}X$ by
\begin{displaymath}
  u^{-1}X= \{y\in A^*\mid uy\in X\}\,.
\end{displaymath}
Next, we say that a word is a \emph{prefix of}\index{prefix of a set}
$X$ if it is a prefix of a word of $X$.

A nonempty set $F\subset A^*$ of words is said to be
\emph{prefix-closed}\index{prefix-closed set} if it contains the prefixes
of all its elements. Symmetrically, it is said to be
\emph{suffix-closed}\index{suffix-closed set} if it contains the suffixes
of all its elements.  It is said to be
\emph{factorial}\index{factorial set} if it contains the factors of all
its elements.

The \emph{right} (resp. \emph{left}) \emph{order}\index{left
  order}\index{right order}\index{order!left}\index{order!right} of a
word $w$ with respect to $F$ is the number of letters $a$ such that
$wa\in F$ (resp. $aw\in F$).

A set $F$ is said to be \emph{right essential}\index{right
  essential}\index{essential!right} if it is prefix-closed and if any
$w\in F$ has right order at least $1$. If $F$ is right essential, then
for any $u\in F$ and any integer $n\ge 1$, there is a word $v$ of
length $n$ such that $uv\in F$.  Symmetrically, a set $F$ is said to
be \emph{left essential}\index{left essential}\index{essential!left}
if it is suffix-closed and if any $w\in F$ has left order at
least~$1$.

\subsection{Recurrent sets}

A set $F$ of words is said to be \emph{recurrent}\index{recurrent set}
if it is factorial and if for every $u,w\in F$ there is a $v\in F$
such that $uvw\in F$. A recurrent set $F\ne\{1\}$ is right and left
essential.

\begin{example}\label{exampleFull}
  The set $F=A^*$ is recurrent.
\end{example}

\begin{example}\label{exampleGolden}
  Let $A=\{a,b\}$. Let $F$ be the set of words on $A$ without factor
  $bb$. Thus $F=A^*\setminus A^*bbA^*$. The set $F$ is
  recurrent. Indeed, if $u,w\in F$, then $uaw\in F$.
\end{example}

A set $F$ is said to be \emph{uniformly recurrent}\index{uniformly
  recurrent} if it is factorial and right essential and if, for any
word $u\in F$, there exists an integer $n\ge 1$ such that $u$ is a
factor of every word in $F\cap A^n$.

\begin{proposition}\label{uniformImpliesRecurrent}
  A uniformly recurrent set is recurrent.
\end{proposition}

\begin{proof}
  Let $u,w\in F$. Let $n$ be such that $w$ is a factor of any word in
  $F\cap A^n$. Since $F$ is right essential, there is a word $v$ of
  length $n$ such that $uv\in F$. Since $w$ is a factor of $v$, we
  have $v=rws$ for some words $r,s$. Thus $urw\in F$.
\end{proof}

The converse of Proposition~\ref{uniformImpliesRecurrent} is not true
as shown in the example below.

\begin{example}
  The set $F=A^*$ on $A=\{a,b\}$ is recurrent but not uniformly
  recurrent since $b\in F$ but $b$ is not a factor of $a^n\in F$ for
  any $n\ge 1$.
\end{example}


\subsection{Recurrent words}\label{subsectionRecurrentWords}

We denote by $F(x)$ the set of factors of an infinite word $x\in
A^\N$.  The set $F(x)$ is factorial and right essential.

An infinite word $x\in A^\N$ is said to be
\emph{recurrent}\index{recurrent word} if for any word $u\in F(x)$
there is a $v\in F(x)$ such that $uvu\in F(x)$. Equivalently, each factor of
a recurrent word $x$ has an infinite
number of occurrences in $x$.

\begin{proposition}\label{propExists}
  For any recurrent set $F$ there is an infinite word $x$ such that $F(x)=F$.
\end{proposition}

\begin{proof}
  Set $F=\{u_1,u_2,\ldots\}$. Since $F$ is recurrent and $u_1,u_2\in
  F$, there is a word $v_1$ such that $u_1v_1u_2\in F$. Further, since
  $u_1v_1u_2,u_3\in F$ there is a word $v_2$ such that
  $u_1v_1u_2v_2u_3\in F$. In this way, we obtain an infinite word
  $x=u_1v_1u_2v_2\cdots$ such that $F(x)=F$.
\end{proof}

\begin{proposition}
  For any infinite word $x$, the set $F(x)$ is recurrent if and only
  if $x$ is recurrent.
\end{proposition}

\begin{proof}
  Set $F=F(x)$.  Suppose first that $F$ is recurrent. For any $u$ in
  $F$, there is a $v\in F$ such that $uvu\in F$. Thus $x$ is
  recurrent. Conversely, assume that $x$ is recurrent. Let $u,v$ be in
  $F$. Then there is a factorization $x=puy$ with $p\in F$ and $y\in
  A^\N$. Since $x$ is recurrent, the word $v$ is a factor of $y$. Set
  $y=qvz$ with $q\in F$ and $z\in A^\N$. Then $uqv$ is in $F$. Thus
  $F$ is recurrent.
\end{proof}

An infinite word $x\in A^\N$ is said to be \emph{uniformly
  recurrent}\index{uniformly recurrent word} if the set $F(x)$ is
uniformly recurrent. There exist recurrent infinite words which are
not uniformly recurrent, as shown in the following example.

\begin{example}
  Let $x$ be the infinite word obtained by concatenating all binary
  words in radix order: by increasing length, and for each length in
  lexicographic order. Thus, $x$ starts as follows.
  \begin{displaymath}
    x=ab\,aaabbabb\,aaaaababaabbbaababbbabbb\cdots
  \end{displaymath}
  The infinite word $x$ is recurrent since every factor occurs
  infinitely often.  However, $x$ is not uniformly recurrent since
  each $a^n$, for $n>1$, is a factor of $x$, thus two consecutive occurrences
  of say the letter $b$ may be arbitrarily far one from each
  other. The word $x$ is closely related to the Champernowne word
  \cite{Champernowne1933}.
\end{example}

We use indifferently the terms of \emph{morphism}\index{morphism} or
\emph{substitution}\index{substitution} for a monoid morphism from
$A^*$ into itself.  Let $f:A^*\to A^*$ be a morphism and assume there
is a letter $a\in A$ such that $f(a)\in aA^+$. The words $f^n(a)$ for
$n\ge 1$ are prefixes of one another. If $|f^n(a)|\to\infty$ with $n$, then
we denote by $f^\omega(a)$ the infinite word which has all $f^n(a)$ as
prefixes. It is called a \emph{fix-point}\index{fix-point of a morphism}
of $f$. 

\begin{example}\label{exampleMorse}
  Set $A=\{a,b\}$. The \emph{Thue--Morse
    morphism}\index{Thue--Morse!morphism} is the substitution
  $f:A^*\to A^*$ defined by $f(a)=ab$ and $f(b)=ba$.  The
  \emph{Thue--Morse word}\index{Thue--Morse!word} $x=abbabaab\cdots$
  is the fix-point $f^\omega(a)$ of $f$. It is uniformly recurrent
  (see~\cite{Lothaire2002} Example 1.5.10). We call \emph{Thue--Morse
    set}\index{Thue--Morse!set} the set of factors of the Thue--Morse
  word. 
\end{example}

An infinite word $x\in A^\N$
\emph{avoids}\index{avoid}\index{word!avoiding a set} a set $X$ of
words if $F(x)\cap X=\emptyset$. We denote by $S_X$ the set of
infinite words avoiding a set $X\subset A^*$.  A (one sided)
\emph{shift space}\index{shift space} is a set $S$ of infinite words of
the form $S_X$ for some $X\subset A^*$.

A shift space $S\subset A^\N$ is \emph{minimal}\index{shift
  space!minimal}\index{minimal shift space} if for any shift space
$T\subset S$, one has $T=\emptyset$ or $T=S$.

For any infinite word $x\in A^\N$, we denote by $S(x)$ the set of
infinite words $y\in A^\N$ such that $F(y)\subset F(x)$. The set
$S(x)$ is a shift space. Indeed, we have $y\in S(x)$ if and only if
$F(y)\subset F(x)$ or equivalently $F(y)\cap X=\emptyset$ for
$X=A^*\setminus F(x)$.

The following property is standard (see for
example~\cite{Lothaire2002} Theorem 1.5.9). 

\begin{proposition}\label{propMinimal}
  An infinite word $x\in A^\N$ is uniformly recurrent if and only if
  $S(x)$ is minimal.
\end{proposition}

\subsection{Episturmian words}\label{subsection-episturmian}

A \emph{Sturmian word}\index{Sturmian!word} is an infinite word $x$ on
a binary alphabet $A$ such that the set $F(x)\cap A^n$ has $n+1$
elements for any $n\ge 0$.

\begin{example}\label{exampleFibonacci}
  Set $A=\{a,b\}$.  The \emph{Fibonacci
    morphism}\index{Fibonacci!morphism} is the substitution $f:A^*\to
  A^*$ defined by $f(a)=ab$ and $f(b)=a$.  The \emph{Fibonacci
    word}\index{Fibonacci!word}
  \begin{displaymath}
    x=abaab aba abaab abaababa abaababaabaab\cdots
  \end{displaymath}
  is the fix-point $f^\omega(a)$ of $f$.  It is a Sturmian word
  (see~\cite{Lothaire2002} Example~2.1.1). We call \emph{Fibonacci
    set}\index{Fibonacci!set} the set of factors of the Fibonacci word.
\end{example}

\emph{Episturmian words} are an extension of Sturmian words to
arbitrary finite alphabets. 

Recall that, given a set $F$ of words over an alphabet $A$, the right
(resp. left) order of a word $u$ in $F$ is the number of letters $a$ such
that $ua\in F$ (resp. $au\in F$). A word $u$ is
\emph{right-special}\index{right-special word}
(resp. \emph{left-special}\index{left-special word}) if its right
order (resp. left order) is at least $2$. A right-special (resp.
left-special) word is \emph{strict}\index{strict right-special
  word}\index{strict left-special word} if its right (resp. left)
order is equal to $\Card(A)$. In the case of a $2$-letter alphabet,
all special words are strict.

By definition, an infinite word $x$ is
\emph{episturmian}\index{episturmian word} if $F(x)$ is
closed under reversal and if $F(x)$ contains, for each $n\ge 1$, at
most one word  of length $n$ which is right-special. 

Since $F(x)$ is closed under reversal, the reversal of a right-special
factor of length $n$ is left-special, and it is the only left-special
factor of length $n$ of $x$. A suffix of a right-special factor is
again right-special. Symmetrically, a prefix of a left-special factor
is again left-special.

As a particular case, a \emph{strict}%
\index{strict episturmian word}\index{episturmian word!strict}
episturmian word is an episturmian word $x$ with the two following
properties: $x$ has exactly one right-special factor of each length
and moreover each right-special factor $u$ of $x$ is strict, that is
satisfies the inclusion $uA\subset F(x)$
(see~\cite{DroubayJustinPirillo2001}).

It is easy to see that for a strict episturmian word $x$ on an
alphabet $A$ with $k$ letters, the set $F(x)\cap A^n$ has $(k-1)n+1$
elements for each $n$. Thus, for a binary alphabet, the strict
episturmian words are just the Sturmian words, since a Sturmian word
has one right-special factor for each length and its set of
factors is closed under reversal.

An episturmian word $s$ is called \emph{standard}%
\index{standard episturmian word}\index{episturmian word!standard} if
all its left-special factors are prefixes of $s$. For any episturmian
word $s$, there is a standard one $t$ such that $F(s)=F(t)$. This is a
rephrasing of Theorem 5 in \cite{DroubayJustinPirillo2001}.

\begin{example}\label{exampleTribonacci}
  Consider the following generalization of the Fibonacci word to the
  ternary alphabet $A=\{a,b,c\}$. Consider the morphism
  $f:A^*\rightarrow A^*$ defined by $f(a)=ab$, $f(b)=ac$ and
  $f(c)=a$. The fix-point
  \begin{displaymath}
    f^\omega(a)=abac aba abac ab abac aba abac abac aba abac ab\cdots
  \end{displaymath}
  is the \emph{Tribonacci word}\index{Tribonacci word}. It is a strict
  standard episturmian word (see~\cite{JustinVuillon2000}).
\end{example}

The following is, in the case of Sturmian words, Proposition 2.1.25
in~\cite{Lothaire2002}. The general case results from Theorems~2 and~5
in \cite{DroubayJustinPirillo2001}.

\begin{proposition}\label{propositionSturmianMinimal}
  An episturmian word $x$ is uniformly recurrent and $S(x)$ is
  minimal.
\end{proposition}

The converse is false as shown by the following example.

\begin{example}
  The Thue--Morse word of Example~\ref{exampleMorse} is not Sturmian.
  Indeed, it has four factors of length $2$.
\end{example}

We recall now some notions and properties concerning episturmian
words. A detailed exposition with proofs is given
in~\cite{JustinVuillon2000,DroubayJustinPirillo2001,JustinPirillo2002,JustinPirillo2004}. See
also the survey paper~\cite{GlenJustin2009}.  For $a\in A$, denote by
$\psi_a$ the morphism  of $A^*$ into itself, called
\emph{elementary morphism}\index{elementary
  morphism}\index{morphism!elementary}, defined by
\begin{displaymath}
\psi_a(b)=\begin{cases}ab&\text{if $b\ne a$}\\
                       a&\text{otherwise}
         \end{cases}
\end{displaymath}
Let $\psi:A^*\rightarrow \End(A^*)$ be the morphism from $A^*$ into
the monoid of endomorphisms of $A^*$ which maps each $a\in A$ to
$\psi_a$. For $u\in A^*$, we denote by $\psi_u$ the image of $u$ by
the morphism $\psi$.  Thus, for three words $u,v,w$, we have
$\psi_{uv}(w)=\psi_u(\psi_v(w))$.

A \emph{palindrome}\index{palindrome word} is a word $w$ which is
equal to its reversal.  Given a word $w$, we denote by $w^{(+)}$ the
\emph{palindromic closure}\index{palindromic closure} of $w$. It is,
by definition, the shortest palindrome which has $w$ as a prefix.

The \emph{iterated palindromic closure}\index{iterated palindromic
  closure} of a word $w$ is the word $\Pal(w)$ defined recursively as
follows. One has $\Pal(1)=1$ and for $u\in A^*$ and $a\in A$, one has
$\Pal(ua)=(\Pal(u)a)^{(+)}$. Since $\Pal(u)$ is a proper prefix of
$\Pal(ua)$, it makes sense to define the iterated palindromic closure
of an infinite word $x$ as the infinite word which is the limit of the iterated palindromic closure
of the prefixes of $x$.

\emph{Justin's Formula}\index{Justin's Formula} is the following. For
every words $u$ and $v$, one has
\begin{displaymath}
  \Pal(uv)=\psi_u(\Pal(v))\Pal(u)\,.
\end{displaymath}
This formula extends to infinite words: if $u$ is a word and $v$ is an
infinite word, then
\begin{equation}\label{JustinInfini}
  \Pal(uv)=\psi_u(\Pal(v))\,.
\end{equation}
There is a precise combinatorial description of standard episturmian
words (see e.g.~\cite{JustinVuillon2000,GlenJustin2009}). 

\begin{theorem}
  An infinite word $s$ is a standard episturmian word if and only if
  there exists an infinite word $\Delta=a_0a_1\cdots$, where the $a_n$
  are letters, such that
  \begin{displaymath}
    s=\lim_{n\to\infty} u_n\,,
  \end{displaymath}
  where the sequence $(u_n)$ is defined by $u_n=\Pal(a_0a_1\cdots
  a_{n-1})$.  Moreover, the word $s$ is episturmian strict if and only
  if every letter appears infinitely often in~$\Delta$.
\end{theorem}

\noindent The infinite word $\Delta$ is called the \emph{directive
  word}\index{directive word} of the standard word $s$. The
description of the infinite word $s$ can be rephrased by the equation
\begin{displaymath}
  s=\Pal(\Delta)\,.
\end{displaymath}
As a particular case of Justin's
Formula, one has
\begin{equation}\label{EquationMagique}
  u_{n+1}=\psi_{a_0\cdots a_{n-1}}(a_n)u_n\,.
\end{equation}
The words $u_n$ are the only prefixes of $s$ which are palindromes. 

\begin{example}
  The Fibonacci word $x$ of Example~\ref{exampleFibonacci} is a
  standard episturmian word. It has the
  directive word $(ab)^\omega$, that is $x=\Pal((ab)^\omega)$
  \cite{GlenJustin2009}.  The Tribonacci word of
  Example~\ref{exampleTribonacci} has the directive word
  $\Delta=(abc)^\omega$ \cite{JustinVuillon2000}.  The corresponding
  sequence $(u_n)$ starts with $u_1=a$, $u_2=aba$,
  $u_3=abacaba$. Observe that $\psi_{ab}(c)=abac$, so that indeed
  $u_3=abac u_2$, as claimed in~\eqref{EquationMagique}.
\end{example}

\begin{example}\label{exampleNonStrict}
  Let $A=\{a,b,c\}$ and $\Delta=c(ab)^\omega$. Then, we have $u_1=c$,
  $u_2=cac$, $u_3=cacbcac$, $u_4=cacbcacacbcac$. By Justin's
  Formula~\ref{JustinInfini}, the limit is the word $x=\psi_c(y)$,
  where $y=\Pal((ab)^\omega)$ is the Fibonacci word on $\{a,b\}$. This
  means that $x$ is obtained from $y$ by inserting a letter $c$ before
  every letter of $y$. The word $x$ is not strict. Indeed, the letters
  $a$ and $b$ are not right-special and the letter $c$ is not strict
  right special since $cc$ is not a factor.
\end{example}

\begin{example} 
   Let $A=\{a,b,c\}$ and $\Delta=abc^\omega$. It is easily checked
   that $\Pal(\Delta)$ is the periodic word $(abac)^\omega$. The only
   right-special factors of this word are $1$ and $a$ (\cite{GlenJustin2009}).
\end{example}

\subsection{Probability distributions}\label{subsectionProbas}

Let $F\subset A^*$ be a prefix-closed set of words. For $w\in F$,
denote by $S(w)$ the set
$S(w)=\{a\in A\mid wa\in F\}$.  A \emph{right probability
  distribution} \index{right probability distribution}%
\index{probability distribution!right} on $F$ is a map
$\proba:F\rightarrow [0,1]$ such that
\begin{enumerate}
 \item[\upshape{(i)}] $\proba(1)=1$, 
 \item[\upshape{(ii)}] $\sum_{a\in S(w)}\proba(wa)=\proba(w)$, for any $w\in F$.
\end{enumerate}
For a right probability distribution $\proba$ on $F$ and a set $X\subset
F$, we denote $\proba(X)=\sum_{x\in X}\proba(x)$.  See
\cite{BerstelPerrinReutenauer2009} for the elementary properties of
right probability distributions. Note in particular that for any $u\in
F$ and $n\ge 0$, one has, as a consequence of condition (ii),
\begin{equation}
  \proba(uA^n\cap F)=\proba(u).\label{eqProbas}
\end{equation}
In particular, if $\proba$ is a right probability distribution on $F$,
then $\proba(F\cap A^n)=1$ for all $n\ge 0$.

The distribution is said to be \emph{positive}\index{positive
  probability distribution}\index{probability distribution!right} on
$F$ if $\proba(x)>0$ for any $x\in F$.

Symmetrically, for a suffix-closed set $F$, a \emph{left probability
  distribution}\index{left probability distribution}%
\index{probability distribution!left} is a map $\proba:F\rightarrow
[0,1]$ satisfying condition (i) above and

\begin{enumerate}
\item[(iii)] $\sum_{a\in P(w)}\proba(aw)=\proba(w)$, for any $w\in F$,
\end{enumerate}
with  $P(w)=\{a\in A\mid aw\in F\}$.

When $F$ is factorial, an \emph{invariant probability
  distribution}\index{invariant probability distribution}%
\index{probability distribution!invariant} is both a left and a right
probability distribution.

\begin{proposition}
  For any right essential set $F$ of words, there exists a positive
  right probability distribution $\proba$ on $F$.
\end{proposition}

\begin{proof}
  Consider the map $\proba:F\to[0,1]$ defined for $w=a_1a_2\cdots a_n$ by
  \begin{displaymath}
    \proba(w)=\frac{1}{d_0d_1\cdots d_{n-1}}
  \end{displaymath}
  where $d_i=\Card(S(a_1\cdots a_i))$ for $0\le i< n$.  Since $F$ is
  right essential, $d_i\ne0$ for $0\le i< n$. By convention,
  $\proba(1)=1$.

  Let us verify that $\proba$ is a right probability distribution on $F$.
  Indeed, let $w=a_1a_2\cdots a_n$.  The set $S(w)$ is nonempty. Let
  $a\in S(w)$, we have $\proba(wa)=1/d_0d_1\cdots d_n$. Since
  $\Card(S(w))=d_n$, we obtain that $\proba$ satisfies condition (ii) and
  thus it is a right probability distribution.  It is clearly
  positive.
\end{proof}

We will now turn to the existence of positive invariant probability
distributions. 

A \emph{topological dynamical system}%
\index{topological dynamical system}%
\index{dynamical system!topological} is a pair $(S,\sigma)$ of a
compact metric space $S$ and a continuous map $\sigma$ from $S$ into
$S$. Any shift space $S$ becomes a topological dynamical system when
it is equipped with the \emph{shift} map defined by
$\sigma(x_0x_1\cdots)=x_1x_2\cdots$.  Indeed, we consider $A^\N$ as a
metric space for the distance defined for $x=x_0x_1\cdots$ and
$y=y_0y_1\cdots$ by $d(x,y)=0$ if $x=y$ and $d(x,y)=2^{-n}$ where $n$
is the least integer such that $x_n\ne y_n$ otherwise.

A subset $T$ of a topological dynamical system $(S,\sigma)$ is said to
be \emph{stable under $\sigma$}%
\index{stable dynamical system}\index{dynamical system!stable} or
\emph{stable} for short if $\sigma(T)\subset T$. A stable subset is
also called (topologically) \emph{invariant}.

The following property is well-known (although usually stated for two
sided-infinite words, see for example Proposition~1.5.1
in~\cite{Lothaire2002}).

\begin{proposition}\label{propShiftSpaces}
  The shift spaces are the stable and closed subsets of
  $(A^\N,\sigma)$.
\end{proposition}

\begin{proof}
  It is clear that a shift space is both closed and
  stable. Conversely, let $S\subset A^\N$ be closed and stable under
  the shift. Let $X$ be the set of words which are not factors of
  words of $S$. Then $S=S_X$. Indeed, if $y\in S$, then $F(y)\cap
  X=\emptyset$ and thus $y\in S_X$. Conversely, let $y\in S_X$. Let
  $w_n$ be the prefix of length $n$ of $y$. Since $w_n\notin X$ there
  is an infinite word $y^{(n)}\in S$ such that $w_n\in
  F(y^{(n)})$. Since $S$ is stable under the shift, we may assume that
  $w_n$ is a prefix of $y^{(n)}$. The sequence $y^{(n)}$ converges to
  $y$. Since $S$ is closed, this forces $y\in S$.
\end{proof}

Let $S$ be a metric space.  The family of \emph{Borel
  subsets}\index{Borel subsets} of $S$ is the smallest family $\cal F$
of subsets of $S$ containing the open sets and closed under complement
and countable union. A function $\mu$ from $\cal F$ to $\R$ is said to
be \emph{countably additive}\index{countably additive!function} if
$\mu(\bigcup_{n\ge 0}X_n)=\sum_{n\ge 0}\mu(X_n)$ for any sequence $(X_n)$
of pairwise disjoint Borel subsets of $S$.  A \emph{Borel probability
  measure}\index{Borel probability measure} on $S$ is a function $\mu$
from $\cal F$ into $[0,1]$ which is
countably additive and such that $\mu(S)=1$.

Let $(S,\sigma)$ be a topological dynamical system.  A Borel
probability measure on $S$ is said to be \emph{invariant}%
\index{Borel probability measure!invariant}%
\index{invariant Borel probability measure} if
$\mu(\sigma^{-1}(B))=\mu(B)$ for any $B\in\cal F$. Note that since
$\sigma$ is continuous, $\sigma^{-1}(B)\in \cal F$ and thus
$\mu(\sigma^{-1}(B))$ is well defined.

The following result is from~\cite[Theorem 4.2]{Queffelec2010}.
\begin{theorem}\label{thKrylovBogolioubov}
  For any topological dynamical system, there exist invariant Borel
  probability measures.
\end{theorem}

A dynamical system $(S,\sigma)$ is said to be \emph{minimal}%
\index{minimal dynamical system}\index{dynamical system!minimal} if
the only closed stable subsets of $S$ are $S$ and $\emptyset$.  Note
that, by Proposition~\ref{propShiftSpaces}, this definition is
consistent with the definition of a minimal shift space.  A Borel
probability measure $\mu$ on $S$ is \emph{positive}%
\index{Borel probability measure!positive}%
\index{positive Borel probability measure} if $\mu(U)>0$ for every
nonempty open set $U\subset S$.

\begin{proposition}\label{propMinimalisPositive}
  Any invariant Borel probability measure on a minimal topological
  dynamical system is positive.
\end{proposition}
\begin{proof}
  Let $\mu$ be an invariant Borel probability measure on the topological
  dynamical system $(S,\sigma)$. Let $U\subset S$ be a nonempty open
  set.  Let $Y=\cup_{n\ge 0}\sigma^{-n}(U)$ and $Z=S\setminus
  Y$. Since $U$ is open and $\sigma$ is continuous, each
  $\sigma^{-n}(U)$ is open.  Thus $Y$ is open and $Z$ is closed. The
  set $Z$ is also stable. Indeed, if for $z\in Z$ we had
  $\sigma(z)\notin Z$, then there would be an integer  $n\ge 0$ such that
  $\sigma(z)\in\sigma^{-n}(U)$.  Thus $z\in\sigma^{-n-1}(U)\subset Y$,
  a contradiction. Thus $\sigma(Z)\subset Z$. Since $(S,\sigma)$ is
  minimal, this implies that $Z=\emptyset$ or $Z=S$. Since $U$ is
  nonempty, we have $Z=\emptyset$ and thus $Y=S$. Since $\mu$ is
  invariant, we have $\mu(\sigma^{-1}(U))=\mu(U)$ and thus
  $\mu(\sigma^{-n}(U))=\mu(U)$ for all $n\ge 0$. Hence we cannot have
  $\mu(U)=0$ since it would imply $\mu(S)\le \sum_{n\ge
    0}\mu(\sigma^{-n}(U))=0$, a contradiction since $\mu(S)=1$.
\end{proof}

\begin{corollary}
  For any recurrent set $F$ there exists an invariant probability
  distribution on $F$. When $F$ is uniformly recurrent, such a
  distribution is positive.
\end{corollary}

\begin{proof}
  Let $F$ be a recurrent set. By Proposition~\ref{propExists} there is
  a recurrent infinite word $x$ such that $F(x)=F$, and if $F$ is
  uniformly recurrent, then $x$ is uniformly recurrent.

  By Theorem~\ref{thKrylovBogolioubov} there is an invariant Borel
  probability measure $\mu$ on $S=S(x)$.

  Let $\proba$ be the map from $F$ to $[0,1]$ defined by
  $\proba(w)=\mu(wA^\N\cap S)$.  Let us verify that $\proba$ is an invariant
  probability distribution.  Indeed, one has $\proba(1)=\mu(S)=1$. Next,
  for $w\in F$
  \begin{displaymath}
    \sum_{a\in S(w)}\proba(wa)=\sum_{a\in S(w)}\mu(waA^\N\cap
    S)=\mu(wA^\N\cap S)=\proba(w).
  \end{displaymath}
  In the same way
  \begin{displaymath}
    \sum_{a\in P(w)}\proba(aw)=\sum_{a\in P(w)}\mu(awA^\N\cap
    S)=\mu(\sigma^{-1}(wA^\N\cap S))=\mu(wA^\N\cap S)=\proba(w).
  \end{displaymath}
  If $x$ is uniformly recurrent, by Proposition~\ref{propMinimal}, the
  shift space $S=S(x)$ is minimal.  By
  Proposition~\ref{propMinimalisPositive}, the measure $\mu$ is
  positive. Since $wA^\N\cap S$ is a nonempty open set for any $w\in
  F$, we have $\proba(w)=\mu(wA^\N\cap S)>0$ and thus $\proba$ is positive.
\end{proof}

In some cases, there exists a unique invariant probability
distribution on the set $F$. A morphism $f:A^*\rightarrow A^*$ is
\emph{primitive}\index{primitive morphism}\index{morphism!primitive}
if there exists an integer $k$ such that, for all $a,b\in A$, the
letter $b$ appears in $f^k(a)$.  If $f$ is a primitive morphism and if
$f(a)$ starts with the letter $a$ for some $a\in A$, then
$x=f^\omega(a)$ is a fix-point of $f$ and there is a unique invariant
probability distribution $\proba_F$ on the set $F(x)$ (\cite[Theorem
5.6]{Queffelec2010}). Moreover, this distribution is positive.  We
illustrate this result by the following examples.

\begin{example}\label{exProba}
  Let $F$ be the Fibonacci set (see
  Example~\ref{exampleFibonacci}). Since the morphism $f$ defined by
  $f(a)=ab$ and $f(b)=a$ is primitive, there is a unique invariant
  probability distribution on $F$. Its values on the words of length
  at most $4$ are shown on Figure~\ref{figProbaFibo} with
  $\lambda=(\sqrt{5}-1)/2$.
\begin{figure}[hbt]
\centering
\gasset{Nadjust=wh}
\begin{picture}(100,50)
\node(1)(0,25){$1$}
\node(a)(15,35){$\lambda$}\node(b)(15,15){$1-\lambda$}
\node(aa)(30,45){$2\lambda-1$}\node(ab)(30,30){$1-\lambda$}
\node(ba)(30,10){$1-\lambda$}
\node(aab)(50,45){$2\lambda-1$}
\node(aba)(50,30){$1-\lambda$}
\node(baa)(50,15){$2\lambda-1$}\node(bab)(50,5){$2-3\lambda$}
\node(aaba)(70,45){$2\lambda-1$}
\node(abaa)(70,35){$2\lambda-1$}\node(abab)(70,25){$2-3\lambda$}
\node(baab)(70,15){$2\lambda-1$}
\node(baba)(70,5){$2-3\lambda$}
\drawedge(1,a){$a$}\drawedge[ELside=r](1,b){$b$}
\drawedge(a,aa){$a$}\drawedge[ELside=r](a,ab){$b$}
\drawedge(aa,aab){$b$}\drawedge(ab,aba){$a$}
\drawedge(aab,aaba){$a$}
\drawedge(aba,abaa){$a$}\drawedge[ELside=r](aba,abab){$b$}
\drawedge[ELside=r](b,ba){$a$}
\drawedge(ba,baa){$a$}\drawedge[ELside=r](ba,bab){$b$}
\drawedge(baa,baab){$b$}\drawedge(bab,baba){$a$}
\end{picture}
\caption{The invariant probability distribution on the 
  Fibonacci set.}\label{figProbaFibo}
\end{figure}
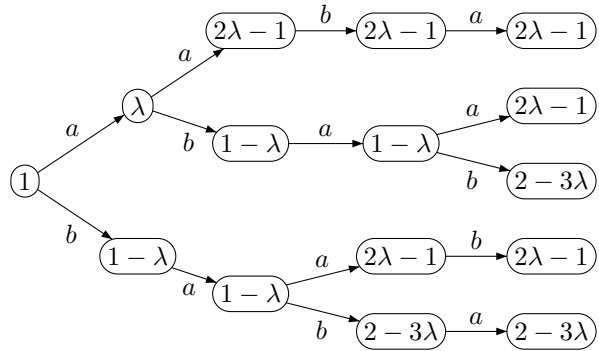
The values of $\proba_F$ can be obtained as follows
(see~\cite{Queffelec2010}). The vector
$v=\begin{bmatrix}\proba_F(a)&\proba_F(b)\end{bmatrix}$ is an eigenvector
for the eigenvalue $1/\lambda$ of the $A\times A$-matrix $M$ defined
by $M_{ab}=|f(a)|_b$. Here, we have
\begin{displaymath}
  M=\begin{bmatrix}1&1\\1&0\end{bmatrix}
\end{displaymath}
This implies $v=\begin{bmatrix}\lambda&1-\lambda\end{bmatrix}$ .  The
other values can be computed using conditions (ii) and (iii) of the
definition of an invariant probability distribution.
\end{example}

\begin{example}\label{exampleDistributionMorse}
  Let $F$ be the Thue--Morse set (see
  Example~\ref{exampleMorse}). Since the Thue-Morse morphism is
  primitive, there is a unique invariant probability distribution on
  $F$. Its values on the words of length at most $4$ are shown on
  Figure~\ref{figProbaMorse}.
\begin{figure}[hbt]
\centering
\gasset{Nadjust=wh}
\begin{picture}(100,63)
\node(1)(0,32){$1$}
\node(a)(15,50){$1/2$}\node(b)(15,15){$1/2$}
\node(aa)(30,60){$1/6$}\node(ab)(30,40){$1/3$}
\node(ba)(30,23){$1/3$}\node(bb)(30,4){$1/6$}
\node(aab)(45,60){$1/6$}
\node(aba)(45,45){$1/6$}\node(abb)(45,35){$1/6$}
\node(baa)(45,28){$1/6$}\node(bab)(45,18){$1/6$}
\node(bba)(45,4){$1/6$}
\node(aaba)(60,63){$1/12$}\node(aabb)(60,56){$1/12$}
\node(abaa)(60,49){$1/12$}\node(abab)(60,42){$1/12$}
\node(abba)(60,35){$1/6$}
\node(baab)(60,28){$1/6$}
\node(baba)(60,21){$1/12$}\node(babb)(60,14){$1/12$}
\node(bbaa)(60,7){$1/12$}\node(bbab)(60,0){$1/12$}
\drawedge(1,a){$a$}\drawedge(1,b){$b$}\drawedge(a,aa){$a$}\drawedge(a,ab){$b$}
\drawedge(aa,aab){$b$}
\drawedge(ab,aba){$a$}\drawedge(ab,abb){$b$}
\drawedge(aab,aaba){$a$}\drawedge(aab,aabb){$b$}
\drawedge(aba,abaa){$a$}\drawedge(aba,abab){$b$}
\drawedge(abb,abba){$a$}
\drawedge(b,ba){$a$}\drawedge(b,bb){$b$}
\drawedge(ba,baa){$a$}\drawedge(ba,bab){$b$}
\drawedge(baa,baab){$b$}
\drawedge(bab,baba){$a$}\drawedge(bab,babb){$b$}
\drawedge(b,bb){$b$}
\drawedge(bb,bba){$a$}
\drawedge(bba,bbaa){$a$}\drawedge(bba,bbab){$b$}
\end{picture}
\caption{The invariant probability distribution on the 
  Thue--Morse set.}\label{figProbaMorse}
\end{figure}
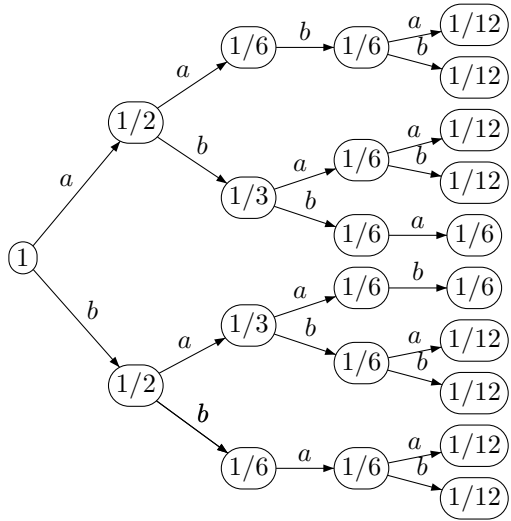
\end{example}

\section{Prefix codes in  factorial sets}\label{sectionPrefixCodes}

In this section, we study prefix codes in a factorial set.  We will
see that most properties known in the usual case are also true in this
more general situation. Some of them are even true in the more general
case of a prefix-closed set instead of a factorial set. In particular,
this holds for the link between prefix codes and probability
distributions (Proposition~\ref{propMaxPrefixCode}).

Recall that a set $X\subset A^+$ of nonempty words over an alphabet
$A$ is a \emph{code}\index{code} if the relation
\begin{displaymath}
  x_1\cdots x_n=y_1\cdots y_m
\end{displaymath}
with $n,m\ge1$ and $x_1,\ldots,x_n,y_1,\ldots,y_m\in X$ implies $n=m$
and $x_i=y_i$ for $i=1,\ldots, n$. For the general theory of codes,
see~\cite{BerstelPerrinReutenauer2009}.

\subsection{Prefix codes}
The \emph{prefix order}\index{prefix order}\index{order!prefix} is defined, for $u,v\in
A^*$, by $u\le v$ if $u$ is a prefix of $v$.  Two words $u,v$ are
\emph{prefix-comparable}\index{prefix-comparable
  words}\index{word!prefix-comparable} if one is a prefix of the
other. Thus $u$ and $v$ are prefix-comparable if and only if there are
words $x,y$ such that $ux=vy$ or, equivalently, if and only if
$uA^*\cap vA^*\ne\emptyset$. The \emph{suffix order}%
\index{order!suffix}, and the notion of suffix-comparable words, are
defined symmetrically.

A set $X\subset A^+$ of nonempty words is a \emph{prefix
  code}\index{prefix code}\index{code!prefix} if any two distinct
elements of $X$ are incomparable for the prefix order. A prefix code
is a code.

The dual notion of a \emph{suffix}\index{suffix
  code}\index{code!suffix} code is defined symmetrically with respect
to the suffix order.

The submonoid $M$ generated by a prefix code satisfies the following
property: if $u,uv\in M$ then $v\in M$. Such a submonoid of $A^*$
is said to be \emph{right unitary}\index{right unitary submonoid}.
One can show that conversely, any right unitary submonoid of $A^*$
is generated by a prefix code
(see~\cite{BerstelPerrinReutenauer2009}).
The symmetric notion of a \emph{left unitary} submonoid is defined
by the condition $v,uv\in M$ implies $u\in M$.

We denote by $\u(X)$ the
\emph{characteristic series}\index{characteristic series} of a set
$X\subset A^*$. By definition, for any $x\in A^*$,
\begin{displaymath}
(\u(X),x)=\begin{cases}1&\text{if $x\in X$}\\0&\text{otherwise}\end{cases}
\end{displaymath}

The following is 
Proposition 3.1.6 in~\cite{BerstelPerrinReutenauer2009}.

\begin{proposition}\label{propositionCharPrefix}
Let $X$ be a prefix code and let $U=A^*\setminus XA^*$. Then
\begin{equation}
\u(A^*)=\underline{X}^*\u(U)\quad
\text{and}\quad\underline{X}-1=\underline{U}(\underline{A}-1).
\label{eqX-1}
\end{equation}
\end{proposition}

\subsection{Automata}\label{sectionAutomata}
We  recall the basic results on
deterministic
automata and prefix codes
(see~\cite{BerstelPerrinReutenauer2009} 
for a more detailed exposition). 

We denote $\A=(Q,i,T)$ a deterministic automaton with $Q$ as set of
states,
$i\in Q$ as initial state and $T\subset Q$ as set of terminal states.
For $p\in Q$ and $w\in A^*$, we denote $p\cdot w=q$ if
 there is a path labeled $w$ from $p$ to the state $q$ and $p\cdot
 w=\emptyset$
otherwise.

The set \emph{recognized}\index{recognized by an automaton} by the
automaton is the set of words $w\in A^*$ such that $i\cdot w\in T$. A
set of words is \emph{rational}\index{rational set} if is recognized
by a finite automaton.

All automata considered in this paper are deterministic and we 
call them simply automata.

The automaton $\A$ is \emph{trim}\index{automaton!trim}\index{trim
  automaton} if for any $q\in Q$, there is a path from $i$ to $q$ and
a path from $q$ to some $t\in T$.

An automaton is called
\emph{simple}\index{automaton!simple}\index{simple automaton} if it is
trim and if it has a unique terminal state which coincides with the
initial state.

An automaton $\A=(Q,i,T)$ is \emph{complete}%
\index{automaton!complete}\index{complete automaton}
if for any state $p\in Q$
and
any letter $a\in A$, one has $p\cdot a\ne\emptyset$.

For a set $X\subset A^*$, we denote by $\A(X)$ the \emph{minimal
  automaton}\index{automaton!minimal}\index{minimal automaton}
 of $X$. The states of $\A(X)$ are the nonempty sets
$u^{-1}X=\{v\in A^*\mid uv\in X\}$ for $u\in A^*$.  The initial state
is the set $X$ and the terminal states are the sets $u^{-1}X$ for
$u\in X$.


Let $X\subset A^*$ be a prefix code. Then there is a simple automaton
$\A=(Q,1,1)$ that recognizes $X^*$. Moreover, the minimal
automaton of $X^*$ is simple.

Let $X$ be a prefix code and let $P$ be the set of proper prefixes of
$X$. The \emph{literal automaton}\index{literal
  automaton}\index{automaton!literal} of $X^*$ is the simple automaton
$\A=(P,1,1)$ with transitions defined for $p\in P$ and $a\in A$ by
\begin{displaymath}
  p\cdot a=
  \begin{cases}
    pa&\text{if $pa\in P$}\,,\\
    1&\text{if $pa\in X$}\,,\\
    \emptyset&\text{otherwise}.
  \end{cases}
\end{displaymath}
One verifies that this automaton recognizes $X^*$.

Let $\A=(Q,i,T)$ be an automaton. For $w\in A^*$, we denote
$\varphi_\A(w)$ the partial map from $Q$ to $Q$ defined by
$p\varphi_\A(w)=q$ if $p\cdot w=q$. The \emph{transition
  monoid}\index{transition monoid}\index{monoid!transition} of
$\A$ is the monoid of partial maps from $Q$ to $Q$ of the form
$\varphi_\A(w)$ for $w\in A^*$.

\subsection{Maximal prefix codes}

Let $F$ be a subset of $A^*$.  A set $X\subset A^*$ is \emph{right
  dense}\index{right dense set} in $F\subset A^*$, or right $F$-dense,
if any $u\in F$ is a prefix of $X$.

A set $X\subset F$ is \emph{right complete}\index{right complete set}
in $F$, or right $F$-complete, if $X^*$ is right dense in $F$, that is
if every word in $F$ is a prefix of $X^*$.

A prefix code $X\subset F$ is \emph{maximal}\index{F-maximal
  code@$F$-maximal code} in $F$, or $F$-maximal, if it is not properly
contained in any other prefix code $Y\subset F$. The notion of an
$F$-maximal suffix code is symmetrical.

The following propositions are  extensions of Propositions 3.3.1 and 3.3.2, and of
Theorem 3.3.5 in~\cite{BerstelPerrinReutenauer2009}.

\begin{proposition}\label{propositionPrefixGlobal}
  Let $F$ be a subset of $A^*$. For any prefix code $X\subset F$, the
  following conditions are equivalent.
  \begin{enumerate}
  \item[\upshape{(i)}] Every element of $F$ is prefix-comparable with
    some element of $X$,
  \item[\upshape{(ii)}] $X$ is an $F$-maximal prefix code.
  \end{enumerate}
\end{proposition}

\begin{proof}
(i) implies (ii). Any word $u\in F$ is prefix-comparable with some
word of $X$. This implies that if $u\notin X$, then $X\cup u$ is no
longer a prefix code. Thus $X$ is an $F$-maximal prefix code.

(ii) implies (i). Assume that $u\in F$ is not prefix-comparable to any
word in $X$. Then $X\cup u$ is prefix, and $X$ is not an $F$-maximal
prefix code.
\end{proof}

\begin{proposition}\label{propositionGlobal}
  Let $F$ be a factorial subset of $A^*$. For any set $X\subset F$ of
  nonempty words, the following conditions are equivalent.
  \begin{enumerate}
  \item[\upshape{(i)}] Every element of $F$ is prefix-comparable with
    some element of $X$,
  \item[\upshape{(ii)}] $XA^*$ is right $F$-dense,
  \item[\upshape{(iii)}] $X$ is right $F$-complete.
  \end{enumerate}
\end{proposition}

\begin{proof}
  (i) implies (ii). Let $u\in F$. Let
  $x\in X$ be prefix-comparable with $u$. Then there exist $v,w$ such
  that $uv=xw$. Thus $XA^*$ is right $F$-dense.

  (ii) implies (iii).  Consider a word $u\in F$. Let us show that $u$
  is a prefix of $X^*$.  Since $XA^*$ is right $F$-dense,
  one has $uw=xw'$ for some word $x\in X$ and $w,w'\in A^*$.  If $u$
  is a prefix of $X$, there is nothing to prove.  Otherwise,
  $x$ is a proper prefix of $u$. Thus $u=xu'$ for some $x\in X$ and
  $u'\in A^*$. Since $u$ is in $F$ and since $F$ is factorial, we have
  $u'\in F$. Since $x\ne 1$, we have $|u'|<|u|$. Arguing by induction,
  the word $u'$ is a prefix of $X^*$. Thus $u$ is a prefix
  of $X^*$.

(iii) implies (i).  Let $u\in F$. Then $u$ is a prefix of
$X^*$, and consequently $u$ is prefix-comparable with a word in $X$.

\end{proof}

The propositions have a dual formulation, replacing prefix by suffix,
and right by left.

\begin{example}\label{exaba}
  The set $X=\{a,ba\}$ is a maximal prefix code in the Fibonacci set
  $F$ since $XA^*$ is right $F$-dense.
\end{example}

The following is a generalization of Propositions 3.7.1 and
 3.7.2 in~\cite{BerstelPerrinReutenauer2009}. 

\begin{proposition}\label{propMaxPrefixCode}
  Let $F$ be a  right essential set.  Let $\proba$ be a positive right
  probability distribution on $F$.  Any prefix code $X\subset F$
  satisfies $\proba(X)\le 1$.  If $X$ is finite, it is $F$-maximal if and
  only if $\proba(X)=1$.
\end{proposition}

\begin{proof}
  Assume first that $X$ is finite.  Let $n$ be the maximal length of
  the words in $X$. We have
  \begin{equation}
    \bigcup_{x\in X}xA^{n-|x|}\cap F\subset A^n\cap F \label{eqX}
  \end{equation}
  and the terms of the union are pairwise disjoint.  Thus, using
  Equation~\eqref{eqProbas}
  \begin{equation}
    \proba(X)=\sum_{x\in X}\proba(xA^{n-|x|}\cap F)\le\proba(A^n\cap F)=1\,.\label{eqPi}
  \end{equation} 
  If $X$ is maximal in $F$, any word in $F\cap A^n$ has a prefix in
  $X$. Thus we have equality in~\eqref{eqX} and thus also
  in~\eqref{eqPi}.  This shows that $\proba(X)=1$. The converse is clear
  since $\proba$ is positive on $F$.

  If $X$ is infinite, then $\proba(Y)\le 1$ for any finite subset $Y$ of
  $X$. Thus $\proba(X)\le 1$.
\end{proof}

The statement has a dual for a suffix code included in a factorial
set $F$ with a positive left probability distribution on $F$.

\begin{example}\label{exabapi}
  Let $F$ be the Fibonacci set.  The set
  $X=\{a,ba\}$ is a maximal prefix code (Example~\ref{exaba}). One has
  $\proba_F(X)=1$ where $\proba_F$ is defined in Example~\ref{exProba}.
\end{example}

We will use the following result in the proof of
Proposition~\ref{propositionInternalTransformation}. 

\begin{proposition}\label{propositionFinitePrefixCodes}
  Let $F$ be a right essential subset of $A^*$, and let $G\subset F$
  be a right essential subset of $F$. For any finite $F$-maximal
  prefix code $X\subset F$, the set $X\cap G$ is a finite $G$-maximal
  prefix code.
\end{proposition}

\begin{proof}
  Set $Y=X\cap G$. The set $Y$ is clearly a finite prefix code.  We
  show that every $u\in G$ is prefix-comparable with some word in $Y$.
  This will imply that $Y$ is $G$-maximal by
  Proposition~\ref{propositionPrefixGlobal}.  Let $u\in G$.  Since $G$ is
  right essential, there are arbitrary long words $w$ such that $uw\in
  G$. Choose the length of $uw$ larger than the maximal length of the
  words of $X$. Since $X$ is an $F$-maximal prefix code, $uw$ has a
  prefix $x$ in $X$. This prefix $x$ is in $Y$ since $uw\in G$. Thus
  $u$ is prefix-comparable to $x\in Y$.
\end{proof}

The following example shows that Proposition~\ref{propositionFinitePrefixCodes}
is false for infinite prefix codes.

\begin{example}
  Let $F\subset A^*$ be a right essential set with $F\ne A^*$, and let $x$
  be a word which is not in $F$.  Let $X=A^*x \setminus A^*xA^+$ be the
  prefix code of words in $A^*$ ending with $x$ and having no other
  occurrence of $x$. $X$ is a maximal prefix code, and $X \cap
  F=\emptyset$ is not $F$-maximal.
\end{example}

We will use later the following result on transformations of prefix
codes.  It is adapted from Proposition~3.4.9
in~\cite{BerstelPerrinReutenauer2009}.

\begin{proposition}\label{propositionOperation1}
  Let $F$ be a factorial set and let $X\subset F$ be an $F$-maximal
  prefix code.  Let $w$ be a nonempty prefix of $X$ and set
  $D=w^{-1}X$. The set $Y=(X\setminus wD)\cup w$ is an $F$-maximal
  prefix code.
\end{proposition}

\begin{proof}
  It is clear that $Y$ is a prefix code. To show that it is
  $F$-maximal, we apply Proposition~\ref{propositionPrefixGlobal} and prove
  that every word $u\in F$ is prefix-comparable with a word of $Y$. So
  consider a word $u\in F$. Since $X$ is $F$-maximal, $u$ is
  prefix-comparable with a word of $X$. Thus $u$ is prefix-comparable
  with a word of $X\setminus wD$ or it is prefix-comparable with a
  word of $wD$. In the second case, either $u$ is a prefix of a word
  $wd$ with $d\in D$ or $u$ has $wd$ as a prefix. Consequently, $u$ is
  prefix-comparable with $w$. This proves that $u$ is
  prefix-comparable with a word of $Y$.
\end{proof}

Proposition~\ref{propositionOperation1} has a dual formulation for suffix codes.

\subsection{Average length}
Let $F$ be a right essential set and let $\proba$ be a right probability
distribution on $F$. Let $X\subset F$ be a prefix code such that
$\proba(X)=1$.  The \emph{average length}\index{average length} of $X$
with respect to $\proba$ is the sum
\begin{displaymath}
  \lambda(X)=\sum_{x\in X}|x|\proba(x)\,.
\end{displaymath}

\begin{proposition}\label{propositionP}
  Let $F$ be a right essential set and let $\proba$ be a positive right
  probability distribution on $F$. Let $X\subset F$ be a finite
  $F$-maximal prefix code and let $P$ be the set of proper prefixes of
  $X$.  Then $\proba(X)=1$ and $\lambda(X)=\proba(P)$.
\end{proposition}

\begin{proof}
  We already know that $\proba(X)=1$ by
  Proposition~\ref{propMaxPrefixCode}.  Let us show that for any $p\in
  P$,
  \begin{equation}
    \proba(p)=\sum_{x\in pA^+\cap X}\proba(x)\,. \label{eqProbas2}
  \end{equation}
  Let indeed $n$ be an integer larger than the lengths of the words of
  $X$.  Then by Equation~\eqref{eqProbas}, $\proba(p)=\proba(pA^n\cap
  F)$. Since $X$ is an $F$-maximal prefix code, each word of $pA^n\cap
  F$ has a prefix in $X$, and conversely, each word in $X$ which has
  $p$ as a prefix is itself a prefix of $pA^n\cap F$. Thus
  \begin{displaymath}
    pA^n\cap F=\bigcup_{x\in pA^+\cap X}xA^{n+|p|-|x|}\cap F\,.
  \end{displaymath} 
  Since $\proba(xA^{n+|p|-|x|}\cap F)=\proba(x)$, this proves
  Equation~\eqref{eqProbas2}.

By Equation~\eqref{eqProbas2}, one gets 
  \begin{displaymath}
    \sum_{p\in P}\proba(p)=\sum_{p\in P}\sum_{x\in pA^+\cap X}\proba(x)=
    \sum_{x\in X}\sum_{p\in P:x\in pA^+}\proba(x)=\sum_{x\in X}|x|\proba(x)\,.
  \end{displaymath}
  Thus
  \begin{displaymath}
    \proba(P)=\sum_{p\in P}\proba(p)=\sum_{x\in X}|x|\proba(x)=\lambda(X)\,.
  \end{displaymath}
\end{proof}
A dual statement of Proposition~\ref{propositionP} holds for a suffix
code and its set of proper suffixes, for a positive left probability
distribution.

\begin{example}
  Let $F$ be the Fibonacci set and let
  $X=\{a,ba\}$.  We have already seen in Example~\ref{exabapi} that
  $X$ is an $F$-maximal prefix code and that $\proba_F(X)=1$ where
  $\proba_F$ is the unique invariant probability distribution on $F$
  defined in  Example~\ref{exProba}. We
  have $\lambda(X)=\lambda+2(1-\lambda)=2-\lambda$. On the other hand
  the set of proper prefixes of $X$ is $P=\{1,b\}$ and thus
  $\proba_F(P)=1+(1-\lambda)=2-\lambda$.
\end{example}

\section{Bifix codes in recurrent sets}\label{sectionBifixCodes}

In this section, we study bifix codes contained in a recurrent set. Since
$A^*$ itself is a recurrent set, it is a generalization of the usual
situation. We will see that all results on maximal bifix codes
can be generalized in this way. In particular, the notions of
degree, of kernel and of derived code can be defined in this
more general framework.

\subsection{Parses}


Recall that a set $X$ of nonempty words is a \emph{bifix
  code}\index{code!bifix}\index{bifix code} if any two distinct
elements of $X$ are incomparable for the prefix order and for the
suffix order.

A \emph{parse}\index{parse of a word} of a word $w$ with respect to a
set $X$ is a triple $(v,x,u)$ such that $w=vxu$ with $v\in
A^*\setminus A^*X$, $x\in X^*$ and $u\in A^*\setminus XA^*$.

\begin{proposition}\label{propParses}
  Let $F$ be a factorial set and let $X\subset F$ be a set.  For any
  factorization $w=uv$ of $w\in F$, there is a parse $(s,yz,p)$ of $w$
  with $y,z\in X^*$, $sy=u$ and $v=zp$.
\end{proposition}

\begin{proof}
  Since $v\in F$, there exist, by
  Proposition~\ref{propositionCharPrefix}, words $z\in X^*$ and $p\in
  A^*\setminus XA^*$ such that $v=zp$. Symmetrically, there exist
  $y\in X^*$ and $s\in A^*\setminus A^*X$ such that $u=sy$. Then
  $(s,yz,p)$ is a parse of $w$ which satisfies the conditions of the
  statement.
\end{proof}


The number of parses of a word $w$ with respect to $X$ is denoted by
$\pars_X(w)$.  The function $\pars_X:A^*\to\N$ is the \emph{parse
  enumerator}\index{parse enumerator} with respect to $X$. 

The \emph{indicator}\index{indicator} of a set $X$
is the series $L_X$ defined for $w\in A^*$ by $(L_X,w)=\pars_X(w)$.

\begin{example}
  Let $X=\emptyset$. Then $\pars_X(w)=|w|+1$.
\end{example}

The following is a reformulation of Proposition 6.1.6
in~\cite{BerstelPerrinReutenauer2009}.

\begin{proposition}\label{propInterpretations}
  Let $F$ be a factorial set and let $X\subset F$ be a prefix code.
  For every word $w\in F$, the number $\pars_X(w)$ is equal to the
  number of prefixes of $w$ which have no suffix in $X$.
\end{proposition}

\begin{proof}
For every prefix $v$ of $w$ which is in $A^*\setminus A^*X$, there is
a unique parse  of $w$ of the form $(v,x,u)$. Since any parse is
obtained
in this way, the statement is proved.
\end{proof}

Proposition~\ref{propInterpretations} has a dual statement for suffix
codes.

Note that, as a consequence of Proposition~\ref{propInterpretations},
we have for two prefix codes $X,Y$, and for all words $w$,
\begin{equation}
 X\subset Y \Rightarrow \pars_Y(w)\le \pars_X(w). \label{eqLsub}
\end{equation}
Indeed, a word without suffix in $Y$ is also a word without suffix in $X$.

\begin{proposition}\label{propositionEqLS}
  Let $X$ be a prefix code  and let $V=A^*\setminus A^*X$. Then
  \begin{equation}
    \underline{V}=L_X(1-\underline{A})\,. \label{eqLS}
  \end{equation}
  If  $X$ is bifix, one has
  \begin{equation}
    1-\u(X)=(1-\underline{A})L_X(1-\underline{A})\,. \label{eqLB}
  \end{equation}
\end{proposition}

\begin{proof} Set $L=L_X$.
  Let $U=A^*\setminus XA^*$. By definition of the indicator, we
  have $L=\u(V)\ \underline{X}^*\u(U)$. Since $X$ is prefix, we have by
  Proposition~\ref{propositionCharPrefix}, the equality
  $\u(A^*)=\underline{X}^*\u(U)$. Thus we obtain $L=\u(V)\u(A^*)$ (note that this
  is actually equivalent to
  Proposition~\ref{propInterpretations}). Multiplying both sides on
  the right by $(1-\u(A))$, we obtain Equation~\eqref{eqLS}.

  If $X$ is suffix, we have by the dual of
  Proposition~\ref{propositionCharPrefix}, the equality
  $1-\u(X)=(1-\u(A))\u(V)$. This gives Equation~\eqref{eqLB} by
  multiplying both sides of Equation~\eqref{eqLS} on the left by
  $1-\u(A)$.
\end{proof}

The following is  Proposition 6.1.11
in~\cite{BerstelPerrinReutenauer2009}. 

\begin{proposition}\label{proposition6111}
  A function $\pars:A^*\to\N$ is the
  parse enumerator of
  some bifix code  if and only if it satisfies the
  following conditions.
  \begin{enumerate}
  \item[\upshape{(i)}] For any $a\in A$ and $w\in A^*$ 
    \begin{equation}
      0\le\pars(aw)-\pars(w)\le 1\,.\label{eq6111a}
    \end{equation}
  \item[\upshape{(ii)}] For any $w\in A^*$ and $a\in A$ 
    \begin{equation}
      0\le\pars(wa)-\pars(w)\le 1\,.\label{eq6111b}
    \end{equation}
  \item[\upshape{(iii)}] For any $a,b\in A$ and $w\in A^*$
    \begin{equation}
      \pars(aw)+\pars(wb)\ge \pars(w)+\pars(awb)\,.\label{eq6111c}
    \end{equation}
  \item[\upshape{(iv)}] $\pars(1)=1$\,.
  \end{enumerate}
\end{proposition}

The following is a reformulation of Proposition 6.1.12
in
\cite{BerstelPerrinReutenauer2009}.
\begin{proposition}\label{propositioneqL}
Let  $X$ be a prefix code. For any $u\in A^*$ and
$a\in A$, one has
\begin{equation}
\pars_X(ua)=\begin{cases}\pars_X(u)&\text{if } ua\in A^*X \\
\pars_X(u)+1&\text{otherwise}
\end{cases} \label{eqL}
\end{equation}
\end{proposition}
\begin{proof}
This follows directly from Proposition~\ref{propInterpretations}.
\end{proof}
Proposition~\ref{propositioneqL} has a dual for suffix codes
expressing
$\pars_X(au)$ in terms of $\pars_X(u)$.

Recall also that by Proposition 6.1.8 in
\cite{BerstelPerrinReutenauer2009}, for a bifix code $X$ and for
all $u,v,w\in F$ such that $uvw\in F$,
one has
\begin{equation}
  \pars_X(v)\le \pars_X(uvw).  \label{eqL1}
\end{equation}
Moreover, if $uvw\in X$ and $u, w \in A^+$ then the inequality is
strict, that is,
\begin{equation}
  \pars_X(v) < \pars_X(uvw).  \label{eqLi}
\end{equation}
\subsection{Maximal bifix codes}

Let $F$ be set of words. A set $X\subset F$ is said to be
\emph{thin}\index{thin set} in $F$, or $F$-thin, if there exists a
word of $F$ which is not a factor of a word in $X$.

The following example shows that there exist a uniformly recurrent  set $F$,
and a bifix code $X\subset F$ which is not $F$-thin.

\begin{example}\label{exampleDenseBifixCode}

  Let $F$ be the Thue--Morse set, which is the set of factors of a
  fix-point of the substitution $f$ defined by $f(a)=ab$, $f(b)=ba$
  (see Example~\ref{exampleMorse}).  Set $x_n=f^n(a)$ for $n\ge
  1$. Note that $x_{n+1}=x_n\bar{x}_n$ where $u\to\bar{u}$ is the
  substitution defined by $\bar{a}=b$ and $\bar{b}=a$.  Note also that
  $u\in F$ if and only if $\bar{u}\in F$.  Consider the set
  $X=\{x_nx_n\mid n\ge 1\}$.  We have $X\subset F$. Indeed, for $n\ge
  1$, $x_{n+2}=x_{n+1}\bar{x}_{n+1}=x_{n}\bar{x}_{n}\bar{x}_{n}x_n$
  implies that $\bar{x}_{n}\bar{x}_{n}\in F$ and thus $x_nx_n\in F$.
  Next $X$ is a bifix code. Indeed, for $n<m$, $x_m$ begins with
  $x_n\bar{x}_n$, and thus cannot have $x_nx_n$
  as a prefix.  Similarly, since $x_{m}$ ends with $\bar{x}_{n}x_{n}$
  or with $x_{n}\bar{x}_{n}$, it cannot have $x_ nx_ n$ as a suffix.
  Finally any element of $F$ is a factor of a word in $X$. Indeed, any
  element $u$ of $F$ is a factor of some $x_n$, and thus of $x_nx_n\in
  X$.

  A simpler proof uses Theorem~\ref{theoremCompletion} proved
  later.
\end{example}

An \emph{internal factor}\index{internal
  factor}\index{factor!internal}\index{word!internal factor} of a word
$x$ is a word $v$ such that $x=uvw$ with $u,w$ nonempty.  Let
$F\subset A^*$ be a factorial set and let $X\subset F$ be a
set. Denote by
\begin{displaymath}
I(X)=\{w\in A^*\mid A^+wA^+\cap X\ne\emptyset\}
\end{displaymath}
 the set of internal factors of words in
$X$.\footnote{The set $I(X)$ is denoted by $H(X)$ in~\cite{BerstelPerrinReutenauer2009}.}

When $F$ is right essential and left essential, then $X$ is $F$-thin
if and only if $\nobreak{F\setminus I(X)\ne\emptyset}$. Indeed, the
condition is necessary.  Conversely, if $w$ is in $\nobreak{F\setminus
  I(X)}$, let $a,b\in A$ be such that $awb\in F$. Since $awb$ cannot
be a factor of a word in $X$, it follows that $X$ is $F$-thin.

We say that a bifix code $X\subset F$ is
\emph{maximal}\index{F-maximal bifix code@$F$-maximal bifix code} in
$F$, or $F$-maximal, if it is not properly contained in any other
bifix code $Y\subset F$.

The following is a generalisation of Proposition 6.2.1 
 in~\cite{BerstelPerrinReutenauer2009}.

\begin{theorem}\label{theoremEquivMax}
  Let $F$ be a recurrent set and let $X\subset F$ be an $F$-thin
  set. The following conditions are equivalent.
  \begin{enumerate}
  \item[\upshape{(i)}] $X$ is an $F$-maximal bifix code.
  \item[\upshape{(ii)}] $X$ is a left $F$-complete prefix code.
  \item[\upshape{(ii')}] $X$ is a right $F$-complete suffix code.
  \item[\upshape{(iii)}] $X$ is an $F$-maximal prefix code and an
    $F$-maximal suffix code.
  \end{enumerate}
\end{theorem}

As a preparation for the proof of Theorem~\ref{theoremEquivMax},
we introduce the following notation. Let $F$ be a recurrent set
and let $X\subset F$.

A \emph{factorization}\index{factorization of a word} of a word $u$ is
a pair $(p,s)$ of words such that $u=ps$. We denote by $\Fact(u)$ the
set of factorizations of $u$.

Let $C(X,F)$ be the set of pairs $(u,v)$ of words such that $uvu\in
F$, $v \ne 1$ and $u$ is not an internal factor of $X$.  We define for
each pair $(u,v)\in C(X,F)$ a relation $\varphi_{u,v}$ on the set
$\Fact(u)$ as follows.  For $\pi=(p,s),\rho=(q,t)\in \Fact(u)$, one
has $(\pi,\rho)\in\varphi_{u,v}$ if and only if the pair $(\pi,\rho)$
satisfies one of the following conditions (see
Figure~\ref{figureVarphi}).
\begin{enumerate}
\item[(i)] $px=q$ for some $x\in X$,
\item[(ii)] $svq=x_1\cdots x_n$ with $n\ge 1$ and $x_i\in X$ for $1\le
  i\le n$, $s$ is a proper prefix of $x_1$ and $q$ is a proper suffix
  of $x_n$.
\end{enumerate}
Since $ps=qt$, the condition (i) is equivalent to $s=xt$. This means that
both conditions are symmetric for reading from left to right or from
right to left.
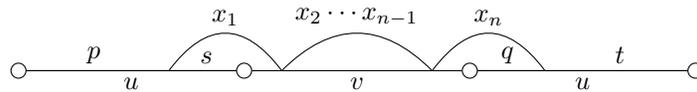
\begin{figure}[hbt]
  \centering
  \gasset{Nadjust=wh,AHnb=0}
  \begin{picture}(90,10)(-10,0)
   \node(p)(-10,0){}
    \node[Nframe=n,Nadjustdist=0](x1)(10,0){}
    \node(v)(20,0){}
    \node[Nframe=n,Nadjustdist=0](x2)(25,0){}
    \node[Nframe=n,Nadjustdist=0](xn)(45,0){}
    \node(q)(50,0){}
    \node[Nframe=n,Nadjustdist=0](x')(60,0){}
    \node(end)(80,0){}
    \drawedge[curvedepth=5](x1,x2){$x_1$}
    \drawedge[curvedepth=5](x2,xn){$x_2\cdots x_{n-1}$}
   \drawedge[curvedepth=5](xn,x'){$x_n$}
    \drawedge[ELside=r](p,v){$u$}
    \drawedge[ELside=r](v,q){$v$}
    \drawedge(q,x'){$q$}
\drawedge(x',end){$t$}
    \drawedge[ELside=r](q,end){$u$}
    \drawedge(p,x1){$p$}
    \drawedge(x1,v){$s$}
  \end{picture}

\caption{The relation $\varphi_{u,v}$ (case (ii)).}\label{figureVarphi}
\end{figure}

\noindent
We prove a series of lemmas concerning the relations $\varphi_{u,v}$
(see Exercise 6.2.1 in~\cite{BerstelPerrinReutenauer2009}).

\begin{lemma}\label{lemmaBijection1}
  Let $F$ be a recurrent set and let $X\subset F$ be an $F$-thin set.
  If $X$ is a prefix code, then for all pairs $(u,v)\in C(X,F)$, the
  relation $\varphi_{u,v}$ is a partial function from $\Fact(u)$ into
  itself, that is
  \begin{equation}
    \label{equationPartial}
    (\pi,\rho),(\pi,\rho') \in \varphi_{u,v}\quad\Rightarrow\quad\rho=\rho'\,.
  \end{equation}
  Conversely, if $X$ is an $F$-maximal suffix code, and if
  \eqref{equationPartial} holds for all pairs $(u,v)\in C(X,F)$, then
  $X$ is a prefix code.
\end{lemma}

Define the \emph{transpose}\index{transpose of a relation}
$\varphi'_{u,v}$ of the relation $\varphi_{u,v}$ by the condition
$(\rho,\pi)\in \varphi'_{u,v}$ if $(\pi,\rho)\in\varphi_{u,v}$. Then
\eqref{equationPartial} expresses the fact that the transpose
$\varphi'_{u,v}$ is injective. \medskip

\begin{proof}
  Assume first that $X$ is a prefix code.  For $(u,v)\in C(X,F)$, let
  $\pi=(p,s),\rho=(q,t),\rho'=(q',t')$ be three factorizations of $u$
  such that $(\pi,\rho),(\pi,\rho') \in \varphi_{u,v}$.  We 
  prove that $\rho = \rho'$.  By definition, the following cases may
  occur for $(\pi,\rho),(\pi,\rho')$ .
  \begin{enumerate}
  \item[(1)] 
    $px=q$ and $px'=q'$, with $x, x' \in X$,
  \item[(2)] 
    $px=q$ with $x \in X$, and $svq'=x'_1\cdots x'_m$, with $m \ge 1$
    and $x'_1,\ldots, x'_m \in X$, and moreover $s$ is a proper prefix
    of $x'_1$ and $q'$ is a proper suffix of $x'_m$,
  \item[(3)] $px'=q'$ with $x' \in X$ and $svq = x_1\cdots x_n$, with
    $n \ge 1$ and $x_1,\ldots, x_n\in X$, and moreover $s$ is a proper
    prefix of $x_1$ and $q$ is a proper suffix of $x_n$,
  \item[(4)] 
    $svq = x_1\cdots x_n$ and $svq'=x'_1\cdots x'_m$, with $n \ge 1$,
    $m \ge 1$, $x_1,\ldots, x_n,\allowbreak x'_1,\ldots, x'_m\in X$,
    and moreover $s$ is a proper prefix both of $x_1$ and of $x'_1$,
    $q$ is a proper suffix of $x_n$ and $q'$ is a proper suffix of
    $x'_m$.
  \end{enumerate}

  (1) Assume that $px=q$, $px'=q'$, with $x, x' \in X$. Since $q$ and
  $q'$ are prefixes of $u$, they are prefix-comparable. Thus $x$ and
  $x'$ are also prefix-comparable. Since $X$ is a prefix code, it
  follows that $x=x'$, whence $q=q'$ and $\rho = \rho'$.

  (2) We show that this case is impossible. Indeed, $x$ is a prefix of
  $s$ (by $ps = qt = pxt$) and $s$ is a proper prefix of $x'_1$, thus
  $x$ is a proper prefix of $x'_1$, and this is impossible because $X$
  is a prefix code.  The same argument holds in the symmetric case
  (3).

  (4) Since $u=qt=q't'$, the words $q$ and $q'$ are
  prefix-comparable. We may suppose that $q=q'w$ (see
  Figure~\ref{figureVarphi2}).  Since $svq, svq'$ are in $X^*$ and $X$
  is a prefix code, we have $w\in X^*$.  Since $X$ is a code, the
  decompositions $svq = x_1\cdots x_n = svq'w= x'_1\cdots x'_m w$
  coincide. Consequently, $w=x_{m+1}\cdots x_n$. By hypothesis, $q =
  q'w = q'x_{m+1} \cdots x_n$ is a proper suffix of $x_n$.  This
  forces $n=m$, $w=1$ and $q=q'$, hence $\rho=\rho'$.
  \begin{figure}[hbt]
    \centering
    \gasset{Nadjust=wh,AHnb=0}
    \begin{picture}(100,18)(-15,-10)
      \node(u1)(-10,0){}
      \node[Nframe=n,Nadjustdist=0](s)(10,0){}
      \node(v)(25,0){}
      \node(u2)(40,0){}
      \node[Nframe=n,Nadjustdist=0](r)(50,0){}
      \node[Nframe=n,Nadjustdist=0](q)(60,0){}
      \node(end)(80,0){}
      \drawedge[ELside=r](u1,v){$u$}
      \drawedge[ELside=r](v,u2){$v$}
      \drawedge(u1,s){$p$}
      \drawedge(s,v){$s$}
      \drawedge(u2,q){$q$}
      \drawedge(q,end){$t$}
      \drawedge[ELside=r,ELpos=60](u2,end){$u$}
      \drawedge[ELside=r](u2,r){$q'$}
      \drawedge[curvedepth=7](s,q){}
       \drawbpedge(s,-70,12,r,-110,12){}
      \drawedge[curvedepth=-7,linecolor=red](r,q){$\textcolor{red}{w}$}      
    \end{picture}
    \caption{The factorizations $(p,s)$, $(q,t)$ and
      $(q',t')$ with $t'=wt$.}\label{figureVarphi2}
  \end{figure}
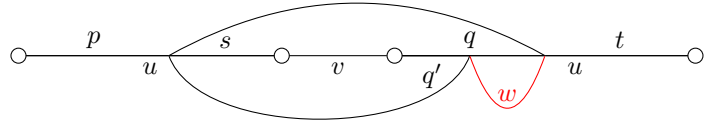

  Conversely, assume that $X$ is an $F$-maximal suffix code and that
  it is not a prefix code. Let $x',x''$ be distinct words in $X$ such
  that $x'$ is a prefix of $x''$. Set $x''=x'r'$ with $r'\ne 1$.

  Since $X$ is $F$-thin, there is a word $w \in F\setminus
  I(X)$. Since $F$ is recurrent, there is a word $r''$ such that
  $x''r''w \in F$. Let $u=r'r''w$. Then $x''r''w=x'u\in F$. Let $t$ be
  a word such that $utx'u \in F$. Set $v=tx'$.  Thus
  $(u,v)\in C(X,F)$ (see Figure~\ref{figLemma211}).  By the dual of
  Equation~\eqref{eqX-1}, there exist $p \in A^*\setminus A^*X$ and
  $z\in X^*$ such that $ut=pz$.

  Since $X$ is left $F$-complete, $p$ is a proper suffix of a word in
  $X$. Since $u \notin I(X)$, $p$ is a prefix of $u$.  Thus $z=1=t$ or
  $z\in
  X^+$. In the latter case, set
  $z=z_1\cdots z_n$ with $z_i \in X$. Since $ut=pz$, one of the
  following two cases holds:

  \begin{enumerate}
  \item[(1)] $u = pz'$, with $z', t \in X^*$,
  \item[(2)] there is an $i$ with $1 \le i \le n$ such that $z_i=rs$
    with $u=pz_1\cdots z_{i-1}r$, $t=sz_{i+1}\cdots z_n$, and $r \ne 1$,
    $s \ne 1$.
  \end{enumerate}

  In case (1), consider the three factorizations $\pi = (u, 1)$, $\rho
  = (1,u)$, $\rho' = (r',r''w)$ of $u$.  Since $r'\ne 1$, we have
  $\rho \ne\rho'$.  We have $v = t x' \in X^+$, and thus
  $(\pi,\rho)\in \varphi_{u,v}$ (this is case (ii) of the definition
  with $s=q=1$).  Next, $v r' = t x' r' = t x'' \in X^+$, with $t \in
  X^*$ and where $r'$ is a proper suffix of $x''$.  Hence $(\pi,\rho')
  \in \varphi_{u,v}$. Thus, $\varphi_{u,v}$ is not a partial function.

  In case (2), let $\pi = (pz_1\cdots z_{i-1},r)$ and let $\rho,
  \rho'$ be as above.  We have $rv = r t x' = rsz_{i+1}\cdots z_n x'
  \in X^+$, whence $(\pi,\rho)\in \varphi_{u,v}$. Next, $r v r' = r s
  z_{i+1}\cdots z_n x' r' = r s z_{i+1}\cdots z_n x'' \in X^+$, and
  $r'$ is a proper suffix of $x''$. Thus $(\pi,\rho') \in
  \varphi_{u,v}$. Since $\rho \ne\rho'$, $\varphi_{u,v}$ is not a
  partial function.
\end{proof}

\begin{figure}[hbt]
  \centering
  \gasset{Nadjust=wh,AHnb=0}
  \begin{picture}(100,25)(0,-5)
    \node(t)(40,15){}
    \node(x')(50,15){}
    \node(r')(60,15){}
    \node(r'')(70,15){}
    \node(w)(80,15){}\node(wend)(100,15){}
    \drawedge(t,x'){$t$}
    \drawedge(x',r'){}
    \drawedge(r',r''){$r'$}
    \drawedge(r'',w){$r''$}
    \drawedge(w,wend){$w$}
    
    \drawedge[curvedepth=3](x',r'){$x'$}
    \drawedge[curvedepth=-5,ELside=r,ELpos=60](x',r''){$x''$}

    \node[Nframe=n,Nadjustdist=0](z)(10,0){}
    \node(ul)(0,0){}
    \node(v)(40,0){}
    \node[Nframe=n,Nadjustdist=0](x'b)(50,0){}
    \node(ur)(60,0){}
    \node[Nframe=n,Nadjustdist=0](r''b)(70,0){}
    \node[Nframe=n,Nadjustdist=0](p'')(90,0){}
    
    \node[Nframe=n,Nadjustdist=0](p')(85,0){}
    \node(end)(100,0){}
    
    \drawedge(ul,z){$p$}
    \drawedge[curvedepth=5](z,x'b){$z$}
    \drawedge[ELside=r](ul,v){$u$}\drawedge[ELside=r](v,ur){$v$}
    \drawedge[ELpos=40](ur,end){$u$}
    
    \drawedge[dash={0.2 0.5}0](t,v){}\drawedge[dash={0.2 0.5}0](r',ur){}
    \drawedge[dash={0.2 0.5}0](x',x'b){}\drawedge[dash={0.2 0.5}0](r'',r''b){}
    
  \end{picture}
  \caption{$\varphi_{u,v}$ is not a partial function.}\label{figLemma211}
\end{figure}
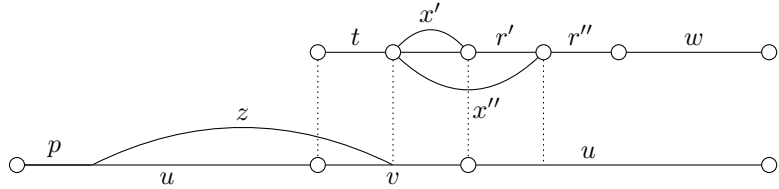

Lemma~\ref{lemmaBijection1} has a dual formulation for suffix codes:
if $X$ is a suffix code, then for all pairs $(u,v)\in C(X,F)$, the
relation $\varphi_{u,v}$ is injective: if $(\pi,\rho),(\pi',\rho) \in
\varphi_{u,v}$ , then $\pi=\pi'$. Conversely, if $X$ is an $F$-maximal
prefix code, and if this implication holds for all pairs $(u,v)\in
C(X,F)$, then $X$ is a suffix code.

Recall that a set $X\subset F$ is right $F$-complete if any
word of $F$ is a prefix of~$X^*$.

\begin{lemma}\label{lemmaBijection2}
  Let $F$ be a recurrent set and let $X\subset F$ be an $F$-thin set.
  The set $X$ is right $F$-complete if and only if, for all pairs
  $(u,v)\in C(X,F)$, the relation $\varphi_{u,v}$ contains a total
  function from $\Fact(u)$ into itself, that is for every
  $\pi\in\Fact(u)$, there exists $\rho\in\Fact(u)$ such that
  $(\pi,\rho)\in \varphi_{u,v}$.
\end{lemma}

\begin{proof}
  Assume first that $X$ is right $F$-complete.  Let $u,v\in F$ be such
  that $(u,v)\in C(X,F)$.  Let $\pi=(p,s)\in\Fact(u)$. Suppose first
  that $s$ has a prefix $x$ in $X$.  Let $s = xt$, with $x \in
  X$. Thus $u = ps = pxt$.  Let $q=px$ and $\rho = (q,t)$. Then
  $(\pi,\rho) \in \varphi_{u,v}$.  Suppose next that $s$ has no prefix
  in $X$.  Since $X$ is right $F$-complete, there exists a word $w$
  such that $svuw = x_1 \cdots x_m$, with $x_1,\ldots,x_m \in X$.

  Let $n$ be the smallest integer such that $sv$ is a prefix of $x_1
  \cdots x_n$, $1 \leq n \leq m$.  Let $q$ be the prefix of $uw$ such
  that $svq = x_1 \cdots x_n$.  Since $sv \not = 1$, $q$ is a proper
  suffix of $x_n$.  The word $q$ is a prefix of $u$ since $u$ is not
  an internal factor of $X$. Define the factorization $\rho=(q,t)$ of
  $u$ by $svq = x_1 \cdots x_n$. Since $s$ has no prefix in $X$, the
  word $s$ is a proper prefix of $x_1$. Therefore, $(\pi,\rho) \in
  \varphi_{u,v}$. This shows that $\varphi_{u,v}$ contains a total
  function.

  Conversely, assume that for all $(u,v)\in C(X,F)$, the relation
  $\varphi_{u,v}$ contains a total function from $\Fact(u)$ into
  itself. We show that any $u\in F$ is prefix-comparable
  with a word of $X$.  By Proposition~\ref{propositionGlobal},
  this implies that $X$ is right $F$-complete.

  Let $u\in F$.  Since $X$ is $F$-thin, the set $F\setminus I(X)$ is
  nonempty.  Let $w\in F\setminus I(X)$ and let $v$ be such that
  $uvw\in F$. Set $r=uvw$. Note that $r\in F\setminus I(X)$. Let $z\ne
  1$ be such that $rzr\in F$. Then $(r,z)\in C(X,F)$. Set $\pi=(1,r)$.
  Since $\varphi_{r,z}$ contains a total function, there is a
  factorization $\rho=(q,t)$ of $r$ such that
  $(\pi,\rho)\in\varphi_{r,z}$. If $q\in X$, then $r$ has the prefix $q$
  in $X$, the word $u$ is prefix-comparable with $q$, and we obtain the
  conclusion. Otherwise, we have $uvwzq=x_1\cdots x_n$ with $x_i\in X$
  and $uvw$ is a prefix of $x_1$, whence our conclusion again.
\end{proof}

Lemma~\ref{lemmaBijection2} has a dual formulation for left
$F$-complete sets: the set $X$ is left $F$-complete if and only if,
for all pairs $(u,v)\in C(X,F)$, the transpose of the relation
$\varphi_{u,v}$ contains a total function from $\Fact(u)$ into itself.

\begin{proposition}\label{lemmaBijection}
  Let $F$ be a recurrent set and let $X\subset F$ be an $F$-thin and
  $F$-maximal prefix code.  Then $X$ is a suffix code if and only if
  it is left $F$-complete.
\end{proposition}

\begin{proof}
  Since $X$ is an $F$-maximal prefix code, by
  Lemmas~\ref{lemmaBijection1} and \ref{lemmaBijection2}, for any pair
  $(u,v)\in C(X,F)$, the relation $\varphi_{u,v}$ is a total function
  from $\Fact(u)$ into itself.

  Assume first that $X$ is a suffix code. Then, by the dual of
  Lemma~\ref{lemmaBijection1}, for any pair $(u,v)\in C(X,F)$, the
  function $\varphi_{u,v}$ from $\Fact(u)$ into itself is
  injective. 
  Since $\Fact(u)$ is a finite set, $\varphi_{u,v}$ is
  also surjective for any pair $(u,v)\in C(X,F)$. This implies by the
  dual of Lemma~\ref{lemmaBijection2} that $X$ is left $F$-complete.

  Assume conversely that $X$ is left $F$-complete. By the dual of
  Lemma~\ref{lemmaBijection2}, the function $\varphi_{u,v}$ maps
  $\Fact(u)$ onto itself for every pair $(u,v)\in C(X,F)$. This
  implies as above that it is also injective. By the dual of
  Lemma~\ref{lemmaBijection1}, and since $X$ is an $F$-maximal prefix
  code, $X$ is a suffix code.
\end{proof}

Proposition~\ref{lemmaBijection} has a dual formulation for an
$F$-maximal suffix code.

\medskip\par
\begin{proofof}{of Theorem \ref{theoremEquivMax}}
  We first show that (i) implies (ii). If $X$ is an $F$-maximal suffix
  code, then $X$ is left $F$-complete and thus condition (ii) is true.
  Assume next that $X$ is an $F$-maximal prefix code. Since $X$ is
  suffix, by Proposition~\ref{lemmaBijection}, it is left $F$-complete
  and thus (ii) holds. Finally assume that $X$ is neither an
  $F$-maximal prefix code nor an $F$-maximal suffix code. Let $y,z\in
  F$ be such that $X\cup y$ is prefix and $X\cup z$ is suffix. Since
  $F$ is recurrent, there is a word $u$ such that $yuz\in
  F$. Then $X\cup yuz$ is bifix and thus we get a contradiction.

The proof that (i) implies (ii') is similar.

(ii) implies (iii). Consider the set $Y=X\setminus A^+X$. It is a
suffix code by definition. It is prefix since it is contained in $X$.
It is left $F$-complete. Indeed, one has $A^*X=A^*Y$ and thus $A^*Y$
is left $F$-dense by the dual of
Proposition~\ref{propositionGlobal}.  Hence $Y$ is an
$F$-maximal suffix code.  By the dual of
Proposition~\ref{lemmaBijection}, the set $Y$ is right
$F$-complete. Thus $Y$ is an $F$-maximal prefix code. This implies
that $X=Y$ and thus that $X$ is an $F$-maximal prefix code and an
$F$-maximal suffix code.

The proof that (ii') implies (iii) is similar.  It is clear that (iii)
implies (i).
\end{proofof}

\begin{example}
  Let $A=\{a,b\}$ and let $F$ be the set of words without factor $bb$
  (Example~\ref{exampleGolden}). The set $X=\{aaa,aaba,ab,baa,baba\}$
  is a finite $F$-maximal bifix code.  As an example of computation of
  the relation $\varphi_{u,v}$, note that for $u=aaa$ and $v=b$, we
  have $\Fact(u)=\{\pi_1,\pi_2,\pi_3,\pi_4\}$ with $\pi_1=(1,aaa)$,
  $\pi_2=(a,aa)$, $\pi_3=(aa,a)$, $\pi_4=(aaa,1)$. The function
  $\varphi_{u,v}$ is the cycle $(\pi_1\pi_4\pi_3)$ and fixes $\pi_2$.
\end{example}

The following example shows that Theorem~\ref{theoremEquivMax} is
false if $F$ is not recurrent.

\begin{example}
  Let $F=a^*b^*$. Then $X=\{aa,ab,b\}$ is an $F$-maximal prefix
  code. It is not a suffix code but it is left $F$-complete as it can
  be easily verified.
\end{example}

Let $F\subset A^*$ be a factorial set. The
$F$-\emph{degree}\index{F-degree@$F$-degree}, denoted $d_F(X)$, of a
set $X\subset A^*$ is the maximal number of parses of words of $F$
with respect to $X$, that is
\begin{displaymath}
   d_F(X)=\max_{w\in F}\, \pars_X(w)\,.
\end{displaymath}
The $F$-degree of a set $X$ is finite or infinite.  The $A^*$-degree
is called the \emph{degree}\index{degree}, and is denoted $d(X)$.
Observe that $d_F(X)=d_F(X\cap F)$, and that $d_F(X)\le d(X)$.

The following is a generalization of Theorem 6.3.1
in~\cite{BerstelPerrinReutenauer2009}.

\begin{theorem}\label{theoremDegree}
  Let $F$ be a recurrent set and let $X\subset F$ be a bifix code.
  Then $X$ is an $F$-thin and $F$-maximal bifix code if and only if
  its $F$-degree $d_F(X)$ is finite.
  In this case,
  \begin{equation}
    I(X)=\{w\in F\mid \pars_X(w)< d_F(X)\}\,. \label{eqH}
  \end{equation}
\end{theorem}

\begin{proof}
  Assume first that $X$ is an $F$-thin and $F$-maximal bifix code.
  Since $X$ is $F$-thin, $F\setminus I(X)$ is not empty.  Let $u\in
  F\setminus I(X)$ and $w\in F$. Since $F$ is recurrent, there is a
  word $v\in F$ such that $uvw\in F$.  Since $X$ is prefix, by
  Proposition~\ref{propInterpretations}, the number of parses of $u$
  is equal to the number of prefixes of $u$ which have no suffix in
  $X$. Since $X$ is left $F$-complete, the set of words in $F$ which
  have no suffix in $X$ coincides with the set $S$ of words which are
  proper suffixes of words in $X$. Since $u$ is not an internal factor
  of a word in $X$, any prefix of $uvw$ which is in $S$ is a prefix of
  $u$.  Thus
  $\pars_X(uvw)=(\u(S)\underline{A}^*,uvw)=(\u(S)\underline{A}^*,u)=\pars_X(u)$. Since by
  Equation~\eqref{eqL1}, $\pars_X(w)\le\pars_X(uvw) $, we get $\pars_X(w)\le
  \pars_X(u)$. This shows that $\pars_X$ is bounded, and thus that the $F$-degree
  of $X$ is finite. Moreover, this shows that $F\setminus I(X)$ is
  contained in the set of words of $F$ with maximal value of $\pars_X$.
   Conversely, consider $w\in I(X)$. Then there exists $w'\in X$ and
  $p,s\in A^+$ such that $w'=pws$. Then by Equation~\eqref{eqLi}
  $\pars_X(w')>\pars_X(w)$, and thus $\pars_X(w)$ is not maximal in $F$.  This proves
  Equation~\eqref{eqH}.

  Conversely, let $w\in F$ be a word with $\pars_X(w)=d_F(X)$.  For any
  nonempty word $u\in F$ such that $uw\in F$ we have $uw\in
  XA^*$. Indeed, set $u=au'$ with $a\in A$ and $u'\in F$. Then
  $\pars_X(au'w)\ge\pars_X(u'w)\ge\pars_X(w)$ by Equation~\eqref{eqL1}. This implies
  $\pars_X(au'w)=\pars_X(u'w)=\pars_X(w)$.  By  the dual of
  Equation~\eqref{eqL} we obtain that $uw\in XA^*$.

  This implies first that $X$ is $F$-thin and next that $XA^*$ is
  right $F$-dense. Indeed suppose that $w$ is an internal factor of a
  word in $X$. Let $p,s\in F\setminus 1$ be such that $pws\in
  X$. Since $pw\in F$, 
  the previous argument shows that $pw\in XA^*$, a
  contradiction. Thus $w\in F\setminus I(X)$.  This shows that $X$ is
  $F$-thin.  

  Next, and since $F$ is recurrent, for any $v \in F$, there is a word
  $u\in F$ such that $vuw\in F$.  Then $vuw\in XA^*$ by using again
  the above argument.  Thus $XA^*$ is right $F$-dense and $X$ is an
  $F$-maximal bifix code by Theorem~\ref{theoremEquivMax}.
\end{proof}

\begin{example}\label{exampleBifixDegree2}
  Let $F$ be the Fibonacci set. The set
  $X=\{a,bab,baab\}$ is a finite bifix code. Since it is finite, it is
  $F$-thin. It is an $F$-maximal prefix code as one may check on
  Figure~\ref{figProbaFibo}. Thus it is, by
  Theorem~\ref{theoremEquivMax}, an $F$-thin and $F$-maximal bifix
  code. The parses of the word $bab$ are $(1,bab,1)$ and
  $(b,a,b)$. Since $bab$ is not in $I(X)$, one has $d_F(X)=2$.
\end{example}

\begin{example}\label{exampleBifix3}
  Let $F$ be the Fibonacci set. The set
  $X=\{aaba,ab,baa,baba\}$ is a bifix code. It is $F$-maximal since it
  is right $F$-complete (see Figure~\ref{figProbaFibo}).  It has
  $F$-degree $3$. Indeed, the word $aaba$ has three parses
  $(1,aaba,1)$, $(a,ab,a)$ and $(aa,1,ba)$ and it is in $F\setminus
  I(X)$.
\end{example}
The following result establishes the link between maximal bifix codes
and $F$-maximal ones.

\begin{theorem} \label{proposition1} Let $F$ be a recurrent set. For
  any thin maximal bifix code $X \subseteq A^+$ of degree $d$, the set
  $Y = X \cap F$ is an $F$-thin and $F$-maximal bifix code.  One has
  $d_F(Y)\le d$ with equality when $X$ is finite.
\end{theorem}

\begin{proof}
  Recall that $d_F(Y)=d_F(X\cap F)=d_F(X)\le d$.  Thus $d_F(Y)$ is
  finite and by Theorem~\ref{theoremDegree}, $Y$ is an $F$-thin and
  $F$-maximal bifix code. If $X$ is finite, then each word which is in
  $F$ and is longer than the longest words in $X$ has $d$
  parses. Thus $d_F(X)=d$, whence $d_F(Y)=d$.
\end{proof}

\begin{example}\label{exFiboDegre2}
  The set $X=a\cup ba^*b$ is a maximal bifix code of degree 2. Let $F$
  be the Fibonacci set. Then $X\cap
  F=\{a,baab,bab\}$ (see Figure~\ref{figProbaFibo}).

  As another example, let $Z = \{a^3, a^2ba, a^2b^2, ab, ba^2, baba,
  bab^2, b^2a, b^3\}$.  The set $Z$ is a finite maximal bifix code of
  degree $3$ (see~\cite{BerstelPerrinReutenauer2009}). Then $Z \cap F
  = \{a^2ba, ab, ba^2, baba \}$ (see Figure~\ref{figProbaFibo}).
\end{example}

\begin{example}\label{exMorseDegre2}
  Let $F$ be the  Thue--Morse set. Consider again
  $X=a\cup ba^*b$. Then $X\cap F=\{a,baab,bab,bb\}$ is a finite
  $F$-maximal bifix code of $F$-degree $2$ (see
  Figure~\ref{figProbaMorse}).
\end{example}

The following examples show that a strict inequality can hold in
Theorem~\ref{proposition1}. The second example shows that this may
happen even if all letters occur in
the words of $F$.

\begin{example}
  Let $A=\{a,b\}$ and let $X=a\cup ba^*b$. The set $X$ is a maximal
  bifix code of degree $2$. Let $F=a^*$. Then $F$ is a recurrent
  set. We have $Y=X\cap F=a$. The $F$-degree of $Y$ is $1$.
\end{example}

\begin{example}\label{exampleStrict} Let $A=\{a,b\}$ and
  let $X\subset A^+$ be the maximal bifix code of degree $3$ with
  kernel $K=\{aa,ab,ba\}$.  Let $F$ be the Fibonacci set.  Since
  $K=A^2\cap F$, $K$ is an $F$-maximal bifix code.  Since $K\subset
  X\cap F$ and $K$ is $F$-maximal, one has $X\cap F=K$.  Next
  $K=A^2\cap F$ and Theorem~\ref{proposition1} imply that
  $d_F(K)=2$. Thus $d(X)=3$ and $d_F(X\cap F)=2$.
\end{example}

\subsection{Derivation}
We first show that the notion of derived code can be extended to $F$-maximal
bifix codes. The following result generalizes Proposition 6.4.4 in
\cite{BerstelPerrinReutenauer2009}.

The \emph{kernel}\index{kernel} of a set of words $X$ is the set of
words in $X$ which are internal factors of words in $X$. We denote by
$K(X)$ the kernel of $X$. Note that $K(X)=I(X)\cap X$.

\begin{theorem}\label{thmDerived}
  Let $F$ be a recurrent set. Let $X\subset F$ be a bifix code of
  finite $F$-degree $d\ge 2$.  Set $I=I(X)$ and $K=K(X)$. Let
  $G=(IA\cap F)\setminus I$ and $D=(AI\cap F)\setminus I$. Then the
  set $X'=K\cup(G\cap D)$ is a bifix code of $F$-degree $d-1$.
\end{theorem}

The code $X'$ is called the \emph{derived}\index{derived
  code}\index{code!derived} code of $X$ with respect to $F$ or
$F$-derived code.

The proof uses two lemmas.  Let $P$ be the set of proper prefixes of
$X$ and let $S$ be the set of proper suffixes of $X$.

\begin{lemma}\label{lemmaDeriv}
  One has $G\subset S$ and $D\subset P$.
\end{lemma}

\begin{proof}
By Theorem~\ref{theoremDegree}, the parse enumerator of $X$ is bounded 
 on $F$ and  $F\setminus I(X)=F\setminus I$ is the set of words in $F$
with maximal value $d_F(X)$.
Let $y=ha$ be in $G$ with $h\in I$ and $a\in A$. 
Since $y\notin I$, we have $\pars_X(ha)>\pars_X(h)$. Thus, by
Proposition~\ref{propositioneqL},
$y=ha$ does not
have a suffix in $X$. Since $A^*X$ is left $F$-dense, this implies that
$y$ is a proper suffix of a word in $X$. Thus $y$ is in $S$.
The proof that $D\subset P$ is symmetrical.
\end{proof}

\begin{lemma}\label{lemmaDeriv2}
  For any $x\in X\setminus K$, the shortest prefix of $x$ which is not
  in $I$ is in $X'$.
\end{lemma}

\begin{proof}
  Since $x\notin K$, we have $x\notin I$.  Let $x'$ be the shortest
  prefix of $x$ which is not in $I$ or, equivalently such that
  $\pars_X(x')=d_F(X)$.  Let us show that $x'\in X'$. First, $x'$ is a
  proper prefix of $x$. Set indeed $x=pa$ with $p\in A^*$ and $a\in
  A$. Since $x\in X$, we have by Equation~\eqref{eqL},
  $\pars_X(x)=\pars_X(p)$. Thus $p \not \in I$ and $x'$ is a prefix of $p$.

  Since $1\in I$, we have $x'\ne1$.  Set $x'=p'a'$ with $p'\in A^*$
  and $a'\in A$. By definition of $x'$ we have $p'\in I$.  Thus $x'\in
  G=(IA\cap F)\setminus I$.

Next, set $x'=a''s$ with $a''\in A$ and $s\in A^*$. Since $x'\notin XA^*$,
we have
by the dual of Equation~\eqref{eqL}, $\pars_X(s)<\pars_X(x')$. Thus $s$ is
in
$I$. This shows that $x'\in D$.
Thus we conclude that $x'\in G\cap D\subset X'$.
\end{proof}

There is a dual of Lemma~\ref{lemmaDeriv2} concerning the shortest
suffix of a word in $X\setminus K$.

\noindent\emph{Proof of Theorem}~\ref{thmDerived}.

We first prove that $X'$ is a prefix code.  Suppose first that $k\in
K$ is a prefix of a word $z$ in $G\cap D$.  By Lemma~\ref{lemmaDeriv},
a word in $D$ is a proper prefix of $X$. Thus $k\in X$ would be a
proper prefix of $X$, which is impossible since $X$ is prefix.

Suppose next that a word $u$ of $G\cap D$ is a prefix of a word $k$ in $K$.
Since $k$ is in $I$, it follows that $u$ is in $I$, a contradiction.

Finally, no word $y\in G \cap D$ can be a proper prefix of another
 word $y'$ in $G \cap D$, otherwise $y' = yz$, with
$z \in A^+$. Therefore, since $G \subset S$ by Lemma~\ref{lemmaDeriv},
there is $t \in A^+$
such that $ty' = tyz \in X$. Consequently, $y \in G \cap I$,
a contradiction.

Thus $X'$ is a prefix code. To show that it is $F$-maximal, it is
enough to show that any word in $X$ has a prefix in $X'$. 

Consider indeed $x\in X$. If $x$ is in $K$ then $x\in X'$. Otherwise,
let $x'$ be the shortest prefix of $x$ which is not in $I$.
 By Lemma~\ref{lemmaDeriv2}, we have $x'\in X'$.

Thus $X'$ is an $F$-maximal prefix code.

A symmetric argument shows that $X'$ is an $F$-maximal suffix code.

Let us show that $d_F(X')=d_F(X)-1$. We first note that $G\cap
D\ne\emptyset$. Indeed, let $x\in X$ be such that $\pars_X(x)$ is maximal
on $X$. If $x$ were an internal factor of a word $y\in X$, then by
Equation~\eqref{eqLi} $\pars_X(x)<\pars_X(y)$ which contradicts our
assumption. Thus $x\notin K$. This shows that $K$ is not an
$F$-maximal bifix code and thus that $X'\setminus K=G\cap D\ne
\emptyset$. Consider $x'\in G\cap D$.  Since $(G \cap D)\cap I(X)$ is
empty, and since $I(X')\subset I(X)$, $x'$ cannot be in $I(X')$.  Thus
the number of parses of $x'$ with respect to $X'$ is $d_F(X')$.

Let $P'$ be the set of proper prefixes of  $X'$.  We show that
$x'$ has $d_F(X)-1$ suffixes which are in $P'$. This will show that
$d_F(X')=d_F(X)-1$ by the dual of
Proposition~\ref{propInterpretations}.

Since $x'\in F\setminus I$, we have $\pars_X(x')=d_F(X)$. Thus $x'$ has
$d_F(X)$ suffixes in $P$. One of them is $x'$ itself since $x'\in
D\subset P$. Let $p$ be a proper suffix of $x'$ which is in $P$. Let
us show that $p$ does not have a prefix in $X'$. Indeed, arguing by
contradiction, assume that $x''\in X'$ is a prefix of $p$.  We cannot
have $x''\in K$ since $p$ is a proper prefix of a word in $X$. We
cannot have either $x''\in G\cap D$.  Indeed, since $x'$ is in $AI$,
$p$ is in $I$ and thus also $x'' \in I$. Thus $p$ cannot have a prefix
in $X'$. Since $X'$ is an $F$-maximal prefix code, this implies that
$p$ is a proper prefix of  $X'$. Thus, the $d_F(X)-1$ proper
suffixes of $x'$ which are in $P$ are in $P'$.  \QED

\begin{example}\label{exampleDerived}
  Let $F$ be the Fibonacci set. Let
  $X=\{a,bab,baab\}$. The set $X$ is an $F$-thin and $F$-maximal bifix
  code of $F$-degree 2 (see Example~\ref{exampleBifixDegree2}).  We have
  $K=\{a\}$, $I=\{1,a,aa\}$, $G=\{b,ab,aab\}$ and $D=\{b,ba,baa\}$.
  Thus $X'=\{a,b\}$.
\end{example}

The following is a generalization of Proposition 6.3.14 in~\cite{BerstelPerrinReutenauer2009}.
\begin{proposition}\label{propositionScapH}
Let $F$ be a recurrent set. Let $X\subset F$ be a
bifix code of $F$-degree $d\ge 2$. Let $S$ be the set of proper suffixes of $X$
and set $I=I(X)$. The set $S\setminus I$ is an $F$-maximal prefix
code
and the set $S\cap I$ is the set of proper suffixes of the derived
code $X'$.
\end{proposition}

The proof uses the following lemma.
\begin{lemma}\label{lemmaMaxSinH}
Let $F$ be a recurrent set. Let $X\subset F$ be an $F$-thin and
$F$-maximal bifix code. Let $S$ be the set of proper suffixes of $X$
and set $I=I(X)$. For any $w\in F\setminus I$
the
longest prefix of $w$ which is in $S$ is not in $I$.
\end{lemma}
\begin{proof}
  Let $s$ be the longest prefix of $w$ which is in $S$.  Set
  $w=st$. Let us show that for any prefix $t'$ of $t$, we have
  $\pars_X(st')=\pars_X(s)$.  It is true for $t'=1$. Assume that it is
  true for $t'$ and let $a\in A$ be the letter such that $t'a$ is a
  prefix of $t$. Since $st'a\notin S$, we have $st'a\in A^*X$. Thus by
  Equation~\eqref{eqL}, this implies $\pars_X(st'a)=\pars_X(st')$.
  Thus $\pars_X(st'a)=\pars_X(s)$. We conclude that
  $\pars_X(st)=\pars_X(s)$. Since $w=st$ is in $F\setminus I$, and
  since $F\setminus I$ is the set of words in $F$ with maximal value
  of $\pars_X$, this implies that $s\in F\setminus I$.
\end{proof}
This lemma has a dual statement for the longest suffix of a word in 
$w\in F\setminus I$ which is in $P$.

\begin{proofof}{of Proposition~\ref{propositionScapH}}
  Set $Y=S\setminus I$.  Let us first show that $Y$ is prefix. Assume
  that $u,uv\in Y$. Since $uv\in S$ there is a nonempty word $p$ such
  that $puv\in X$. Since $u\notin I$, this forces $v=1$. Thus $Y$ is
  prefix.

  We show next that $YA^*$ is right $F$-dense. Consider $u\in F$ and
  let $w\in F\setminus I$. Since $F$ is recurrent, there exists $v\in
  F$ such that $uvw\in F$. Let $s$ be the longest word of $S$ which is
  a prefix of $uvw$. By Lemma~\ref{lemmaMaxSinH}, we have $s\in
  F\setminus I$.  Thus $s\in S\setminus I=Y$ and $uvw\in YA^*$. This
  shows that $YA^*$ is right $F$-dense.

  Let us now show that the set $S'$ of proper suffixes of the words of
  $X'$ is $S\cap I$.  Let $s$ be a proper suffix of a word $x'\in
  X'$. If $x'\in K$, then $s$ is in $S\cap I$. Suppose next that
  $x'\in G\cap D$. Since $G\subset S$ by Lemma~\ref{lemmaDeriv}, we
  have $s\in S$. Furthermore, since $D\subset AI$, we have $s\in
  I$. This shows that $s\in S\cap I$.

  Conversely, let $s$ be in $S\cap I$. Let $x\in X$ be such that $s$
  is a proper suffix of $x$. If $x$ is in $K$ then $x$ is in $X'$ and
  thus $s$ is in $S'$. Otherwise, let $y$ be the shortest suffix of
  $x$ which is in not in $I$. By the dual of Lemma~\ref{lemmaDeriv2},
  the word $y$ is in $X'$. Then $s$ is a proper suffix of $y$ (since
  $s\in I$ and $y\notin I$) and therefore $s$ is in $S'$.
\end{proofof}

There is a dual version of Proposition~\ref{propositionScapH}
concerning the set of proper prefixes of an $F$-thin and $F$-maximal
bifix code $X\subset F$.

The following property generalizes Theorem 6.3.15
in~\cite{BerstelPerrinReutenauer2009}.

\begin{theorem}\label{theoremDisjointUnion}
  Let $F$ be a recurrent set. Let $X$ be a bifix code of finite
  $F$-degree $d$.  The set of its nonempty proper suffixes is a
  disjoint union of $d-1$ $F$-maximal prefix codes.
\end{theorem}

\begin{proof}
  Let $S$ be the set of proper suffixes of $X$.  If $d=1$, then
  $S\setminus 1$ is empty. If $d\ge 2$, by
  Proposition~\ref{propositionScapH}, the set $Y=S\setminus I$ is an
  $F$-maximal prefix code and the set $S\cap I$ is equal to the set
  $S'$ of proper suffixes of the words of the derived code
  $X'$. Arguing by induction, the set $S'\setminus 1$ is a disjoint
  union of $d-2$ $F$-maximal prefix codes. Thus $S\setminus 1=Y\cup
  (S'\setminus 1)$ is a disjoint union of $d-1$ $F$-maximal prefix
  codes.
\end{proof}

The following generalizes Corollary 6.3.16 in
\cite{BerstelPerrinReutenauer2009}, with two restrictions.
First, it applies only in the case of finite maximal bifix codes instead of
thin 
bifix codes (in order to be able to use
Proposition~\ref{propMaxPrefixCode}).
Next, it applies only for recurrent sets such that there exists a
positive
invariant probability distribution (in order to be able to use
Proposition~\ref{propositionP}).
\begin{corollary}\label{corollaryAveragelength}
Let $F$ be a recurrent set such that there exists a positive
invariant
probability distribution $\proba$ on $F$. Let $X$ be a finite  bifix code
of finite
$F$-degree $d$. The average length of $X$ with respect to $\proba$  is equal to $d$.
\end{corollary}
\begin{proof}
Let $\proba$ be a positive invariant probability distribution on $F$. By
the dual of
Proposition \ref{propositionP}, one has $\lambda(X)=\proba(S)$. In view
of Theorem~\ref{theoremDisjointUnion}, we have $S\setminus 1=Y_1\cup\ldots\cup
Y_{d-1}$
where each $Y_i$ is a finite $F$-maximal prefix code. By
Proposition~\ref{propMaxPrefixCode},
we have $\proba(Y_i)=1$ for $1\le i\le d-1$. Thus $\lambda(X)=d$.
\end{proof}

\begin{example}
  Let $F$ be the Fibonacci set and let $X=\{a,bab,baab\}$
  (Example~\ref{exampleDerived}). The set $X$ is an $F$-maximal bifix
  code of $F$-degree $2$.  With respect to the unique invariant
  probability distribution of $F$ (Example~\ref{exProba}), we have
  $\lambda(X)=\lambda+3(2-3\lambda)+4(2\lambda-1)=2$.
\end{example}

Now we show that an $F$-thin and $F$-maximal bifix code
is determined by its $F$-degree and its kernel.
We first prove the following generalization of  Proposition 6.4.1
from~\cite{BerstelPerrinReutenauer2009}. 

\begin{proposition}\label{proposition631}
  Let $F$ be a recurrent set.  Let $X\subset F$ be a bifix code of
  finite $F$-degree $d$ and let $K$ be the kernel of $X$. Let $Y$ be a
  set such that $K\subset Y\subset X$. Then for all $w\in I(X)\cup Y$,
  \begin{equation}
    \pars_Y(w)=\pars_X(w).\label{eqprop631a}
  \end{equation}
  For all $w\in F$,
  \begin{equation}
    \pars_X(w)=\min\{d,\pars_Y(w)\}.\label{eqprop631}
  \end{equation}
\end{proposition}

\begin{proof}
  Denote by $F(w)$ the set of factors of the word $w$.  Notice that
  Equation~\eqref{eqLB} is equivalent to $L_X = \u(A^*)(1-\u(X))\u(A^*)$.
  Thus, to prove \eqref{eqprop631a}, we have to show that for any
  $w\in I(X)\cup Y$ one has $F(w)\cap X=F(w)\cap Y$.  The inclusion
  $F(w)\cap Y\subset F(w)\cap X$ is clear.  Conversely, if $w$ is in
  $I(X)$, then $F(w)\cap X\subset K$ and thus $F(w)\cap X\subset
  F(w)\cap Y$. Next, assume that $w$ is in $Y$. The words in $F(w)\cap
  X$ other than $w$ are all in $K$. Thus we have again $F(w)\cap
  X\subset F(w)\cap Y$.

  To show Equation~\eqref{eqprop631}, assume first that $w\in
  I(X)$. Then $\pars_X(w)<d$ by Theorem~\ref{theoremDegree}. Moreover,
  $\pars_X(w)=\pars_Y(w)$ by Equation~\eqref{eqprop631a}. Thus
  Equation~\eqref{eqprop631} holds.  Next, suppose that $w\in
  F\setminus I(X)$.  Then $\pars_X(w)=d$. Since $Y\subset X$, we have
  $\pars_X(w)\le\pars_Y(w)$ by Equation~\eqref{eqLsub}.  This
  proves~\eqref{eqprop631}.
\end{proof}

Proposition~\ref{proposition631} will be used to prove the following
generalization of Theorem 6.4.2 in~\cite{BerstelPerrinReutenauer2009}.

\begin{theorem}\label{theoremkerneldegree}
  Let $F$ be a recurrent set and let $X\subset F$ be a bifix code of
  finite $F$-degree $d$. For any $w\in F$, one has
  \begin{displaymath}
    \pars_X(w)=\min\{d,\pars_{K(X)}(w)\}\,.
  \end{displaymath}
  In particular $X$ is determined by its $F$-degree and its kernel.
\end{theorem}

\begin{proof}
  Take $Y=K(X)$ in Proposition~\ref{proposition631}. Then the formula
  follows from Equation~\eqref{eqprop631}. Next $X$ is determined by
  $L_X$, and so by $\pars_X$, through Equation~\eqref{eqLB}.
\end{proof}

We now state the following generalization of Theorem 6.4.3
in~\cite{BerstelPerrinReutenauer2009}.

\begin{theorem}\label{theorem643}
  Let $F$ be a recurrent set.  A bifix code $Y\subset F$ is the kernel
  of some bifix code of finite $F$-degree $d$ if and only if
  \begin{enumerate}
  \item[\upshape{(i)}] $Y$ is not an $F$-maximal bifix code,
  \item[\upshape{(ii)}] $\max\{\pars_Y(y)\mid  y\in Y\}\le d-1$.
  \end{enumerate}
\end{theorem}

\begin{proof}
  Let $X$ be an $F$-thin and $F$-maximal bifix code of $F$-degree $d$
  and let $Y=K(X)$ be its kernel. Condition (i) is satisfied because
  $X=Y$ implies that $X$ is equal to its derived code which has
  $F$-degree $d-1$. Moreover, for every $y\in Y$ one has $\pars_X(y)\le
  d-1$. Since $\pars_X(y)=\pars_Y(y)$ by Equation~\eqref{eqprop631a},
  condition (ii) is also satisfied.

  Conversely, let $Y\subset F$ be a bifix code satisfying conditions
  (i) and (ii).  Let $\pars:A^*\rightarrow\N$ be the function defined by
  \begin{displaymath}
    \pars(w)=\min\{d,\pars_Y(w)\}.
  \end{displaymath}
  It can be verified that the function
$\pars$ satisfies the four conditions of
  Proposition~\ref{proposition6111}. Thus $\pars$ is the parse
  enumerator  of
  a bifix code $Z$. Let $X=Z\cap F$. Then $\pars_X$ and $\pars$ have
  the same restriction to $F$. Since $\pars$ is bounded on $F$, the
  same holds for $\pars_X$. This implies that the code $X$ is
  an $F$-thin and $F$-maximal bifix code by
  Theorem~\ref{theoremDegree}. Since the code $Y$ is not an
  $F$-maximal bifix code, the $F$--parse
  enumerator $\pars_Y$ is not
  bounded. Consequently $\max\{\pars(w)\mid w\in F\}=d$, showing that $X$
  has $F$-degree $d$.  Let us prove finally the $Y$ is the kernel of
  $X$. Since, by condition (ii), $\max\{\pars_Y(y)\mid y\in Y\}\le d-1$,
  we have $Y\subset I(X)$.

  Moreover, for $w\in I(X)$ we have $\pars_X(w)=\pars_Y(w)$.  Let $L$
  (resp. $L_Y$) be the indicator of $X$ (resp. of $Y$).  Since
  $1-\u(X)=(1-\u(A))L(1-\u(A))$ and $1-\u(Y)=(1-\u(A))L_Y(1-\u(A))$ by
  Equation~\eqref{eqLB}, we conclude that for $w\in I(X)$, we have
  $(\u(X),w)=(\u(Y),w)$. This implies that if $w\in I(X)$, then $w$ is
  in $X$ if and only if $w$ is in $Y$. Thus $K(X) = I(X) \cap X =
  I(X)\cap Y =Y$ and $Y$ is the kernel of $X$.
\end{proof}

\begin{figure}[hbt]
  \centering
  \gasset{Nadjust=wh,AHnb=0}
  \begin{picture}(110,20)
    \put(0,0){
      \begin{picture}(20,20)
        \node(1)(0,5){}\node(a)(10,10){}\node(b)(10,0){}
        \node[Nmr=0](aa)(20,15){}\node[Nmr=0](ba)(20,0){}\node[Nmr=0](ab)(20,5){}
        \drawedge(1,a){$a$}\drawedge[ELside=r](1,b){$b$}
        \drawedge(a,aa){$a$}\drawedge[ELside=r](a,ab){$b$}\drawedge[ELside=r](b,ba){$a$}
      \end{picture}
    }
    \put(30,0){
      \begin{picture}(40,20)(0,-5)
        \node(1)(0,5){}\node[Nmr=0](a)(10,10){}\node(b)(10,0){}
        \node(ba)(20,0){}\node(baa)(30,5){}\node[Nmr=0](bab)(30,-5){}
        \node[Nmr=0](baab)(40,5){}
        \drawedge(1,a){$a$}\drawedge[ELside=r](1,b){$b$}
        \drawedge(b,ba){$a$}
        \drawedge(ba,baa){$a$}
        \drawedge[ELside=r](ba,bab){$b$}\drawedge(baa,baab){$b$}
      \end{picture}
    }
    \put(80,0){
      \begin{picture}(20,20)(0,)
        \node(1)(0,5){}\node(a)(10,10){}\node[Nmr=0](b)(10,0){}
        \node[Nmr=0](aa)(20,15){}\node(ab)(20,5){}
        \node[Nmr=0](aba)(30,10){}
        \drawedge(1,a){$a$}\drawedge[ELside=r](1,b){$b$}
        \drawedge(a,aa){$a$}\drawedge[ELside=r](a,ab){$b$}\drawedge(ab,aba){$a$}
      \end{picture}
    }
  \end{picture}
  \caption{The three $F$-maximal bifix codes of $F$-degree $2$ in the Fibonacci
    set~$F$.}\label{figureBifix2}
\end{figure}
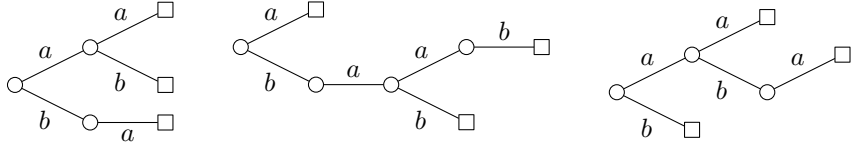

\begin{example}\label{exampleFiboDegree2}
  Let $A=\{a,b\}$ and let $F\subset A^*$ be the Fibonacci set. There
  are three maximal bifix codes of $F$-degree $2$ in $F$ represented
  on Figure~\ref{figureBifix2}. Indeed, by Theorem~\ref{theorem643},
  the possible kernels are $\emptyset$, $\{a\}$ and $\{b\}$.
\end{example}

\subsection{Finite maximal bifix codes}

The following generalizes Theorem 6.5.2
of~\cite{BerstelPerrinReutenauer2009}. 

\begin{theorem}
  For any recurrent set $F$ and any integer $d\ge 1$ there is a finite
  number of finite $F$-maximal bifix codes $X\subset F$ of $F$-degree $d$.
\end{theorem}

\begin{proof}
  The only $F$-maximal bifix code of $F$-degree $1$ is $F\cap A$.  Arguing
  by induction on $d$, assume that there are only finitely many finite
  $F$-maximal bifix codes of $F$-degree $d$.  Each finite $F$-maximal
  bifix code $X\subset F$ of $F$-degree $d+1$ is determined by its kernel
  which is a subset of its derived code $X'$. Since $X'$ is a finite
  $F$-maximal bifix code of $F$-degree $d$, there are only a finite number
  of kernels and we are done.
\end{proof}

\begin{example}
  Let $A=\{a,b\}$ and let $F$ be the set of words without factor $bb$.
  There are two finite $F$-maximal bifix codes of $F$-degree $2$,
  namely the code $\{aa,ab,ba\}$ with empty kernel and the code
  $\{aa,aba,b\}$ with kernel $b$. The code of $F$-degree $2$ with kernel
  $a$ is $a\cup ba^+b$, and thus is infinite.
\end{example}

The following result shows that the case of a uniformly recurrent set
contrasts with the case $F=A^*$ since in $A^*$, as soon as
$\Card(A)\ge 2$, there exist infinite maximal bifix codes of degree
$2$ and thus of all degrees $d\ge 2$; see
e.g. \cite[Example~6.4.7]{BerstelPerrinReutenauer2009} for degree~$2$
and \cite[Theorem~6.4.6]{BerstelPerrinReutenauer2009} for the general
case.


\begin{theorem}\label{theoremCompletion}
  Let $F$ be a uniformly recurrent set.  Any $F$-thin bifix code
  $X\subset F$ is finite.  Any finite bifix code is contained in a
  finite $F$-maximal bifix code.
\end{theorem}

\begin{proof}
  Let $X\subset F$ be an $F$-thin bifix code.  Since $X$ is $F$-thin,
  there exists a word $w\in F\setminus I(X)$. Since $F$ is uniformly
  recurrent there is an integer $r$ such that $w$ is factor of every
  word in $F_r=F\cap A^r$. Assume $F_k\cap X\ne\emptyset$ for some $k
  \geq r+2$, and let $x\in F_k\cap X$. Set $x=pqs$, with $q \in
  F_r\cap I(X)$, and $p$, $s$ nonempty. Then $w$ is factor of $q$,
  hence $w$ is in $I(X)$, a contradiction. We deduce that each $x$ in
  $X$ has length at most $r+1$.  Thus $X$ is finite.

Let $X\subset F$ be a finite bifix code which is not $F$-maximal.  Let
$d=\max\{\pars_X(x)\mid x\in X\}$. By Theorem~\ref{theorem643}, $X$ is
the kernel of an $F$-thin and $F$-maximal bifix code $Z\subset F$ of
$F$-degree $d+1$. By the previous argument, $Z$ is finite.
\end{proof}

By Theorem 6.6.1 of~\cite{BerstelPerrinReutenauer2009}, any rational
bifix code is contained in a maximal rational bifix code. We have seen
that the situation is simpler for bifix codes in uniformly recurrent
sets.

\begin{example}
  Let $F$ be the Fibonacci set. Let
  $X=\{a,bab\}$. Then $X$ is contained in the bifix code
  $\{a,bab,baab\}$ which has $F$-degree $2$ (see
  Figure~\ref{figureBifix2}). It is also the kernel of
  $\{a,baabaab,baababaab,bab\}$ which is a bifix code of $F$-degree $3$
  (see Table~\ref{tableBifix}).
\end{example}

The following is a generalization of Proposition~6.2.10
in~\cite{BerstelPerrinReutenauer2009}.  The equality $d(Y)=d(X)$ is
stated as a comment following Proposition~6.3.9 in~\cite[page
243]{BerstelPerrinReutenauer2009}, in a more general framework.

\begin{proposition}\label{propositionInternalTransformation}
  Let $F$ be a recurrent set, let $X\subset F$ be a finite $F$-maximal
  bifix code and let $w$ be a nonempty word in $F$. Let $G=Xw^{-1}$, and $D=
  w^{-1}X$. If
  \begin{displaymath}
    G\ne\emptyset\,,\quad D\ne\emptyset\,,\text{ and }\quad Gw\cap
    wD=\emptyset\,,  
  \end{displaymath}
  then the  set
  \begin{equation}
    Y=(X\cup w\cup (GwD\cap F))\setminus (Gw\cup
    wD)\label{eqInternalTransformation} 
  \end{equation}
  is a finite $F$-maximal bifix code with the same $F$-degree as $X$.
\end{proposition}

We use in the proof the following proposition which is an extension of
Corollary~3.4.7 of~\cite{BerstelPerrinReutenauer2009}.

\begin{proposition}\label{propositionCA347}
  Let $F$ be a recurrent set and let $X\subset F$ be a $F$-maximal
  prefix code.  Let $X=X_1\cup X_2$ be a partition of $X$ into two
  prefix codes and let $Y$ be a finite prefix code such that $Y\cap
  x^{-1}F$ is $x^{-1}F$-maximal for all $x\in X_2$. Then the set
  $Z=X_1\cup(X_2Y\cap F)$ is an $F$-maximal prefix code.
\end{proposition}

\begin{proof}
  We first prove that $Z$ is a prefix code.  Let $z$ and $z'$ be
  distinct words in $Z$.  We show that they are not prefix-comparable.
  Since $X_1$ is a prefix code, this holds if both words are in $X_1$.
Assume next that $z\in X_2Y$. Then
$z=xy$ with $x\in X_2$ and $y\in Y$.  

If $z'$ is in $X_1$, then $z'$
and $x$ are not prefix-comparable because they are distinct since
$z'\in X_1$ and $x\in X_2$, and so $z$ and $z'$ are not
prefix-comparable.

If $z'\in X_2Y$, set $z'=x'y'$ with $x'\in X_2$ and $y'\in
Y$. Either $x$ and $x'$ are not prefix-comparable, and then so are $z$
and $z'$, or $x=x'$. In the latter case, $y$ and $y'$ are not
prefix-comparable because $Y$ is a prefix code, and again $z$ and $z'$
are not prefix-comparable.
Thus $Z$ is a prefix code.

Let us show that $Z$ is $F$-maximal. Let $u\in F$. Since $X$ is an
$F$-maximal
prefix code, there is an $x\in X$ which is prefix-comparable with $u$.
If $x$ is in $X_1$, then $x\in Z$ and thus $u$ is prefix-comparable
with a word of $Z$. Otherwise, we have $x\in X_2$.

Suppose first that $u$ is a prefix of $x$.  Since $Y$
is a finite $x^{-1}F$-maximal prefix code, it is not empty and 
$u$ is a prefix of  $xv$ for every $v\in Y\cap x^{-1}F$.

Suppose next that $u=xv$ for some word $v$. Since $v$ is in $x^{-1}F$
and since $Y\cap x^{-1}F$ is an $x^{-1}F$-maximal prefix code, the
word $v$ is prefix-comparable with some $y\in Y\cap x^{-1}F$.  Thus
$u$ is prefix-comparable with $xy\in Z$.
\end{proof}

\begin{proofof}{of Proposition~\ref{propositionInternalTransformation}}
  The condition $G\ne\emptyset$ ( resp. $D\ne\emptyset$) means that
  $w$ is a suffix of $X$ (resp. a prefix of $X$). The condition
  $Gw\cap wD=\emptyset$ implies that $w$ is not in $X$.

  By Theorem~\ref{theoremEquivMax}, the set $X$ is an $F$-maximal
  prefix code.  By Proposition~\ref{propositionOperation1}, the set
  $Y_1=(X\cup w)\setminus wD$ is an $F$-maximal prefix code. Next, we
  have
\begin{displaymath}
Y=(Y_1\setminus Gw)\cup (GwD\cap F).
\end{displaymath}

We show that $Y$ is an $F$-maximal prefix code, by applying
Proposition~\ref{propositionCA347}. Indeed, consider the partition
$Y_1=X_1\cup X_2$ with $X_1= Y_1\setminus Gw$ and $X_2=Gw$.  Then
$Y=X_1\cup (X_2D\cap F)$.  Clearly $D$ is a finite $w^{-1}F$-maximal
prefix code. Since $(gw)^{-1}F$ is a subset of $w^{-1}F$ for all
$g\in G$, the set $D\cap (gw)^{-1}F$ is a finite $(gw)^{-1}F$-maximal
prefix code for all $g\in G$ by
Proposition~\ref{propositionFinitePrefixCodes}, So the claim follows
from by Proposition~\ref{propositionCA347}. This proves that $Y$ is an
$F$-maximal prefix code.  Since $Y$ it is also a suffix code, it
follows that $Y$ is an $F$-maximal bifix code by
Theorem~\ref{theoremEquivMax}. 

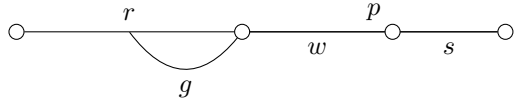
\begin{figure}[hbt]
  \centering
  \gasset{Nadjust=wh,AHnb=0,Nfill=y,fillgray=1,ELdist=1.5}
  \begin{picture}(65,10)(0,-5)
    \node(0)(0,0){}
    \node[Nframe=n,Nadjustdist=0](g)(15,0){}
    \node(r)(30,0){}\drawedge(0,r){$r$}
    \node(s)(65,0){}\drawedge(r,s){$p$}
    \node(w)(50,0){}
    \drawedge[ELside=r](r,w){$w$}\drawedge[ELside=r](w,s){$s$}
    \drawedge[curvedepth=-5,ELside=r](g,r){$g$}
      \end{picture}
  \caption{Construction of $\varphi(p)$ (second case).}\label{figInternal1}
\end{figure}
To show that $X$ and $Y$ have the same degree, consider a word $u\in
F$ which is not an internal factor of $X$ nor $Y$. Such a word exists
since $X$ and $Y$ are finite. Let $P$ (resp. $Q$) be the set of proper
prefixes of the words of $X$ (resp. $Y$).  We define a bijection
$\varphi$ between the set $P(u)$ of suffixes of $u$ which are in $P$
and the set $Q(u)$ of suffixes of $u$ which are in $Q$. This will
imply that $d_F(X)=d_F(Y)$ by the dual of
Proposition~\ref{propInterpretations}.

Let $p\in P$ be a suffix of $u$ and set $u=rp$. If $w$ is not a prefix
of $p$, then $p$ is in $Q$. Otherwise, set $p=ws$. Since the words in
$P$ starting with $w$ are all prefixes of $wD$, the word $s$ is a
proper prefix of $D$.  Since $G$ is an $Fw^{-1}$-maximal suffix code,
$r$ is suffix-comparable with a word of $G$. If $r$ is a proper suffix
of $G$, then $urws$ is an internal factor of $GwD$, a
contradiction. Thus $r$ has a suffix $g\in G$. This suffix is unique
because $G$ is a suffix code.  Since $gp=gws$, the word $gp$ is a
proper prefix of $GwD$, and thus a proper prefix of $Y$. Thus $gp\in
Q(u)$.  We set (see Figure~\ref{figInternal1})
\begin{displaymath}
  \varphi(p)=
  \begin{cases}
    p&\text{if $p\notin wA^*$,}\\
    gp&\text{if $p\in wA^*$ and $g\in G$ is the suffix of $r$ in $G$.}
  \end{cases}
\end{displaymath}
Thus $\varphi$ maps $P(u)$ into $Q(u)$.  We show that it is
injective. Suppose that $\varphi(p)=\varphi(p')$ for some $p,p'\in
P(u)$. Assume that $\varphi(p)=gp$ and $\varphi(p')=g'p'$ with
$g,g'\in G$. Since $p$ and $p'$ start with $w$, the word $gp=g'p'$
starts with the words $gw$ and $g'w$ which are in $X$. This shows that
$g=g'$ and thus $p=p'$. Assume next that $p=g'p'$ with $g'\in G$ and
$p'\in wA^*$. But then $g'w$ is a prefix of $p$, with $p$ in $P$ and $g'w$
 in $X$, a contradiction.

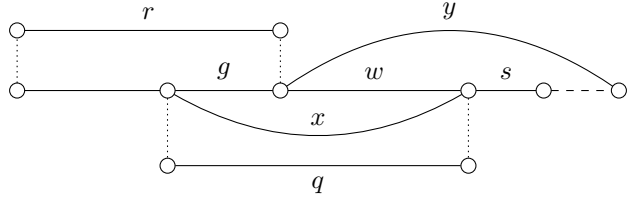
\begin{figure}[hbt]
  \centering
  \gasset{Nadjust=wh,AHnb=0,Nfill=y,fillgray=1,ELdist=1.5}
  \begin{picture}(80,22)(0,-10)
    \node(0)(0,0){}
    \node(g)(20,0){}
    \node(r)(35,0){}
    \node(s)(70,0){}    
    \node(q)(60,0){}
    \node(y)(80,0){}
    \node(qleft)(20,-10){}
    \node(qright)(60,-10){}\drawedge[ELside=r](qleft,qright){$q$}
    \drawedge(0,g){}\drawedge(g,r){$g$}\drawedge(r,q){$w$}
    \drawedge(q,s){$s$}
    \drawedge[curvedepth=-6](g,q){$x$}
    \drawedge[curvedepth=8](r,y){$y$}
    \node(0')(0,8){}
    \node(r')(35,8){}
    \drawedge(0',r'){$r$}
    \drawedge[dash={1 1}0](s,y){}
    \drawedge[dash={0.2 0.5}0](0,0'){}
    \drawedge[dash={0.2 0.5}0](r,r'){}
 \drawedge[dash={0.2 0.5}0](g,qleft){} 
 \drawedge[dash={0.2 0.5}0](q,qright){}
      \end{picture}
  \caption{Reconstruction of the factorization.}\label{figInternal2}
\end{figure}


To show that $\varphi$ is surjective, consider $q\in Q(u)$. Assume first
that $q$ has a prefix $x$ in $X$ (see Figure~\ref{figInternal2}). 
By Equation~\eqref{eqInternalTransformation}, one has
$x=gw$ and 
 $q=gws$ for some $g\in G$ and $s$ a proper prefix of the word $d$ in
$D$. Thus $ws$ is a proper prefix of $wD\subset X$, and consequently
$ws$ is a proper prefix of $X$. Since $ws$ is a suffix of $q$, it is a
suffix of $u$.  Thus $ws\in P(u)$. Set $u=rws$. Then $g$ is a suffix
of $r$. Moreover $ws\in wA^*$. Consequently $\varphi(ws)=q$.

Finally, if $q$ has no prefix in $X$, then $q$ is a proper prefix of
$X$. Moreover, since $q$ is a prefix of $Y$, either $q$ is a proper
prefix of $w$ or $q$ is not a prefix of $wD$. In both cases, $w$ is
not a prefix of $q$ and therefore $\varphi(q)=q$.  Thus $\varphi$ is
surjective.
\end{proofof}

The set $Y$ defined by Equation~\eqref{eqInternalTransformation}, is
said to be obtained from $X$ by \emph{internal
  transformation}\index{internal transformation} (with respect to
$w$).

\begin{example}\label{exampleInternalFibo1}
  Let $F$ be the Fibonacci set. The set $X=\{aa,ab,ba\}$ is an
  $F$-maximal bifix code of $F$-degree $2$. Then $Y=\{aa,aba,b\}$ is a
  bifix code of $F$-degree $2$ which is obtained from $X$ by internal
  transformation with respect to $w=b$. Indeed, here $G=D=\{a\}$,
  $Gw=\{ab\}$, $wD=\{ba\}$ and $GwD=\{aba\}$.
\end{example}

The following theorem is due to C\'esari. It is Theorem~6.5.4
in~\cite{BerstelPerrinReutenauer2009}.

\begin{theorem}\label{theoremCesari}
  For any finite maximal bifix code $X$ over $A$ of degree $d$, there
  is a sequence of internal transformations which, starting from the
  code $A^d$, gives the code~$X$.
\end{theorem}

Theorem~\ref{theoremCesari} has been generalized to finite $F$-maximal
bifix codes when $F$ is the set of paths in a strongly connected graph
(see~\cite{DeFelice1988}).  It is not true in any recurrent, or even
uniformly recurrent set, as shown by the following example.

\begin{example}
  Let $F$ be the Fibonacci set. The set $X=\{a,bab,baab\}$ is a finite
  bifix code of $F$-degree $2$. It cannot be obtained by a sequence of
  internal transformations from the code $A^2\cap
  F=\{aa,ab,ba\}$. Indeed, the only internal transformation which can
  be realized is with respect to $w=b$.  The result is $\{aa,aba,b\}$
  by Example~\ref{exampleInternalFibo1}.  Next, no internal
  transformation can be realized from this code. See also
  Figure~\ref{figureBifix2}. 
\end{example}

A more general form of internal transformation is described
in~\cite{BerstelPerrinReutenauer2009} in Proposition~6.2.8. We do not
know whether its adaptation to finite $F$-maximal bifix codes allows
one to obtain all finite $F$-maximal bifix codes of $F$-degree $d$
starting with the code $A^d\cap F$.

\section{Bifix codes in Sturmian sets}

In this section, we study bifix codes in Sturmian sets. This time, the
situation is completely specific. First of all, as we have already
seen, any $F$-thin bifix code included in a uniformly recurrent set
$F$ is finite (Theorem~\ref{theoremCompletion}).  Next, in a
Sturmian set $F$, any bifix code of finite $F$-degree $d$ on $k$
letters has $(k-1)d+1$ elements (Theorem~\ref{theoremBifixd+1}).  
Since $A^d$ is a bifix code of degree $d$, this
generalizes the fact that $\Card(F\cap A^d)=(k-1)d+1$ for all $d\ge
1$.

Additionally, if an infinite word $x$ is $X$-stable, that is if , for
some thin maximal bifix code $X$, one has $d_{F(y)}(X)=d_{F(x)}(X)$
for all suffixes $y$ of $x$, then the inequality $\Card(X\cap F(x))\le
d_{F(x)}(X)$ implies that $x$ is ultimately periodic
(Theorem~\ref{theoremPeriod}).

\subsection{Sturmian sets}

Let $F$ be a factorial set on the alphabet $A$. Recall that a word $w$
is strict right-special if $wA\subset F$. It is strict left-special if
$Aw\subset F$. A suffix of a (strict) right-special word is (strict)
right-special, a prefix of a (strict) left-special word is (strict)
left-special.

A set of words $F$ is called \emph{Sturmian}\index{Sturmian set} if it
is the set of factors of a strict episturmian word.  By
Proposition~\ref{propositionSturmianMinimal} a Sturmian set $F$ is
uniformly recurrent. Moreover, every right-special (left-special) word
in $F$ is strict.

The following statement gives a direct definition of Sturmian sets.

\begin{proposition}
  A set $F$ is Sturmian if and only if it is uniformly recurrent and
  \begin{enumerate}
  \item[\upshape{(i)}] it is closed under reversal,
  \item[\upshape{(ii)}] for each $n$, there is exactly one right-special
    word in $F$ of length $n$, and this right-special
    word is strict.
  \end{enumerate}
\end{proposition}
\begin{proof}
  If $F=F(x)$ for some strict episturmian word, then the conclusions
  of the proposition hold. 

  Conversely, assume that $F$ has the required properties. For each
  $n$, the reversal of the strict right-special word of length $n$ is
  a strict left-special word. Since all these left-special words are
  prefixes one of the other, there is an infinite word $x$ that such
  that all its prefixes are these strict left-special words. Clearly,
  $F(x)\subset F$. To show that $x$ is strict episturmian, we verify
  that $F(x)$ is closed under reversal. Let $u\in F(x)$. Then $u\in
  F$. Since $F$ is uniformly recurrent, there is an integer $m$ such
  that $u$ is a factor of the right-special word $w$ of length
  $m$. Consequently the reversal $\widetilde{u}$ of $u$ is a factor of
  the left-special word $\widetilde{w}$ of length $m$, and therefore
  is in $F(x)$.

  To prove that $F\subset F(x)$, let $u\in F$. Since $F$ is uniformly
  recurrent, there is an integer $m$ such that $u$ is a factor of the
  left-special word $w$ of length $m$. Since $w$ is a prefix of $x$,
  this shows that $u\in F(x)$.
\end{proof}

The following statement is a direct consequence of the previous proof.

\begin{proposition}\label{leftInfinite}
  Let $F$ be a Sturmian set of words.  There is a unique strict standard
  episturmian infinite word $s$ such that $F=F(s)$.
\end{proposition}

As a consequence of Proposition~\ref{leftInfinite}, for every
left-special word $w$ of a Sturmian set $F$, exactly one of the words
$wa$, for $a\in A$, is left-special in $F$.  Symmetrically, for every
right-special word $w$ in $F$, exactly one of the words $aw$ for $a\in A$ is
right-special in $F$. More generally, for every $n\ge 1$ there is exactly one
word $u$ of length $n$ such that $uw$ is a right-special word in $F$.

\begin{proposition}\label{propositionPrefixSpecial}
  Any word in a Sturmian set $F$ is a prefix of some right-special
  word in $F$.
\end{proposition}

\begin{proof} Let indeed $u\in F$.  Since $F$ is uniformly recurrent,
  there is an integer $n$ such that $u$ is a factor of any word in $F$
  of length $n$.  Let $w$ be the right-special word of length
  $n$. Then $u$ is a factor of $w$, thus $w=pus$ for some words
  $p,s$. Since $w$ is right-special, its suffix $us$ is also
  right-special.  Thus $u$ is a prefix of a right-special word.
\end{proof}

The following example shows that for a Sturmian set $F$, there exists
bifix codes $X\subset F$ which are not $F$-thin (we have seen such an
example for a uniformly recurrent but not Sturmian set in
Example~\ref{exampleDenseBifixCode}).

\begin{example}
Let $F$ be a Sturmian set. Consider the
following sequence $(x_n)_{n\ge 1}$ of words of $F$. Set $x_1=a$, for
some $a\in A$.

Suppose inductively that $x_1,\ldots,x_n$ have been defined in such a
way that $X_n=\{x_1,x_2,\ldots,x_n\}$ is bifix and not $F$-maximal bifix.
Define $x_{n+1}$ as follows.  By Theorem~\ref{theoremEquivMax}, $X_n$
is not right $F$-complete, thus there is a word $u$ in $F$ which is
incomparable for the prefix order with the words of $X_n$.  By
Proposition~\ref{propositionPrefixSpecial}, the word $u$ is a prefix
of a right special word $v$ in $F$. Symmetrically, since $X_n$ is not
an $F$-maximal bifix code, there is a word $w\in F$ which is
incomparable with the words of $X_n$ for the suffix order. Since $F$
is recurrent, there is a word $t$ such that $vatw\in F$.  Then we
choose $x_{n+1}=vatw$.

The set $X_{n+1}=X_n\cup x_{n+1}$ is a bifix code since $x_{n+1}$ is
incomparable with the words of $X_n$ for the prefix and for the suffix
order. It is not an $F$-maximal prefix code since $vb$, for all
letters $b\ne a$, is incomparable for the prefix order with the words
of $X_{n+1}$: indeed, its prefix $u$ is incomparable for the prefix
order with all words in $X_n$ and $vb$ is incomparable with
$x_{n+1}$. Since it is finite, it is not an $F$-maximal bifix code by
Theorem~\ref{theoremEquivMax}.  The infinite set
$X=\{x_1,x_2,\ldots\}$ is a bifix code included in $F$ and it is not
$F$-thin by Theorem~\ref{theoremCompletion}.
\end{example}

\begin{proposition}\label{propPrefixSpecial}
  Let $F$ be a Sturmian set and let $X\subset F$ be a prefix code.
  Then $X$ contains at most one left-special word. If $X$ is a finite
  $F$-maximal prefix code, it contains exactly one left-special word.
\end{proposition}

\begin{proof} Assume on the contrary that $x,y\in X$ are two
  left-special words. We may assume that $|x|<|y|$.  Let $x'$ be the
  prefix of $y$ of length $|x|$. Then $x'$ is left-special and thus
  $x,x'$ are two left-special words of the same length.  This implies
  that $x=x'$. Thus $x$ is a prefix of $y$. Since $X$ is prefix, this
  implies $x=y$.

  Assume now that $X$ is a finite $F$-maximal prefix code. Let $n$ be
  the maximal length of the words in $X$. Let $u\in F$ be the
  left-special word of length $n$. Since $XA^*$ is right $F$-dense,
  there is a prefix $x$ of $u$ which is in $X$. Thus $x$ is a
  left-special element of $X$. It is unique by the previous statement.
\end{proof}

A dual of Proposition~\ref{propPrefixSpecial} holds for suffix codes
and right-special words.
\subsection{Cardinality}

The following result shows that Theorem~\ref{theoremCompletion}
can be made much more precise for Sturmian sets.

\begin{theorem}\label{theoremBifixd+1}
  Let $F$ be a Sturmian set on an alphabet with $k$ letters. For any
  finite $F$-maximal bifix code $X\subset F$, one has
  $\Card(X)=(k-1)d_F(X)+1$.
\end{theorem}

The following corollary is strong generalization of a result related
to Sturmian words. 

\begin{corollary}\label{corSturm}
  Let $x$ be a Sturmian word over $A=\{a,b\}$, and let $X\subset
  A^+$ be a finite maximal bifix code of degree $d$. Then $\Card(X\cap
  F(x))=d+1$. 
\end{corollary}

Indeed, since $A^d$ is a finite maximal bifix code of degree $d$, this
corollary (re)proves that any Sturmian word $x$ has $d+1$ factors of
length $d$, and it extends this to arbitrary finite maximal bifix code
of degree $d$. A similar extension holds for strict episturmian
words.\medskip

\begin{proofof}{of Corollary \ref{corSturm}}
  Set $F=F(x)$. In view of Theorem~\ref{proposition1}, one has
  $d=d_F(X\cap F)$. Consequently, by the formula of
  Theorem~\ref{theoremBifixd+1}, $\Card(X\cap F)=d_F(X)+1=d+1$.
\end{proofof}

The proof of Theorem~\ref{theoremBifixd+1} uses two lemmas.

\begin{lemma}\label{lemmadSpecial}
  Let $F$ be a Sturmian set. Let $X\subset F$ be a finite bifix code
  of finite $F$-degree $d$ and let $P$ be the set of proper prefixes
  of $X$.  There exists a right-special word $u\in F$
  such that $\pars_X(u)=d$. The $d$ suffixes of $u$ which are in $P$ are
  the right-special words contained in $P$.
\end{lemma}

\begin{proof}
  Let $n\ge 1$ be larger than the length of the words of $X$.  By
  definition, there is a right-special word $u$ of length $n$. Then
  $u$ is not a factor of a word of $X$. By Theorem~\ref{theoremDegree}
  it implies that $\pars_X(u)=d_F(X)$.

  By the dual of Proposition~\ref{propInterpretations}, the word $u$
  has $d_F(X)$ suffixes which are in $P$. They are all right-special
  words. Furthermore, any right-special word $p$ contained in $P$ is a
  suffix of $u$. Indeed, the suffix of $u$ of the same length than $p$
  is the unique right-special word of this length.
\end{proof}

The next lemma is a well-known property of trees translated into the
language of prefix codes.  Let $X$ be a prefix code or the set $\{1\}$
and let $P$ be the set of proper prefixes of $X$. For $p\in P$, let
$d(p)=\Card\{a\in A\mid pa\in P\cup X\}$.

\begin{lemma}\label{lemmaTree}
  Let $A$ be an alphabet with $k$ letters.  Let $X\subset A^*$ be a
  finite prefix code or the set $\{1\}$ and let $P$ be the set of
  proper prefixes of the words of $X$. Assume that for all $p\in P$,
  $d(p)=k$ or $1$. Let $Q_X=\{p\in P\mid d(p)=k\}$. Then,
  $\Card(X)=(k-1)\Card(Q_X)+1$.
\end{lemma}

\begin{proof}
  Let us prove the property by induction on the maximal length $n$ of
  the words in $X$. The property is true for $n=0$ since in this case
  $X=\{1\}$ and $P=Q_X=\emptyset$. Assume $n \geq 1$.  If $1 \notin
  Q_X$, then all words of $X$ begin with the same letter $a$.  We have
  then $X=aY$, $Y$ is a prefix code or the set $\{1\}$ and $\Card(Q_Y)
  = \Card(Q_X)$. Hence, by induction hypothesis $\Card(X) = \Card(Y) =
  (k-1)\Card(Q_Y)+1=(k-1)\Card(Q_X)+1$.  Otherwise, $X = \cup_{a\in
    A}aX_a$. Set $t_a = \Card(Q_{X_a})$.  We have $\sum_{a\in A}t_a=
  \Card(Q_X) -1$. By induction hypothesis, $\Card(X_a) = (k-1)t_a+1$.
  Therefore, $\Card(X) = \sum_{a\in A}\Card(X_a) = \sum_{a\in
    A}(k-1)t_a+ k = (k-1)\Card(Q_X)+1$.
\end{proof}

\begin{proofof}{of Theorem \ref{theoremBifixd+1}}
  Let $P$ be the set of proper prefixes of $X$.  An element $p$ of $P$
  satisfies $pA\subset P\cup X$ if and only it is right-special.  Thus
  the conclusion follows directly by Lemmas~\ref{lemmadSpecial} and
  \ref{lemmaTree}.
\end{proofof}

\begin{table}[hbt]
\begin{displaymath}
\begin{array}{|l|l|l|}
\hline \text{code}      &\text{kernel }          &\text{derived  code}\\ \hline
aab,aba,baa,bab         &\emptyset               &aa,ab,ba         \\ \cline{1-2}
aa,aba,baab,bab         &aa                      & \\ \cline{1-2}
aaba,ab,baa,baba         &ab                      &\\    \cline{1-2}
aab,abaa,abab,ba        &ba                      &       \\   \cline{1-2}
aa,ab,baaba,baba         &aa,ab                   &\\   \cline{1-2}
aa,abaab,abab,ba        &aa,ba                   &\\  \cline{1-2}
aabaa,aababaa,ab,ba     &ab,ba                   &\\ \hline
a,baabaab,baabab,babaab         &a                      &a,baab,bab\\ \cline{1-2}
a,baab,babaabaabab,babaabab     &a,baab                  &\\ \cline{1-2}
a,baabaab,baababaab,bab  &a,bab                   &\\ \hline

aaba,abaa,ababa,b       &b                       &aa,aba,b \\\cline{1-2}
aa,abaaba,ababa,b     &aa,b                    &           \\  \cline{1-2}
aabaa,aababa,aba,b      &aba,b                   &            \\
\hline
\end{array}
\end{displaymath}
\caption{The $13$ $F$-maximal bifix codes of $F$-degree $3$ 
 in the Fibonacci set~$F$.}\label{tableBifix}
\end{table}

\begin{example}\label{exampleallDegree3}
  Let $F$ be the Fibonacci set.  We have seen in
  Example~\ref{exampleFiboDegree2} that there are $3$ $F$-maximal
  bifix codes of $F$-degree 2. It appears that there are $13$
  $F$-maximal bifix codes of degree $3$ listed in
  Table~\ref{tableBifix}. These codes are determined by their derived
  $F$-maximal bifix codes of $F$-degree 2, and by the choice of the
  kernel. The construction of the code can be done by
  Theorem~\ref{theoremkerneldegree}.  By
  Theorem~\ref{theoremBifixd+1}, all these codes have $4$ elements.
\end{example}

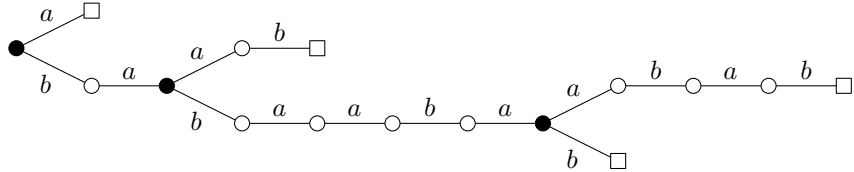
\begin{figure}[hbt]
  \gasset{Nadjust=wh,AHnb=0}
  \centering
  \begin{picture}(110,20)(0,-10)
    \node[fillgray=0](1)(0,5){}\node[Nmr=0](a)(10,10){}\node(b)(10,0){}
    \node[fillgray=0](ba)(20,0){}\node(baa)(30,5){}\node[Nmr=0](baab)(40,5){}
    \node(bab)(30,-5){}
    \node(baba)(40,-5){}
    \node(babaa)(50,-5){}
    \node(babaab)(60,-5){}
    \node[fillgray=0](babaaba)(70,-5){}
    \node(babaabaa)(80,0){}
    \node(babaabaab)(90,0){}
    \node(babaabaaba)(100,0){}
    \node[Nmr=0](babaabaabab)(110,0){}
    \node[Nmr=0](babaabab)(80,-10){}
    \drawedge(1,a){$a$}\drawedge[ELside=r](1,b){$b$}
    \drawedge(b,ba){$a$}\drawedge(ba,baa){$a$}
    \drawedge[ELside=r](ba,bab){$b$}\drawedge(baa,baab){$b$}
    \drawedge(bab,baba){$a$}
    \drawedge(baba,babaa){$a$}
    \drawedge(babaa,babaab){$b$}
    \drawedge(babaab,babaaba){$a$}
    \drawedge(babaaba,babaabaa){$a$}
    \drawedge[ELside=r](babaaba,babaabab){$b$}
    \drawedge(babaabaa,babaabaab){$b$}
    \drawedge(babaabaab,babaabaaba){$a$}
    \drawedge(babaabaaba,babaabaabab){$b$}%
  \end{picture}
  \caption{The $F$-maximal bifix code of $F$-degree 3 with kernel $\{a,baab\}$.}
  \label{figureExd+1}
\end{figure}

\begin{example}\label{exampleDegree3}
  We may illustrate the proof of Theorem~\ref{theoremBifixd+1} on the
  code $X=\{a,baab,babaabaabab,babaabab\}$ (see
  Table~\ref{tableBifix}). According to Lemma~\ref{lemmadSpecial}, the
  right-special word $ababaaba$ (which is the reversal of the prefix
  $abaababa$ of the Fibonacci word) has exactly three suffixes which
  are proper prefixes of words of $X$, namely $1$, $ba$ and $babaaba$
  (these are the ``fork nodes'', that is the nodes with two childred,
  indicated in black on Figure~\ref{figureExd+1}). This implies, by
  Lemma~\ref{lemmaTree}, that $X$ has four elements.
\end{example}

The following example shows that Theorem~\ref{theoremBifixd+1} is
not true for the set of factors of an episturmian word which is not
strict.

\begin{example}\label{exampleNotStrict}
  Set $A=\{a,b,c\}$.  Let $y$ be the Fibonacci word and let
  $x=\psi_c(y)$ be the infinite word of
  Example~\ref{exampleNonStrict}. It is an episturmian word which is
  not strict. Set $F=F(x)$.  Let $\psi:A^*\rightarrow G$ be the
  morphism from $A^*$ onto the group $G=(\Z/2\Z)^3$ defined by
  $\psi(a)=(1,0,0)$, $\psi(b)=(0,1,0)$ and $\psi(c)=(0,0,1)$. Let $Z$
  be the group code such that $Z^*=\psi^{-1}(0,0,0)$.  Since $G$ has
  $8$ elements, the degree of $Z$ is $8$ (see Proposition~\ref{propositionGroupAutomaton} below). The bifix code $X=Z\cap F$
  has $10$ elements obtained by inserting $c$ in two possible ways in
  the $5$ words of the bifix code $Z\cap F(y)$. The latter has degree~4
  by Theorem~\ref{theoremBifixd+1}.
 The bifix code $X=Z\cap F$ is given in
  Figure~\ref{figureExNonStrict}. The numbering of the nodes is for
  later use, in Example~\ref{exNonStrictGroup}.

\begin{figure}[hbt]
\begin{picture}(108,55)(0,-10)
\gasset{Nadjust=wh,AHnb=0}
\node(1)(0,20){$1$}
\node(a)(9,30){$2$}\node(b)(9,20){$3$}\node(c)(9,10){$4$}
\node(ac)(18,30){$5$}\node(bc)(18,20){$6$}\node(ca)(18,10){$7$}\node(cb)(18,-1){$8$}
\node(aca)(27,40){$9$}\node(acb)(27,30){$10$}\node(bca)(27,20){$11$}
\node(cac)(27,10){$12$}\node(cbc)(27,-1){$13$}
\node[Nmr=0](acac)(36,40){$1$}\node(acbc)(36,30){$14$}\node(bcac)(36,20){$15$}
\node[Nmr=0](caca)(36,14){$1$}\node(cacb)(36,7){$16$}
\node(cbca)(36,-1){$17$}
\node(acbca)(45,30){$18$}\node(bcaca)(45,22){$19$}\node(bcacb)(45,15){$20$}
\node(cacbc)(45,7){$21$}\node(cbcac)(45,-1){$22$}
\node(acbcac)(54,30){$23$}\node(bcacac)(54,22){$24$}\node(bcacbc)(54,15){$25$}
\node(cacbca)(54,7){$26$}\node(cbcaca)(54,-1){$27$}\node(cbcacb)(54,-10){$28$}
\node(acbcaca)(63,40){$29$}\node(acbcacb)(63,30){$9$}
\node(bcacacb)(63,22){$9$}\node(bcacbca)(63,15){$9$}
\node(cacbcac)(63,7){$30$}\node(cbcacac)(63,-1){$31$}
\node(cbcacbc)(63,-10){$32$}
\node(acbcacac)(72,40){$33$}\node[Nmr=0](acbcacbc)(72,30){$1$}
\node[Nmr=0](bcacacbc)(72,22){$1$}\node[Nmr=0](bcacbcac)(72,15){$1$}
\node(cacbcaca)(72,10){$34$}
\node[Nmr=0](cacbcacb)(72,4){$1$}
\node[Nmr=0](cbcacacb)(72,-1){$1$}\node[Nmr=0](cbcacbca)(72,-10){$1$}
\node(acbcacacb)(81,40){$20$}
\node(cacbcacac)(81,10){$35$}
\node(acbcacacbc)(90,40){$25$}
\node(cacbcacacb)(90,10){$28$}
\node(acbcacacbca)(99,40){$9$}
\node(cacbcacacbc)(99,10){$32$}
\node[Nmr=0](acbcacacbcac)(108,40){$1$}
\node[Nmr=0](cacbcacacbca)(108,10){$1$}

\drawedge(1,a){$a$}\drawedge(1,b){$b$}\drawedge(1,c){$c$}
\drawedge(a,ac){$c$}\drawedge(b,bc){$c$}\drawedge(c,ca){$a$}\drawedge(c,cb){$b$}
\drawedge(ac,aca){$a$}\drawedge(ac,acb){$b$}
\drawedge(bc,bca){$a$}
\drawedge(ca,cac){$c$}\drawedge(cb,cbc){$c$}
\drawedge(aca,acac){$c$}\drawedge(acb,acbc){$c$}
\drawedge(bca,bcac){$c$}
\drawedge(cac,caca){$a$}\drawedge(cac,cacb){$b$}
\drawedge(cbc,cbca){$a$}
\drawedge(acbc,acbca){$a$}
\drawedge(bcac,bcaca){$a$}\drawedge(bcac,bcacb){$b$}
\drawedge(cacb,cacbc){$c$}\drawedge(cbca,cbcac){$c$}
\drawedge(acbca,acbcac){$c$}
\drawedge(bcaca,bcacac){$c$}\drawedge(bcacb,bcacbc){$c$}
\drawedge(acbcac,acbcaca){$a$}\drawedge(acbcac,acbcacb){$b$}
\drawedge(bcacac,bcacacb){$b$}
\drawedge(cacbc,cacbca){$a$}
\drawedge(cbcac,cbcaca){$a$}\drawedge[ELside=r](cbcac,cbcacb){$b$}

\drawedge(acbcaca,acbcacac){$c$}\drawedge(acbcacb,acbcacbc){$c$}
\drawedge(bcacacb,bcacacbc){$c$}\drawedge(bcacbc,bcacbca){$a$}
\drawedge(bcacbca,bcacbcac){$c$}
\drawedge(cacbca,cacbcac){$c$}
\drawedge(cbcaca,cbcacac){$c$}\drawedge(cbcacb,cbcacbc){$c$}
\drawedge(cacbcac,cacbcaca){$a$}\drawedge(cacbcac,cacbcacb){$b$}
\drawedge(cbcacac,cbcacacb){$b$}\drawedge(cbcacbc,cbcacbca){$a$}

\drawedge(acbcacac,acbcacacb){$b$}
\drawedge(cacbcaca,cacbcacac){$c$}

\drawedge(acbcacacb,acbcacacbc){$c$}
\drawedge(cacbcacac,cacbcacacb){$b$}

\drawedge(acbcacacbc,acbcacacbca){$a$}
\drawedge(cacbcacacb,cacbcacacbc){$c$}

\drawedge(acbcacacbca,acbcacacbcac){$c$}
\drawedge(cacbcacacbc,cacbcacacbca){$a$}

\end{picture}
\caption{An $F$-maximal bifix code with $10$ elements. The numbers in
  the vertices are for later use.}\label{figureExNonStrict}
\end{figure}
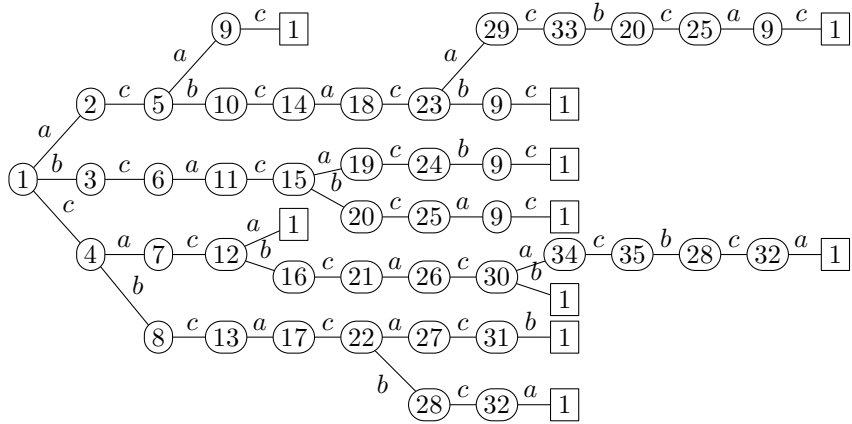
By Theorem~\ref{proposition1}, $X$ is an
$F$-maximal bifix code. Its $F$-degree is $8$. Indeed, the word
$acbcacbc$ has $8$ parses. Thus Theorem~\ref{theoremBifixd+1} is not
true in this case.
\end{example}


As a consequence of Theorem~\ref{theoremBifixd+1}, an internal
transformation does not change the cardinality of a finite $F$-maximal
bifix code for a Sturmian set $F$. Indeed, by
Proposition~\ref{propositionInternalTransformation}, an internal
transformation preserves the $F$-degree.

Actually, if $Y$ is obtained from $X$ by internal transformation with
respect to $w$, we have
\begin{equation}\label{eqInternal}
  Y=(X\cup w\cup (GwD\cap F))\setminus (Gw\cup wD)
\end{equation}
and
\begin{displaymath}
  \Card(Y)=\Card(X)+1+\Card(GwD\cap F)-\Card(G)-\Card(D).
\end{displaymath}

The fact that internal transformations preserve the cardinality can be
proved directly by the following statement. This statement applies to
the internal transformation~\eqref{eqInternal} because $Gw\cup wD$ is
a bifix code, which implies property (i) and $DA^*=w^{-1}XA^*$
(resp. $A^*G=A^*Xw^{-1}$) which implies property (ii) (resp. (iii)).

\begin{proposition}\label{propositionChristophe}
  Let $F$ be Sturmian set, let $w\in F$ be a nonempty word and let
  $D,G$ be finite sets such that
  \begin{enumerate}
  \item[\upshape{(i)}] any word $u$ has at most one factorization $u=gwd$
    with $g\in G$ and $d\in D$,
  \item[\upshape{(ii)}] $wD$ is a prefix code contained in $F$ and $DA^*$ is
    right $w^{-1}F$-dense,
  \item[\upshape{(iii)}] $Gw$ is a suffix code contained in $F$ and
    $A^*G$ is left $Fw^{-1}$-dense.
  \end{enumerate}
  Then $\Card(GwD\cap F)=\Card(G)+\Card(D)-1$.
\end{proposition}

\begin{proof}
  Let $V=(1\otimes G)\cup (D\otimes 1)$ be a set made of copies of $G$
  and $D$.  The tensor product notation is used to emphasize that the
  copies of $G$ and $D$ are disjoint.  Let $H=(V,E)$ be the undirected
  graph having $V$ as set of vertices and as edges the pairs
  $\{1\otimes g,d\otimes 1\}$ such that $gwd\in F$ (this graph is
  close to, but slightly different from the incidence graph for $GwD$
  as it will be defined in Section~\ref{sectionPreliminaries}).  We
  have $\Card(V)=\Card(G)+\Card(D)$ and, by condition (i),
  $\Card(E)=\Card(GwD\cap F)$.  We show that the graph $H$ is a
  tree. This implies our conclusion since, in a tree, one has
  $\Card(E)=\Card(V)-1$.

  Let us prove that the graph $H$ is a tree by induction on the sum of
  the lengths of the words of $D$, assuming that the pair $G,D$ satisfies
  conditions~(ii) and~(iii). Assume first that $D=\{1\}$. Since
  $Gw\subset F$, one has $GwD\subset F$. Consequently, $\{1\otimes
  g,1\otimes 1\}\in E$ for any $g\in G$. Thus $H$ is a tree.

  Next, assume that $D\ne\{1\}$. Let $d\in D$ be of maximal
  length. Set $d=d'a$ with $a\in A$.

  Suppose first that $d'A\cap D=\{d\}$. Let $D'=(D\cup d')\setminus
  d$.  Since $DA^*$ is $w^{-1}F$-dense, the word $wd'$ is not
  right-special. Thus for each $g\in G$, we have $gwd'\in F$ if and
  only if $gwd\in F$. This shows that the graph $H$ is isomorphic to
  the graph $H'$ corresponding to the pair $(G,D')$.  The set $D'$
  satisfies condition (ii). By induction
  $H'$ is a tree. Consequently $H$ is a tree.

  Suppose next that $d'A\cap D$ has more than one element. Then $d'$
  is right-special and $d'A \cap D = d'A$. Let $D'=(D\cup d')\setminus d'A$.  Then $D'$
  satisfies condition (ii). Let $H'$ be the graph corresponding to the
  pair $(G,D')$. By induction hypothesis, the graph $H'$ is tree.
  Since $wD\subset F$, $wd'$ is right-special. Let $uwd'$ be a
  right-special word such that $u$ is longer than any word of $G$.
  Since $A^*G$ is $Fw^{-1}$-dense, and since $u\in Fw^{-1}$, $u$ has a
  suffix $g$ in $G$. Thus $gwd'$ is right-special. We have $\{1\otimes
  g,d'a\otimes 1\}\in E$ for all $a\in A$. For any other element
  $g'\in G$ such that $g'wd'\in F$, since $g'wd'$ is not
  right-special, there is exactly one $a'\in A$ such that $g'wd'a'\in
  F$. There is a path between $g$ and every $g'\ne g$, since
  $\{1\otimes g',1\otimes d'a'\}\in E$ for some $a'$ and $\{1\otimes
  g,1\otimes d'a'\}\in E$ for all $a'$ (see
  Figure~\ref{figureGraphs}). Thus the graph $H$ is connected and
  acyclic, and
  therefore is a tree.
\end{proof}

\begin{figure}[hbt]
  \centering\gasset{AHnb=0}
  \begin{picture}(80,30)
    \put(0,0){
      \begin{picture}(30,40)
        \node(g')(0,30){$g'$}
        \node(g)(0,20){$g$}
        \node(g'')(0,10){$g''$}
        \node[Nframe=n](G)(0,0){$G$}
        \node(d')(30,20){$d'$}
        \node[Nframe=n](D')(30,0){$D'$}
        \drawedge(g,d'){}\drawedge(d',g'){}\drawedge(d',g''){}
      \end{picture}
    }
    \put(50,0){
      \begin{picture}(30,40)
        \node(g')(0,30){$g'$}
        \node(g)(0,20){$g$}
        \node(g'')(0,10){$g''$}
        \node[Nframe=n](G)(0,0){$G$}
        \node(d'a)(30,25){$d'a$}\node(d'b)(30,15){$d'b$}
        \drawedge(g,d'a){}\drawedge(g,d'b){}
        \drawedge(d'a,g'){}\drawedge(d'b,g''){}
        \node[Nframe=n](D)(30,0){$D$}
      \end{picture}
    }
  \end{picture}
  \caption{The graphs $H'$ and $H$.}\label{figureGraphs}
\end{figure}
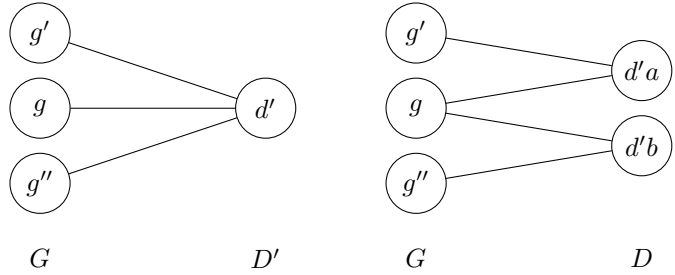

The following example shows  that condition (i) is necessary. 

\begin{example}
  Let $F$ be the Fibonacci set. Let $G=\{ab,aba\}$, $w=a$ and
  $D=\{ab,b\}$.  Then conditions (ii) and (iii) are satisfied but not
  condition (i).  We have $GwD=\{abaab,abab\}$ and thus the conclusion
  of Proposition~\ref{propositionChristophe} is false.
\end{example}

\begin{example}
  Let $F$ be the Fibonacci set and let $X=\{aaba,abaa,abab,\allowbreak
  baab,
  baba\}$ be the set of words of $F$ of length $4$. The internal
  transformation from $X$ relative to the word $w=aba$ gives
  $Y=\{aabaa,aabab,aba,baab,babaa\}$. We have $G=D=\{a,b\}$. The
  codes $X,Y$ and the graph $H$ of the proof of
  Proposition~\ref{propositionChristophe} are represented on
  Figure~\ref{exInternal}.
\end{example}

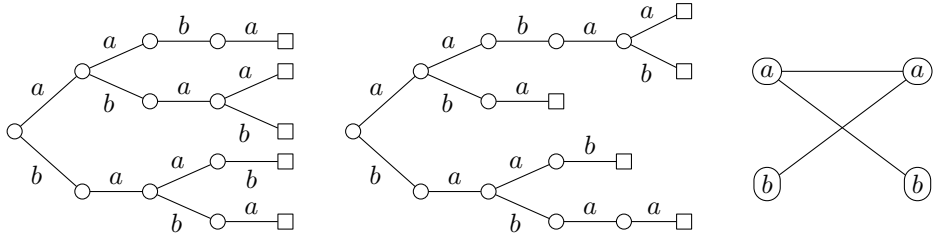
\begin{figure}[hbt]
  \gasset{Nadjust=wh,AHnb=0}
  \begin{picture}(120,30)
    \put(0,0){
      \begin{picture}(40,30)
        \node(1)(0,12){}\def\h{9}
        \node(a)(\h,20){}\node(b)(\h,4){}\def\h{18}
        \node(aa)(\h,24){}\node(ab)(\h,16){}\node(ba)(\h,4){}\def\h{27}
        \node(aab)(\h,24){}\node(aba)(\h,16){}\node(baa)(\h,8){}
        \node(bab)(\h,0){} \def\h{36}
        \node[Nmr=0](aaba)(\h,24){}\node[Nmr=0](abaa)(\h,20){}\node[Nmr=0](abab)(\h,12){}
        \node[Nmr=0](baab)(\h,8){}\node[Nmr=0](baba)(\h,0){}
        \drawedge(1,a){$a$}\drawedge[ELside=r](1,b){$b$}
        \drawedge(a,aa){$a$}\drawedge[ELside=r](a,ab){$b$}\drawedge(b,ba){$a$}
        \drawedge(aa,aab){$b$}\drawedge(ab,aba){$a$}\drawedge(ba,baa){$a$}
        \drawedge[ELside=r](ba,bab){$b$}
        \drawedge(aab,aaba){$a$}\drawedge(aba,abaa){$a$}\drawedge[ELside=r](aba,abab){$b$}
        \drawedge[ELside=r](baa,baab){$b$}\drawedge(bab,baba){$a$}
      \end{picture}
    }
    \put(45,0){
      \begin{picture}(40,30)
        \node(1)(0,12){}\def\h{9}
         \node(a)(\h,20){}\node(b)(\h,4){}\def\h{18}
        \node(aa)(\h,24){}\node(ab)(\h,16){}\node(ba)(\h,4){}\def\h{27}
        \node(aab)(\h,24){}\node[Nmr=0](aba)(\h,16){}\node(baa)(\h,8){}
        \node(bab)(\h,0){} \def\h{36}
        \node(aaba)(\h,24){}
        \node[Nmr=0](baab)(\h,8){}\node(baba)(\h,0){}\def\h{44}
        \node[Nmr=0](aabaa)(\h,28){}\node[Nmr=0](aabab)(\h,20){}
        \node[Nmr=0](babaa)(\h,0){}
        \drawedge(1,a){$a$}\drawedge[ELside=r](1,b){$b$}
        \drawedge(a,aa){$a$}\drawedge[ELside=r](a,ab){$b$}\drawedge(b,ba){$a$}
        \drawedge(aa,aab){$b$}\drawedge(ab,aba){$a$}\drawedge(ba,baa){$a$}
        \drawedge[ELside=r](ba,bab){$b$}
        \drawedge(aab,aaba){$a$}
        \drawedge(baa,baab){$b$}\drawedge(bab,baba){$a$}
        \drawedge(aaba,aabaa){$a$}\drawedge[ELside=r](aaba,aabab){$b$}
        \drawedge(baba,babaa){$a$}
      \end{picture}
    }
    \put(100,5){
      \begin{picture}(20,20)
        \gasset{AHnb=0}
        \node(1a)(0,15){$a$}\node(1b)(0,0){$b$}
        \node(a1)(20,15){$a$}\node(b1)(20,0){$b$}
        \drawedge(1a,a1){}\drawedge(1a,b1){}
        \drawedge(1b,a1){}
      \end{picture}
    }
  \end{picture}
  \caption{The codes $X,Y$ and the graph $H$.}\label{exInternal}
\end{figure}

\subsection{Periodicity}

Let $x=a_0a_1\cdots$, with $a_i\in A$, be an infinite word.  It is
\emph{periodic}\index{periodic word}\index{word!periodic} if there is
an integer $n\ge 1$ such that $a_{i+n}=a_i$ for all $i\ge 0$.  It is
\emph{ultimately periodic}\index{ultimately periodic
  word}\index{word!ultimately periodic} if the equalities hold for all
$i$ large enough. Thus, $x$ is ultimately periodic if there is a word
$u$ and a periodic infinite word $y$ such that $x=uy$.  The following
result, due to Coven and Hedlund, is well-known
(see~\cite{Lothaire2002}, Theorem 1.3.13).

\begin{theorem}\label{theoremCovenHedlund}
  Let $x\in A^\N$ be an infinite word on an alphabet with $k$
  letters. If there exists an integer $d\ge 1$ such that $x$ has at
  most $d+k-2$ factors of length $d$ then $x$ is ultimately periodic.
\end{theorem}

We will prove a  generalization of this result. 

Let $x$ be an infinite word and let $X$ be a thin maximal bifix
code. Let $u$ be a prefix of $x$ and set $x=uy$. Since $F(y)\subset
F(x)$, one has $d_{F(y)}(X)\le d_{F(x)}(X)$.  The word $x$ is called
$X$-\emph{stable}\index{stable word}\index{X-stable@$X$-stable
  word}\index{word!stable} if $d_{F(y)}(X)=d_{F(x)}(X)$ for all
suffixes $y$ of $x$.  Let $u$ be a prefix of $x$ such that
$d_{F(y)}(X)$ is minimal. Then the infinite word $y$ is $X$-stable.

For example, if $x=ba^\omega$ and $X=a\cup ba^*b$, then an $X$-stable
suffix of $x$ is~$a^\omega$.

\begin{theorem}\label{theoremPeriod} 
 Let  $X$ be a thin maximal
 bifix code and let $x\in A^\N$ be an $X$-stable infinite word. 
If $\Card(X\cap F(x))\le  d_{F(x)}(X)$, then $x$ is ultimately periodic.
\end{theorem}

\begin{corollary}\label{corollaryPeriod} 
  Let $x\in A^\N$ be an infinite word. If there
  exists a finite maximal bifix code $X$ of degree $d$ such that
  $\Card(X\cap F(x))\le d$, then $x$ is ultimately periodic.
\end{corollary}
\begin{proof}
  Since any long enough word has $d$ parses, $d_{F(x)}(X)=d$ and $x$
  is $X$-stable.  Since $\Card(X\cap F(x))\le d$, the conclusion
  follows by Theorem~\ref{theoremPeriod}.
\end{proof}

Corollary~\ref{corollaryPeriod} implies
Theorem~\ref{theoremCovenHedlund} in the case $k=2$ since $A^d$ is a
maximal bifix code of degree $d$.

\begin{example}
  Let us consider again the finite maximal bifix code $X$ of degree
  $3$ defined by $X = \{a^3, a^2ba, a^2b^2, ab, ba^2, baba, bab^2,
  b^2a, b^3\}$ (see Example~\ref{exFiboDegre2}). Assume that $X \cap F
  = \{ a^2ba,ab, baba \}$, where $F = F(x)$ and $x \in A^{\N}$. Since
  $ba$ is a factor of $x$, there exist a word $u$ and an infinite word
  $y$ such that $x=ubay$. Next, the first letter of $y$ is $b$
  (otherwise, $ba^2 \in X \cap F$) and the second letter of $y$ is $a$
  (otherwise, $bab^2 \in X \cap F$). This argument shows that whenever
  $uba$ is a prefix of $x$ then $ubaba$ is also a prefix of $x$, i.e.,
  $x = u (ba)^{\omega}$, with $u \in A^*$.
\end{example}

\begin{example}
  The set $X=a\cup ba^*b$ is a maximal bifix code of degree 2.  An
  argument similar to the previous one shows that any infinite word
  $x\in A^\N$ such that $X\cap F(x)=\{a,bab\}$ belongs to the set
  $a^*(ba)^\omega$. Thus it is ultimately periodic.
\end{example}

\begin{corollary}\label{corollaryPeriod2} 
  Let $x\in A^\N$ be an infinite word and let $X$ be a thin maximal
  bifix code.  Let $y$ be an $X$-stable suffix of $x$ and let
  $F=F(y)$.  If $\Card(X\cap F)\le d_F(X)$, then $x$ is ultimately
  periodic.
\end{corollary}

\begin{proof}
  By Theorem~\ref{theoremPeriod}, the word $y$ is ultimately periodic,
  and so is $x$.
\end{proof}

The following example shows that Corollary~\ref{corollaryPeriod2}
may become false if we replace $F=F(y)$ by $F=F(x)$ in the
statement.

\begin{example}
  Let $X$ be the maximal bifix code of degree $4$ on the alphabet
  $A=\{a,b,c\}$ with kernel $K=\{a,b\}^2$.

Let $x=ccay$ where $y$ is an infinite word without any occurrence of
$c$.  Then $cca$ has no factor in $X$. Indeed, a word of $X$ of length
at most $3$ is in the kernel of $X$ and thus is not a factor of $cca$.
Thus $cca$ has $4$ parses with respect to $X$, namely $(1,1,cca)$,
$(c,1,ca)$, $(cc,1,a)$ and $(cca,1,1)$.  Thus we have $d_{F(x)}(X)=4$. On
the other hand $X\cap F(x)\subset\{a,b\}^2$ and thus $\Card(X\cap
F(x))\le d_{F(x)}(x)$ although $x$ need not be ultimately
periodic. This shows that we cannot replace $F(y)$ by $F(x)$ in the
statement of Corollary~\ref{corollaryPeriod2}.
\end{example}

The proof uses the Critical Factorization Theorem
(see~\cite{Lothaire1983,CrochemorePerrin1991}) that we recall
below. For a pair of words $(p,s)\ne(1,1)$, consider the set of
nonempty words $r$ such that
\begin{displaymath}
  A^*p\cap A^*r\ne\emptyset,\quad sA^*\cap rA^*\ne\emptyset\, .
\end{displaymath}
This is the set of nonempty words $r$ which are prefix-comparable with
$s$ and suffix-comparable with $p$. This set is nonempty since it
contains $r = sp$.  The \emph{repetition}\index{repetition}
$\rep(p,s)$ is the minimal length of such a nonempty word $r$.

Let $w=a_1a_2\cdots a_m$ be a word with $a_i\in A$.  An
integer $n\ge 1$ is a \emph{period}\index{period of a word} of $w$ if
for $1\le i\le j\le m$, $j-i=n$ implies $a_i=a_j$.  Recall that a
\emph{factorization}\index{factorization of a word} of a word $w\in
A^*$ is a pair $(p,s)$ of words such that $w=ps$.

\begin{theorem}[Critical Factorization
  Theorem]\label{theoremCriticalFactorization} 
  For any word $w\in A^+$, the maximal value of $\rep(p,s)$ for all
  factorizations $(p,s)$ of $w$ is the least period of $w$.
\end{theorem}

We will also use the following lemma.

\begin{lemma}\label{lemmaPeriodInfinite}
  Let $x$ be an infinite word and $n\ge 1$ be an integer such that the
  least period of an infinite number of prefixes of $x$ is at most
  $n$. Then $x$ is periodic.
\end{lemma}

\begin{proof}
  Since the least period of an infinite number of prefixes of $x$ is
  at most $n$, an infinity of them have the same least period. Let $p$
  be such that an infinite number of prefixes of $x$ have least period
  $p$.  Set $x=a_0a_1\cdots$ with $a_i\in A$. For each $i\ge 0$, there
  is a prefix of $x$ of length larger than $i+p$ with least period
  $p$. Thus $a_i=a_{i+p}$. This shows that $x$ is periodic.
\end{proof}

\begin{proofof}{of Theorem \ref{theoremPeriod}}

  Let $S=A^*\setminus A^*X$ and $P=A^*\setminus XA^*$. Set $F=F(x)$
  and $d=d_F(X)$. Since $\Card(X\cap F)\le d_F(X)\le d(X)$, the set
  $X\cap F$ is finite.  Since $x$ is $X$-stable, there are an infinite number of
  factors and therefore also of prefixes of $x$ which have $d$ parses
  with respect to $X$. Indeed, for any factorization $x=uy$, we have
  $d_{F(y)}(X)=d$ and thus $y$ has a factor which has $d$ parses, so
  it has a prefix $w$ with $d$ parses, and finally $uw$ is a
  prefix of $x$ with $d$ parses.

  Let $n$ be the maximal length of the words in $X\cap F$. Let $u$ be
  a prefix of $x$ of length larger than $n$ which has $d$ parses and
  set $x=uy$.  Let $w$ be a nonempty prefix of $y$ and set $y=wz$.
  Let $v$ be a prefix of $z$ of length larger than $n$ which has $d$
  parses.

  Let $(p,s)$ be a factorization of $w$. We show that $\rep(p,s)\le
  n$.

Since $up$ has $d$
parses with respect to $X$, there are $d$  suffixes
$p_1,p_2,\ldots,p_d$
of $up$
which are in $P$. We may assume that $p_1=1$. Similarly, there are
$d$ prefixes $s_1,s_2,\ldots,s_d$ of $sv$ which are in $S$. We may
assume that $s_1=1$. 

Since $upsv$ has $d$ parses, for each $p_i$ with $2\le i\le d$ there
is exactly one $s_j$ with $2\le j\le d$ such that $p_is_j\in
X$. Indeed, there is a prefix $s'$ of $sv$ such that $p_is'\in
X$. Since $s'$ must be one of the $s_j$, the conclusion follows.

We may renumber the $s_i$ in such a way that $p_is_i\in X$ for $2\le
i\le d$. Set $x_i=p_is_i$. Since $up\notin S$, we have $up\in A^*X$.
Let $x_0$ be the word of $X$ which is a
suffix of $up$. Similarly, let $x_1$ be the word of $X$ which is a
prefix of $sv$ (see Figure~\ref{figureInterpretations}).

\begin{figure}[hbt]
\centering
\gasset{Nadjust=wh,AHnb=0}
\begin{picture}(80,20)(0,-5)
\node(uh)(0,10){}\node(ph)(20,10){}\node(sh)(40,10){}\node(zh)(60,10){}\node(endh)(80,10){}
\node(u)(0,0){}\node(p)(20,0){}\node(x1)(25,0){}\node(xi)(32,0){}\node(s)(40,0){}\node(fin)(48,0){}
\node(x2)(55,0){}
\node(z)(60,0){}\node(end)(80,0){}

\drawedge[curvedepth=3](xi,fin){$x_i$}
\drawedge(uh,ph){$u$}\drawedge(ph,sh){$p$}\drawedge(sh,zh){$s$}\drawedge(zh,endh){$v$}
\drawedge(u,p){$u$}\drawedge(p,x1){}\drawedge(x1,xi){}\drawedge(xi,s){}
\drawedge(s,fin){}\drawedge(fin,x2){}\drawedge(x2,z){}\drawedge(z,end){$v$}
\drawedge[curvedepth=-5,ELside=r](x1,s){$x_0$}\drawedge[curvedepth=-5,ELside=r](s,x2){$x_1$}
\end{picture}
\caption{The $d+1$ words $x_0,x_1,\ldots,x_d$.}\label{figureInterpretations}
\end{figure}
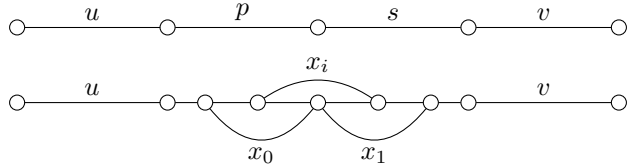
Since $\Card(X\cap F)\le d$, two of the $d+1$ words
$x_0,x_1,\ldots,x_d$ are equal.

If $x_0=x_1$, then $\rep(p,s)\le n$. 

If $x_0=x_i$ for an index $i$ with $2\le i\le d$, then $s_i$ is a
suffix of $up$ (since it is a suffix of $x_0$) and a prefix of $s v$
(by definition of $s_i$). Furthermore $|s_i| \leq n$ (since $n$ is the
maximal length of the words of $X\cap F$).  Thus $\rep(p,s) \leq |s_i| \leq
n$.

The case where $x_i=x_1$ for an index $i$ with $2\le i\le d$ is
similar.

Assume finally that $x_i=x_j$ for some indices $i,j$ such that $2\le
i<j\le d$.  We may assume that $|p_i| < |p_j|$. Thus $p_j = p_i t$,
$ts_j = s_i$. As a consequence, $t$ is both a suffix of $up$ (since it
is a suffix of $p_j$) and a prefix of $s v$ (since it is a prefix of
$s_i$). Thus again, $\rep(p,s)\le|t|\le n$.

By the Critical Factorization Theorem, this implies that the least
period of $w$ is at most equal to $n$.  Thus an infinite number of
prefixes of $y$ have least period at most $n$. By
Lemma~\ref{lemmaPeriodInfinite}, it implies that $y$ is periodic.
  
\end{proofof}

\section{Bases of subgroups}\label{sectionBasis}

In this section, we push further the study of bifix codes in Sturmian
sets.  The main result of Section~\ref{subsectionSturmianBasis} is
Theorem~\ref{theoremGroups}. It states that a $F$-maximal bifix code
$X\subset F$ of $F$-degree $d$ is a basis of a subgroup of index $d$
of the free group on $A$. The proof uses two sets of preliminary
results. The first part concerns bases of subgroups composed of words
over $A$, already considered in~\cite{Reutenauer1979}. The second one
uses the first return words, which were introduced independently
in~\cite{Durand1998},~\cite{HoltonZamboni1999}, and which we use in the framework
of~\cite{JustinVuillon2000} and~\cite{Vuillon2001}, up to a left-right
symmetry (see also~\cite{AraujoBruyere2005}).

We denote by $A^\circ$ the free group generated by $A$. The
\emph{rank}\index{rank!free group} of $A^\circ$ is $\Card(A)$.  Note
that all sets generating a free group of rank $k$ have at least $k$
elements. A basis is a minimal generating set. All bases have exactly
$k$ elements (see e.g.~\cite{LyndonSchupp2001}).

Let $H$ be a subgroup of rank $n$ and of index $d$ of a free group of
rank $k$. Then
\begin{equation}
  n=d(k-1) +1\,. \label{equationSchreier}
\end{equation}
Formula~\eqref{equationSchreier} is called \emph{Schreier's
  Formula}\index{Schreier's Formula}.

The free monoid $A^*$ is viewed as embedded in $A^\circ$. An element
of the free group is represented by its unique reduced word on the
alphabet $A\cup A^{-1}$. The elements of the free monoid $A^*$ are
themselves reduced words since they do not contain any letter in
$A^{-1}$. Thus $A^*$ is a submonoid of $A^\circ$. The subgroup of
$A^\circ$ generated by a subset $X$ of $A^\circ$ is denoted $\langle
X\rangle$. 

In any group $G$, the \emph{right cosets}\index{right cosets of a
  subgroup}\index{subgroup!right cosets} of a subgroup $H$ are the
sets of the form $Hg$ for $g\in G$. Two right cosets of the same
subgroup are disjoint or equal. The \emph{index}\index{index of a
  subgroup}\index{subgroup!index} of a subgroup is the number of its
distinct right cosets.  If $K$ is a subgroup of the subgroup $H$, then
the index of $K$ in $G$ is the product of the index of $K$ in $H$ and
of the index of $H$ in $G$. If $H,K$ are two subgroups of index $d$ of
a group $G$, then $H\subset K$ implies $H=K$.

Assume now that $G$ is a group of permutations over a set $Q$. 
For any $q$ in $Q$, the set of elements of $G$ that fixes $q$ is a
subgroup of $G$.

The group $G$ is \emph{transitive}\index{transitive permutation
  group}\index{permutation group!transitive} if, for all $p,q\in Q$,
there is an element $g\in G$ such that $pg=q$. In this case, the
subgroup $H$ of permutations fixing a given element $p$ of $Q$ has
index $\Card(Q)$. Indeed, for each $q\in Q$ let $g_q$ be an element of
$G$ such that $pg_q=q$. 
If $g\in G$ is such that $pg=q$, then
$pgg_q^{-1}=p$ and consequently $gg_q^{-1}\in H$, whence $g\in
Hg_q$. Thus each $g\in G$ is in one of the right cosets $Hg_q$, for
$q\in Q$.  Since these right cosets are pairwise disjoint, the index
of $H$ is $\Card(Q)$.

\subsection{Group automata}

A simple automaton $\A=(Q,1,1)$ is said to be
\emph{reversible}\index{automaton!reversible}\index{reversible
  automaton} if for any $a\in A$, the partial map
$\varphi_\A(a):p\mapsto p\cdot a$ is injective. This condition allows
to construct the \emph{reversal} of the automaton\index{reversal of an
  automaton} as follows: whenever $q\cdot a =p$ in $\A$, then $p\cdot
a =q$ in the reversal automaton. The state~$1$ is the initial and the
unique terminal state of this automaton.  Any reversible automaton is
minimal~\cite{Reutenauer1979}. The set recognized by a reversible
automaton is a left and right unitary submonoid. Thus
it is generated by a bifix code.

An automaton $\A=(Q,1,1)$
is a \emph{group automaton}\index{group!automaton}\index{automaton!group}
if for any $a\in A$ the map $\varphi_\A(a):p\mapsto p\cdot a$
is a permutation of $Q$. When $Q$ is finite, a group automaton
is a reversible automaton which is complete. 

 The
following result is from~\cite{Reutenauer1979} (see also Exercise
6.1.2 in \cite{BerstelPerrinReutenauer2009}).

\begin{proposition}\label{lemmaExercise612}
  Let $X\subset A^+$ be a bifix code.  The following conditions are
  equiva\-lent.
  \begin{enumerate}
  \item[\upshape{(i)}] $X^*=\langle X\rangle\cap A^*$;
  \item[\upshape{(ii)}] the minimal automaton of $X^*$ is reversible.
  \end{enumerate}
\end{proposition}

Let $\A=(Q,i,T)$ be a deterministic automaton. A \emph{generalized
  path} \index{generalized path}\index{path!generalized} is a sequence
$(p_0,a_1,p_1,a_2,\ldots,p_{n-1},a_n,p_n)$ with $a_i\in A\cup A^{-1}$
and $p_i\in Q$, such that for $1\le i\le n$, one has $p_{i-1}\cdot
a_i=p_i$ if $a_i\in A$ and $p_{i}\cdot a^{-1}_i=p_{i-1}$ if $a_i\in
A^{-1}$.  The \emph{label} of the generalized path is the element
$a_1a_2\cdots a_n$ of the free group $A^\circ$. Note that if
$\A=(Q,1,1)$, the set of labels of generalized paths from $1$ to $1$ in $\A$ is a
subgroup of $A^\circ$. It is called the \emph{subgroup
  described}\index{subgroup described} by $\A$.

A path in an automaton is a particular case of a generalized
path. In the case where $\A$ has a unique terminal state which is
equal to the initial state, the submonoid of $A^*$ recognized by $\A$
is contained in the subgroup of $A^\circ$ described by $\A$.

\begin{example}\label{exGroupRecognized}
  Let $\A=(Q,1,1)$ be the automaton defined by $Q=\{1,2\}$, $1\cdot
  a=1\cdot b=2$ and $2\cdot a=2\cdot b=\emptyset$. The submonoid
  recognized by $\A$ is $\{1\}$.  The subgroup described by $\A$ is
  the cyclic group generated by $ab^{-1}$.
\end{example}

\begin{proposition}\label{propGeneratedGroup}
  Let $\A$ be a simple automaton and let $X$ be the prefix code
  generating the submonoid recognized by $\A$. The subgroup $H$ described
  by $\A$ is generated by $X$. If moreover $\A$ is reversible, then
  $X^*=H\cap A^*$.
\end{proposition}
\begin{proof}
  Set $\A=(Q,1,1)$.  Let $K$ be the subgroup
  described by $\A$. Let us show that $K=H$.  First, $X\subset K$
  implies $\langle X\rangle=H\subset K$. To prove the converse
  inclusion, let $(p_0,a_1,p_1,a_2,\ldots ,p_{n-1},a_n,p_n)$ be a
  generalized path with $a_i\in A\cup A^{-1}$, $p_i\in Q$ and
  $p_0=p_n=1$. Let $h\in A^\circ$ be the label of the path. Let $r$ be
  the number of indices $i$ such that $a_i\in A^{-1}$. We show by
  induction on $r$ that $h\in H$. This holds clearly if $r=0$. Assume
  that it is true for $r-1$. Let $i$ be the least index such that
  $a_i\in A^{-1}$. Set $u=a_1\cdots a_{i-1}$, $a=a_i^{-1}$,
  $v=a_{i+1}\cdots a_n$ in such a way that $h=ua^{-1}v$. Set also
  $p=p_{i-1}$ and $q=p_i$. Thus $1\cdot u=p$, $q\cdot a=p$ and $v$ is
  the label of a generalized path from $q$ to $1$.
Since $\A$ is trim there exist words
  $w,t\in A^*$ such that $p\cdot t=1$ and $1\cdot w=q$.  Since $1\cdot
  ut=1\cdot wat=1$ (see Figure~\ref{figureStallings}), we have
  $ut,wat\in X^*$.  By induction hypothesis, since $wv$ is the label
  of a generalized path from $1$ to $1$, we have $wv\in H$.  Then
  $ua^{-1}v=utt^{-1}a^{-1}w^{-1}wv=ut(wat)^{-1}wv$ is in $H$ and thus
  $h\in H$.
  
  Assume now that $\A$ is reversible. Then is minimal and, by
  Proposition~\ref{lemmaExercise612}, one has $X^*=H\cap A^*$.
\end{proof}

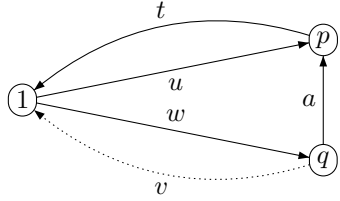
\begin{figure}[hbt]
  \centering
  \gasset{Nadjust=wh}
  \begin{picture}(40,20) 
    \node(1)(0,10){$1$}\node(p)(40,18){$p$}\node(q)(40,2){$q$}
    \drawedge[ELside=r](1,p){$u$}\drawedge(q,p){$a$}
    \drawedge(1,q){$w$}
    \drawedge[ELside=r,curvedepth=-6](p,1){$t$}
    \drawedge[dash={0.2 0.5}0,curvedepth=6](q,1){$v$}
  \end{picture}
  \caption{Paths in the automaton $\A$. The generalized path is dashed.}\label{figureStallings}
\end{figure}

For any subgroup $H$ of $A^\circ$, the submonoid $H\cap A^*$ is right
and left unitary. Thus $H\cap A^*$ is generated by a bifix code.
A subgroup $H$ of $A^\circ$ is \emph{positively
  generated}\index{subgroup!positively generated} if there is a set
$X\subset A^*$ which generates $H$. In this case, the set $H\cap A^*$
generates the subgroup $H$. Let $X$ be the bifix code which generates
the submonoid $H\cap A^*$. Then $X$ generates the subgroup $H$. This
shows that, for a positively generated subgroup $H$, there is a bifix
code which generates $H$.

\begin{proposition}\label{propStallings}
  For any positively generated subgroup $H$ of $A^\circ$, there is a
  unique reversible  automaton $\A$ such that $H$ is the
  subgroup described by $\A$. 
\end{proposition}

\begin{proof}
  Let $X$ be the bifix code generating the submonoid $H\cap A^*$, so
  that $X^*=H\cap A^*$.  Since $H$ is positively generated, the
  subgroup generated by $X$ is equal to $H$, that is $\langle
  X\rangle=H$. Thus $X^*=\langle X\rangle\cap A^*$. In view of
  Proposition~\ref{lemmaExercise612}, the minimal automaton $\A$ of
  $X^*$ is reversible.  Thus the submonoid recognized by $\A$ is
  $H\cap A^*$ and by Proposition~\ref{propGeneratedGroup}, $H$ is the
  subgroup described by $\A$.

  If $\B$ is another reversible automaton such that $H$ is the
  subgroup described by $\B$, then by
  Proposition~\ref{lemmaExercise612}, $\B$ recognizes the set $H\cap
  A^*$. Since $\B$ is minimal and since minimal automata are unique,
  the uniqueness follows.
\end{proof}

The reversible automaton $\A$ such that $H$ is the subgroup described
by $\A$ is called the \emph{Stallings automaton}\index{Stallings
  automaton} of the subgroup $H$. It can also be defined for a
subgroup which is not positively generated
(see~\cite{BartholdiSilva2011} or \cite{KapovichMyasnikov2002}).

\begin{proposition}\label{propositionGroupAutomaton}
  The following conditions are equivalent for a submonoid $M$ of
  $A^*$.
  \begin{enumerate}
  \item[\rm (i)] $M$ is recognized by a group automaton with $d$
    states.
  \item[\rm (ii)] $M=\varphi^{-1}(K)$, where $K$ is a subgroup of
    index $d$ of a group $G$ and $\varphi$ is a surjective morphism
    from $A^*$ onto $G$.
  \item[\rm (iii)] $M=H\cap A^*$, where $H$ is a subgroup of index $d$
    of $A^\circ$.
  \end{enumerate}
  If one of these conditions holds, the minimal generating set of $M$
  is a maximal bifix code of degree $d$.
\end{proposition}

\begin{proof}
  \noindent (i) implies (ii). Let $\A=(Q,1,1)$ be a group automaton
  with $d$ states and let $M$ be the set recognized by $\A$. Since a
  composition of permutations is a permutation, the monoid
  $G=\varphi_\A(A^*)$ is a permutation group. Since $\A$ is trim,
  there is a path from every state $q$ to any state $q'$ in
  $\A$. Consequently, $G$ is transitive.  Let $K$ be the subgroup of
  $G$ formed of the permutations fixing $1$. As we have seen earlier,
  $K$ has index $d$.  Then $M=\varphi_\A^{-1}(K)$.

\noindent (ii) implies (iii). Let $\psi$ be the morphism from
$A^\circ$ onto $G$ extending $\varphi$. Then $H=\psi^{-1}(K)$ is a
subgroup
of index $d$ of $A^\circ$ and $M=H\cap A^*$.

\noindent (iii) implies (i). Let $Q$ be the set of  right cosets of
$H$ with $1$ denoting the right coset $H$.
 The representation of $A^\circ$ by permutations on $Q$ defines
a group automaton $\A$ with $d$ states and $M$ is recognized by $\A$.

Finally, let $X$ be the minimal generating set of a submonoid $M$
satisfying one of these conditions. It is clearly a bifix code. Let
$P$ be the set of proper prefixes of $X$.  The number of suffixes of a
word which are in $P$ is at most equal to $d$. Indeed, let
$\A=(Q,1,1)$ be a group automaton recognizing $X^*$. If $s,t$ are
distinct suffixes of a word $w$ which are in $P$, then $1\cdot s\ne
1\cdot t$. Indeed, otherwise, since $s$ and $t$ are suffix-comparable,
we may assume that $s=ut$. Let $p=1\cdot u$.  Then $p\cdot t=1\cdot
ut=1\cdot s=1\cdot t$ and thus $p=1$ since $\A$ is reversible. Thus
$s=t$. Let $w$ be a word with the maximal number of suffixes in
$P$. Then $w$ cannot be an internal factor of $X$. Moreover the number
of suffixes of $w$ in $P$ is equal to $d$. Indeed, since $\A$ is a
group code, for any $q\in Q$, there is a state $q'$ such that $q'\cdot
w=q$. Since $w$ is not an internal factor of $X$, there is a
factorization $w=sxp$ such that $q'\cdot s=1$, $1\cdot x=1$ and
$1\cdot p=q$, and such that $(s,x,p)$ is a parse of $w$. Thus $X$ has
degree $d$.
\end{proof}

A bifix code $Z$ such that $Z^*$ satisfies one of the equivalent
conditions of the proposition~\ref{propositionGroupAutomaton} is
called a \emph{group code}\index{group!code}\index{code!group}.

The following proposition shows in particular that a subgroup
of finite index is positively generated.

\begin{proposition}\label{propGroupCode}
Let $H$ be a subgroup of finite index of $A^\circ$. The minimal
automaton $\A$ of $H\cap A^*$ is a group automaton which
describes the subgroup $H$. Let $X$ be the group code such that
$\A$ recognizes $X^*$. The subgroup generated by $X$ is $H$.
\end{proposition}
\begin{proof}
   By Proposition~\ref{propositionGroupAutomaton}, the monoid $H\cap
  A^*$ is recognized by a group automaton $\A=(Q,1,1)$, which is the
  minimal automaton of $H\cap A^*$.  The morphism $\varphi_\A$ from
  $A^*$ onto the group $G=\varphi_\A(A^*)$ of
  Proposition~\ref{propositionGroupAutomaton} extends to a morphism
  $\psi$ from $A^\circ$ onto $G$. The subgroup $K$ is composed of the
  permutations that fix~$1$, and the subgroup $H$ is formed of the
  elements $w\in A^\circ$ such that the permutation $\psi(w)$ fixes
  $1$. There is a
  generalized path in $\A$ from $p$ to $q$ labeled $w$ 
  if and only if $p\psi(w)=q$.
  Thus $\psi(w)$ fixes $1$ if and only there is a generalized
path from $1$ to $1$ labeled $w$, that is if $w\in
  H$. Thus the subgroup described by $\A$ is $H$.
By Proposition~\ref{propGeneratedGroup}, the subgroup $H$ is generated
by $X$.
\end{proof}




\begin{example}\label{exampleAd}
  The set $A^d$ is a group code by
  Proposition~\ref{propositionGroupAutomaton}(ii). Thus it is a
  maximal bifix code of degree $d$.  The intersection of the subgroup
  generated by $A^d$ with $A^*$ is the submonoid generated by $A^d$
  (Proposition~\ref{propGroupCode}). It is composed of the words with
  length a multiple of $d$.
\end{example}

\subsection{Main result}\label{subsectionSturmianBasis}

We will prove the following
result.

\begin{theorem}\label{theoremGroups}
  Let $F$ be a Sturmian set and let $d\ge 1$ be an integer. A bifix
  code $X\subset F$ is a basis of a subgroup of index $d$ of $A^\circ$
  if and only if it is a finite $F$-maximal bifix code of $F$-degree
  $d$.
\end{theorem}

Note that Theorem~\ref{theoremBifixd+1} is contained in
Theorem~\ref{theoremGroups} (we will use Theorem~\ref{theoremBifixd+1}
in the proof of Theorem~\ref{theoremGroups}).  Indeed, let $X$ be an
$F$-maximal bifix code of $F$-degree $d$.  By
Theorem~\ref{theoremCompletion}, $X$ is finite.  By
Theorem~\ref{theoremGroups}, the subgroup $\langle X\rangle$ has rank
$\Card(X)$ and index $d$ in the free group $A^\circ$.  By Schreier's
Formula~\eqref{equationSchreier}, one get $\Card(X)=(\Card(A)-1)d+1$.

Before proving Theorem~\ref{theoremGroups}, we list some corollaries.

\begin{corollary}
  Let $F$ be a Sturmian set. For any $d\ge 1$, the set of words in $F$
  of length $d$ is a basis of the subgroup of $A^\circ$ generated by
  $A^d$.
\end{corollary}

\begin{proof}
  The set $A^d$ is a group code (see Example~\ref{exampleAd}), and
  therefore is a maximal bifix code. The set $A^d\cap F$ is a finite
  bifix code. By Theorem~\ref{proposition1}, it is an $F$-maximal
  bifix code and has $F$-degree $d$.  The corollary follows from
  Theorem~\ref{theoremGroups}. Indeed $\langle A^d\cap F\rangle=
  \langle A^d\rangle$ since both are subgroups of $A^\circ$ of
  index~$d$.
\end{proof}

The following is also a complement to Theorem~\ref{proposition1}.  It
shows in particular that for any Sturmian set $F$, any subgroup of
$A^\circ$ of finite index has a basis contained in $F$. Note that this
This provides contains the fact that every subgroup of finite index
has a positive basis, see also Proposition~\ref{propGroupCode}.


\begin{corollary}\label{newCorollary}
  Let $F$ be a Sturmian set. The map which associates to $X\subset F$
  the subgroup $\langle X\rangle$ of $A^\circ$ generated by $X$ is a
  bijection between $F$-maximal bifix codes of $F$-degree $d$ and
  subgroups of $A^\circ$ of index $d$. Such a bifix code $X$ is a
  basis of $\langle X\rangle$.  The reciprocal bijection associates,
  to a subgroup $H$ of $A^\circ$, the set $Z\cap F$ where $Z$ is the
group code which is the
  minimal generating set of the submonoid $H\cap A^*$ of $A^*$.
\end{corollary}

\begin{proof}
Let first $X$ be a finite $F$-maximal bifix code of $F$-degree $d$. Then
$\langle X\rangle$ is a subgroup of index $d$ by
Theorem~\ref{theoremGroups}.

Conversely, let $H$ be a subgroup of index $d$ of $A^\circ$ and let
$Z$ be the group code such that $Z^*=H\cap A^*$. By
Theorem~\ref{proposition1}, the set $X=Z\cap F$ is an $F$-maximal
bifix code of $F$-degree $e\le d$. 
By Theorem~\ref{theoremCompletion}, $X$ is finite. 
By Theorem~\ref{theoremGroups}, the
subgroup $\langle X\rangle$ has index $e$. Since $\langle X\rangle$ is
a subgroup of $H$, $e$ is a multiple of $d$. Thus $d=e$ and $\langle
X\rangle=H$.

Finally, let $X$ be an $F$-maximal bifix code of $F$-degree $d$.
Then  $H=\langle X\rangle$ is a subgroup of index $d$ of $A^\circ$.
Let $Z$ be the group code such that $Z^*=H\cap A^*$ and let $Y=Z\cap F$.
Then $X\subset Y$ and thus $X=Y$ since $X$ is an $F$-maximal bifix code.
This shows that the two maps are mutually inverse.

\end{proof}

A set $W$ of words of $\{a,b\}^*$ is \emph{balanced}\index{balanced
  set of words} if for all $w,w'\in W$, $|w|=|w'|$ implies
$||w|_a-|w'|_a|\le 1$. It is a classical property that the set of
factors of a Sturmian word is balanced (Theorem 2.1.5
in~\cite{Lothaire2002}). Thus any Sturmian set on two letters is
balanced.

Following Richomme and S\'e\'ebold~\cite{RichommeSeebold2010}, we say
that a subset $X$ of $\{a,b\}^*$ is \emph{factorially
  balanced}\index{factorially balanced set of words}\index{balanced
  set of words!factorially} if the set of factors of words of $X$ is
balanced.  They show that a finite set $X\subset\{a,b\}^*$ is
contained in some Sturmian set if and only if it is factorially
balanced. Thus, we have the following consequence of
Theorem~\ref{theoremGroups}.

\begin{corollary}
  Let $X\subset \{a,b\}^*$ be a bifix code. The following conditions
  are equivalent.
  \begin{enumerate}
  \item[\upshape{(i)}] There exists a Sturmian set $F\subset\{a,b\}^*$
    such that $X\subset F$ and $X$ is a finite $F$-maximal bifix code.
  \item[\upshape{(ii)}] $X$ is a factorially balanced basis of a
    subgroup of finite index of $\{a,b\}^\circ$.
  \end{enumerate}
\end{corollary}

As a further consequence of Theorem~\ref{theoremGroups}, we have the
following result.

\begin{corollary}
  Let $F$ be a Sturmian set on an alphabet with $k$ letters.  The
  number $N_{d,k}$ of finite $F$-maximal bifix codes $X\subset F$ of
  $F$-degree $d$ satisfies $N_{1,k}=1$ and
  \begin{displaymath}
    N_{d,k}=d(d!)^{k-1}-\sum_{i=1}^{d-1}[(d-i)!]^{k-1}N_{i,k}.
  \end{displaymath}
\end{corollary}
The formula results directly from the formula, due to
Hall~\cite{Hall1949}, for the number of subgroups of index $d$ in a
free group of rank $k$.

The values for $k=2$ are given by the recurrence
\begin{displaymath}
N_{d,2}=d \ d!-\sum_{i=1}^{d-1}(d-i)!N_{i,2}.
\end{displaymath}
The first values are
\begin{displaymath}
\begin{array}{r|cccccccccc}
d&1 &2&3&4&5&6&7 &8&9&10\\ \hline
N_{d,2} &1 &3&13&71&461&3447&29093&273343& 2829325& 31998903 
\end{array}
\end{displaymath}
The values for $d=2,3$ are consistent with
Examples~\ref{exampleFiboDegree2} and \ref{exampleDegree3}.

The formula is known to enumerate also the indecomposable permutations
on $d+1$ elements (see~\cite{DressFranz1985},
\cite{OssonaRosenstiehl2004} and \cite{Cori2009}).

\subsection{Incidence graph}\label{sectionPreliminaries}

Let $X$ be a set, let $P$ be the set of its proper prefixes and
$S$ be the set of its proper suffixes. Set $P'=P\setminus 1$ and
$S'=S\setminus 1$. The \emph{incidence graph}\index{incidence
  graph}\index{graph!incidence} of $X$ is the undirected graph $G$
defined as follows. The set of vertices is $V=1\otimes P'\cup
S'\otimes 1$. Similarly to the proof of
Proposition~\ref{propositionChristophe}, the tensor product notation
is used to emphasize that $V$ is made of two disjoint copies of $P'$
and $S'$.  The edges of $G$ are the pairs $\{1\otimes p,s\otimes 1\}$,
for $p\in P'$ and $s\in S'$, such that $ps\in X$.

Let $C$ be a connected component of $G$, that is a maximal set of
vertices connected by paths. The \emph{trace}\index{trace of a
  component} of $C$ on $P'$ is the set of $p\in P'$ such that
$1\otimes p\in C$. Similarly, the trace of $C$ on $S'$ is the set of
$s\in S'$ such that $s\otimes 1\in C$.

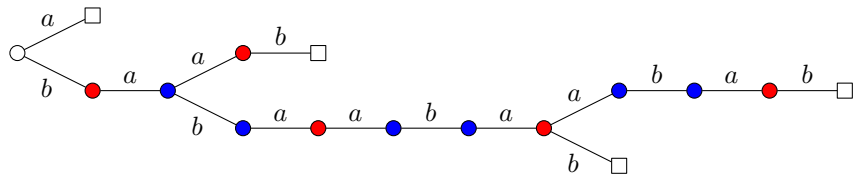
\begin{figure}[hbt]
  \gasset{Nadjust=wh,AHnb=0}
  \centering
  \begin{picture}(110,20)(0,-10)
    \node(1)(0,5){}\node[Nmr=0](a)(10,10){}\node[fillcolor=red](b)(10,0){}
    \node[fillcolor=blue](ba)(20,0){}\node[fillcolor=red](baa)(30,5){}\node[Nmr=0](baab)(40,5){}
    \node[fillcolor=blue](bab)(30,-5){}
    \node[fillcolor=red](baba)(40,-5){}
    \node[fillcolor=blue](babaa)(50,-5){}
    \node[fillcolor=blue](babaab)(60,-5){}
    \node[fillcolor=red](babaaba)(70,-5){}
    \node[fillcolor=blue](babaabaa)(80,0){}
    \node[fillcolor=blue](babaabaab)(90,0){}
    \node[fillcolor=red](babaabaaba)(100,0){}
    \node[Nmr=0](babaabaabab)(110,0){}
    \node[Nmr=0](babaabab)(80,-10){}
    \drawedge(1,a){$a$}\drawedge[ELside=r](1,b){$b$}
    \drawedge(b,ba){$a$}\drawedge(ba,baa){$a$}
    \drawedge[ELside=r](ba,bab){$b$}\drawedge(baa,baab){$b$}
    \drawedge(bab,baba){$a$}
    \drawedge(baba,babaa){$a$}
    \drawedge(babaa,babaab){$b$}
    \drawedge(babaab,babaaba){$a$}
    \drawedge(babaaba,babaabaa){$a$}
    \drawedge[ELside=r](babaaba,babaabab){$b$}
    \drawedge(babaabaa,babaabaab){$b$}
    \drawedge(babaabaab,babaabaaba){$a$}
    \drawedge(babaabaaba,babaabaabab){$b$}%
  \end{picture}
  \caption{The $F$-maximal bifix code of $F$-degree 3 with kernel
    $\{a,baab\}$.}
  \label{copie de figureExd+1}
\end{figure}

\begin{example}\label{exampleIncidenceGraph}
  Consider the $F$-maximal bifix code of $F$-degree $3$ in the
  Fibonacci set $F$ given in Figure~\ref{copie de figureExd+1}. It is
  a colored copy of Figure~\ref{figureExd+1}.  The incidence graph of
  $X$ is given in Figure~\ref{figureGrapheD3}.  It has two connected
  components colored red and blue.  The vertices on the left side are
  the $1\otimes p$ (written simply $p$ for convenience). The vertices
  on the right side are the $s\otimes 1$ with the same convention.

  The color on the node in Figure~\ref{copie de figureExd+1}
  corresponds to the color of the corresponding prefix in
  Figure~\ref{figureGrapheD3}.
\begin{figure}[hbt]
  \gasset{Nw=0,Nh=0,Nframe=n,AHnb=0,ExtNL=y,NLangle=0,NLdist=1}
  \centering
  \begin{picture}(40,66)(0,0)
    \node[NLangle=180](g0)(0,0){\textcolor{red}{$\vphantom{g}babaabaaba$}}
    \node(d0)(40,0){\textcolor{red}{$\vphantom{g}b$}}
    \node[NLangle=180](g1)(0,6){\textcolor{blue}{$\vphantom{g}babaabaab$}}
    \node(d1)(40,6){\textcolor{blue}{$\vphantom{g}ab$}}
    \node[NLangle=180](g2)(0,12){\textcolor{blue}{$\vphantom{g}babaabaa$}}
    \node(d2)(40,12){\textcolor{blue}{$\vphantom{g}bab$}}      
    \node[NLangle=180](g3)(0,18){\textcolor{red}{$\vphantom{g}babaaba$}}
    \node(d3)(40,18){\textcolor{red}{$\vphantom{g}aab$}}
    \node[NLangle=180](g4)(0,24){\textcolor{blue}{$\vphantom{g}babaab$}}
    \node(d4)(40,24){\textcolor{red}{$\vphantom{g}abab$}}
    \node[NLangle=180](g5)(0,30){\textcolor{blue}{$\vphantom{g}babaa$}}
    \node(d5)(40,30){\textcolor{blue}{$\vphantom{g}aabab$}}
    \node[NLangle=180](g6)(0,36){\textcolor{red}{$\vphantom{g}baba$}}
    \node(d6)(40,36){\textcolor{blue}{$\vphantom{g}baabab$}}
    \node[NLangle=180](g7)(0,42){\textcolor{blue}{$\vphantom{g}bab$}}
    \node(d7)(40,42){\textcolor{red}{$\vphantom{g}abaabab$}}
    \node[NLangle=180](g8)(0,48){\textcolor{red}{$\vphantom{g}baa$}}
    \node(d8)(40,48){\textcolor{blue}{$\vphantom{g}aabaabab$}}
    \node[NLangle=180](g9)(0,54){\textcolor{blue}{$\vphantom{g}ba$}}
    \node(d9)(40,54){\textcolor{blue}{$\vphantom{g}baabaabab$}}
    \node[NLangle=180](g10)(0,60){\textcolor{red}{$\vphantom{g}b$}}
    \node(d10)(40,60){\textcolor{red}{$\vphantom{g}abaabaabab$}}
    \drawedge[linecolor=red](g10,d10){} 
    \drawedge[linecolor=red](g10,d7){} 
    \drawedge[linecolor=red](g10,d3){}
    \drawedge[linecolor=blue](g9,d9){}
    \drawedge[linecolor=blue](g9,d6){}\drawedge[linecolor=blue](g9,d1){}
    \drawedge[linecolor=red](g8,d0){}
    \drawedge[linecolor=blue](g7,d8){}\drawedge[linecolor=blue](g7,d5){}
    \drawedge[linecolor=red](g6,d7){}\drawedge[linecolor=red](g6,d4){}
    \drawedge[linecolor=blue](g5,d6){}\drawedge[linecolor=blue](g5,d2){}
    \drawedge[linecolor=blue](g4,d5){}\drawedge[linecolor=blue](g4,d1){}
    \drawedge[linecolor=red](g3,d4){}\drawedge[linecolor=red](g3,d0){}
    \drawedge[linecolor=blue](g2,d2){}
    \drawedge[linecolor=blue](g1,d1){}
    \drawedge[linecolor=red](g0,d0){}

  \end{picture}
  \caption{The incidence graph of $X$.}
  \label{figureGrapheD3}
\end{figure}
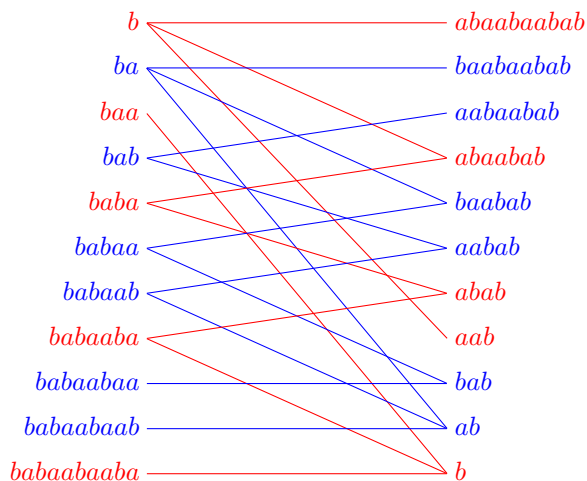
\end{example}
  The following lemma uses an argument similar to
Lemma~\ref{lemmaTree}.

\begin{lemma}\label{lemmaLCP}
  Let $v_1,v_2,\ldots,v_{n+1}$ be words such that $v_i,v_{i+1}$ are
  not prefix-comparable for $1\le i\le n$. Let $p_i$ be the longest
  common prefix of $v_i,v_{i+1}$, for $1\le i\le n$. If two of the
  $v_i$ are prefix-comparable, then two of the $p_i$ are equal.
\end{lemma}

\begin{proof}
  Let $V=\{v_1,\ldots,v_{n+1}\}$, let $P$ be the set of proper
  prefixes of $V$ and let $W=V\setminus P$. The set $W$ is the set of
  words of $V$ which have no proper prefix in $V$. The set $W$ is a
  prefix code. If two distinct words in $V$ are prefix-comparable,
  then $\Card(W)<\Card(V)\le n+1$.

  Let $m$ be the number of distinct $p_i$.  Since $v_i,v_{i+1}$ are
  not prefix-comparable for $1\le i\le n$, for each $p_i$ there are at
  least two distinct letters $a,b$ such that $p_ia,p_ib\in P\cup
  W$. This implies $m<\Card(W)$.  Indeed, the set $W$ can be seen as
  the set of leaves in a tree, and each $p_i$ is a fork node
  (i. e. a node with at least two children) in this tree. It is well-known
  that the number of fork nodes is strictly less than the number of
  leaves. If two of the $v_i$ are prefix-comparable, the inequality
  $\Card(W)<n+1$ implies $m<\Card(W)\le n$, and consequently two of
  the $p_i$ are equal.
\end{proof}

\begin{lemma}\label{lemmaSuite}
  Let $F$ be a Sturmian set and let $X\subset F$ be a bifix code. Let
  $P'$ (resp. $S'$) be the set of nonempty proper prefixes
  (resp. suffixes) of $X$ and let $G$ be the incidence graph of~$X$.
  \begin{itemize}
  \item[\upshape{(i)}] The graph $G$ is acyclic, that is a union of trees.
  \item[\upshape{(ii)}] The trace on $P'$ (resp. on $S'$) of a
    connected component $C$ of $G$ is a suffix (resp. prefix) code.
  \item[\upshape{(iii)}] In particular, this trace on $P'$ (resp. on
    $S'$) contains at most one right-special (resp. left-special)
    word.
  \end{itemize}
\end{lemma}

\begin{proof} The last assertion follows from the second by
  Proposition~\ref{propPrefixSpecial}.  We call a path
  \emph{reduced}\index{reduced path}\index{path, reduced}
  if it does not use equal consecutive edges.

  We prove by induction on $n\ge 1$ that if $s\otimes 1$ and $t\otimes
  1$ (resp. $1\otimes p$ and $1\otimes q$) are connected by a reduced
  path of length $2n$ in $G$, then $s,t$ are not prefix-comparable
  (resp. $p,q$ are not suffix-comparable). This shows that $G$ is
  acyclic. Indeed, if there were a cycle from $s$ to $t=s$ in
  $G$, then $s$ and $t$ would be prefix-comparable. This shows also
  that two words in the same trace on $P'$ (resp. on $S'$) are not
  suffix-comparable (resp. are not prefix-comparable).

\begin{figure}[hbt]
  \gasset{Nw=0,Nh=0,Nframe=n,AHnb=0,ExtNL=y,NLangle=0,NLdist=1}
  \centering
  \begin{picture}(20,8)(0,-10)
    \node(s)(20,6){$\vphantom{gd}s$}
    \node[NLangle=180](q)(0,3){$\vphantom{gd}q$}
    \node(t)(20,0){$\vphantom{gd}t$}
    \drawedge[linecolor=red](q,s){}
    \drawedge[linecolor=red](q,t){}
  \end{picture}
\qquad\qquad\qquad\qquad
  \begin{picture}(20,24)(0,0)
    \node(s)(20,24){$\vphantom{gd}s$}
    \node[NLangle=180](u1)(0,21){$\vphantom{gd}u_1$}
    \node(v2)(20,18){$\vphantom{gd}v_2$}
    \node(vdot)(20,12){$\vdots$}
    \node(vn)(20,6){$\vphantom{gd}v_n$}
    \node[NLangle=180](un)(0,3){$\vphantom{gd}u_n$}
    \node(t)(20,0){$\vphantom{gd}t$}
    \drawedge[linecolor=red](u1,s){}
    \drawedge[linecolor=red](u1,v2){}
    \drawedge[linecolor=red](un,vn){}
    \drawedge[linecolor=red](un,t){}
  \end{picture}
  \caption{A path $(s\otimes 1,1\otimes q, t\otimes 1)$ on the left,
    and a path of length $2n$ on the right.}
  \label{figurelemmaSuite1}
\end{figure}
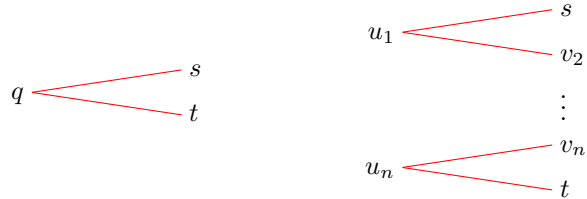

The property holds for $n=1$. Indeed, a reduced path of length $2$
from $s\otimes 1$ to $t\otimes 1$ is of the form $(s\otimes 1,1\otimes
q, t\otimes 1)$ with $qs,qt\in X$. Since the path is reduced, $s\ne
t$, and since $X$ is prefix, $s$ and $t$ are not prefix-comparable,
see Figure~\ref{figurelemmaSuite1}. The proof for prefixes is similar.

Let $n\ge 2$. A path of length $2n$ from $s\otimes 1$ to $t\otimes 1$
is a sequence $(v_1\otimes 1,1\otimes u_1,v_2\otimes 1,\ldots,1\otimes
u_n,v_{n+1}\otimes 1)$ with $s=v_1$ and $t=v_{n+1}$ such that the $2n$
words defined for $1\le i\le n$ by
\begin{displaymath}
  x_{2i-1}=u_{i}v_{i},\quad x_{2i}=u_{i}v_{i+1}.
\end{displaymath}
are in $X$. Moreover, since the path is reduced, one has $x_j\ne
x_{j+1}$ for $1\le j< 2n$.

For $1\le i\le n$, let $p_i$ be the longest common prefix of
$v_i,v_{i+1}$.  Since $x_{2i-1}\ne x_{2i}$ and since the code $X$ is
prefix, the words $v_i$ and $v_{i+1}$ are not prefix-comparable.

Arguing by contradiction, assume that $v_1$ and $v_{n+1}$ are
prefix-comparable. By Lemma~\ref{lemmaLCP}, we have $p_i=p_j$ for some
indices $i,j$ with $1\le i<j\le n$.

Set $v_i=p_iv'_i$ and $v_{i+1}=p_iv''_i$. Since $v_i,v_{i+1}$ are not
prefix-comparable, the words $v'_i,v''_i$ are nonempty. Since their
longest common prefix is empty, their initial letters are
distinct. Thus $u_ip_i$ is right-special.  Similarly $u_jp_j$ is
right-special. Thus $u_ip_i$ and $u_jp_j$ are suffix-comparable.
Since $p_i=p_j$, $u_i$ and $u_j$ are suffix-comparable.

But $1\otimes u_i$ and $1\otimes u_j$ are connected by the path
$(1\otimes u_i, v_{i+1}\otimes 1,\ldots,v_j\otimes 1,1\otimes u_j)$
of length $2(j-i)\le2(n-1)$. By the induction hypothesis, $u_i$ and $u_j$ are
not suffix-comparable, a contradiction.

The proof that if $1\otimes p$ and $1\otimes q$ are connected by a
path of length $2n$ in $G$, then $p,q$ are not suffix-comparable is
similar.
\end{proof}

Let $X$ be a bifix code and let $P$ be the set of proper prefixes of
$X$. Consider the equivalence $\theta_X$ on $P$ which is the
transitive closure of the relation formed by the pairs $p,q\in P$ such
that $ps,qs\in X$ for some $s\in A^+$. Such a pair corresponds, when
$p,q\ne1$, to a path $(1\otimes p,s\otimes1,1\otimes q)$ in the
incidence graph of $X$.  Thus a class of $\theta_X$ is either reduced
to the empty word or it is the trace on $P\setminus1$ of a connected
component of the incidence graph of $X$.


\begin{example}\label{exampleTheta}
  Consider the code $X$ of Example~\ref{exampleIncidenceGraph}
  above. The three classes of $\theta_X$ are the class $\{1\}$ of the
  empty word, and the two suffix codes which are the traces of
  connected components of the incidence graph on the set of nonempty
  proper prefixes of $X$. These codes are 
$\{babaabaaba,babaaba,baba,baa,b\}$
and
$\{babaabaab,babaabaa,babaab,babaa,bab,ba\}$.
They are shown in
  Figure~\ref{figureSuffixMax}.
\end{example}

\begin{figure}[hbt]
  \gasset{Nadjust=wh,AHnb=0,linecolor=red}
  \centering
  \begin{picture}(110,20)(0,-8)
    \node[Nmr=0](babaabaaba)(0,5){}
     \node(abaabaaba)(10,5){}
    \node(baabaaba)(20,5){}
    \node(aabaaba)(30,5){}
    \node[Nmr=0](babaaba)(30,-5){}
    \node(abaaba)(40,0){}
    \node(baaba)(50,0){}
    \node(aaba)(60,0){}
    \node[Nmr=0](baba)(60,-10){}
    \node(aba)(70,-5){}
    \node[Nmr=0](baa)(70,5){}
    \node(ba)(80,-5){}
    \node(aa)(80,5){}
    \node(a)(90,0){}
    \node[Nmr=0](b)(90,-10){}
    \node(1)(100,-5){}
\drawedge(babaabaaba,abaabaaba){$b$}
\drawedge(abaabaaba,baabaaba){$a$}
\drawedge(baabaaba,aabaaba){$b$}
\drawedge(aabaaba,abaaba){$a$}
\drawedge(abaaba,baaba){$a$}
\drawedge(baaba,aaba){$b$}
\drawedge(aaba,aba){$a$}
\drawedge(aba,ba){$a$}
\drawedge[ELside=r](ba,a){$b$}
\drawedge(a,1){$a$}
\drawedge[ELside=r](babaaba,abaaba){$b$}
\drawedge[ELside=r](baba,aba){$b$}
\drawedge(baa,aa){$b$}
\drawedge(aa,a){$a$}
\drawedge[ELside=r](b,1){$b$}
  \end{picture}
\gasset{linecolor=blue}
  \begin{picture}(90,35)(-10,-18)
   \node[Nmr=0](babaabaab)(-10,-5){}
           \node(abaabaab)(0,-5){}
   \node[Nmr=0](babaabaa)(0,5){}
   \node(baabaab)(10,-5){}
   \node(abaabaa)(10,5){}
   \node(aabaab)(20,-5){}
   \node[Nmr=0](babaab)(20,-15){}
     \node(baabaa)(20,5){}
    \node(aabaa)(30,5){}
    \node[Nmr=0](babaa)(30,-3){}
    \node(abaab)(30,-10){}
    \node(abaa)(40,0){}
    \node(baab)(40,-10){}
    \node(baa)(50,0){}
    \node(aab)(50,-10){}
    \node[Nmr=0](bab)(50,-20){}
    \node(aa)(60,0){}
    \node[Nmr=0](ba)(60,-8){}
    \node(ab)(60,-15){}
    \node(a)(70,-5){}
    \node(b)(70,-15){}
    \node(1)(80,-10){}
\drawedge(babaabaab,abaabaab){$b$}
\drawedge(babaabaa,abaabaa){$b$}
\drawedge(abaabaab,baabaab){$a$}
\drawedge(abaabaa,baabaa){$a$}
\drawedge(baabaab,aabaab){$b$}
\drawedge(baabaa,aabaa){$b$}
\drawedge(aabaab,abaab){$a$}
\drawedge(aabaa,abaa){$a$}
\drawedge(abaa,baa){$a$}
\drawedge(baa,aa){$b$}
\drawedge(aa,a){$a$}
\drawedge(a,1){$a$}
\drawedge[ELside=r](ba,a){$b$}

\drawedge[ELside=r](babaab,abaab){$b$}
\drawedge(abaab,baab){$a$}
\drawedge(baab,aab){$b$}
\drawedge(aab,ab){$a$}
\drawedge(ab,b){$a$}
\drawedge[ELside=r](b,1){$b$}

\drawedge[ELside=r](babaa,abaa){$b$}
\drawedge[ELside=r](bab,ab){$b$}
\drawedge[ELside=r](ba,a){$b$}
  \end{picture}
  \caption{The two suffix codes which are classes of the equivalence
    $\theta_X$.}
  \label{figureSuffixMax}
\end{figure}
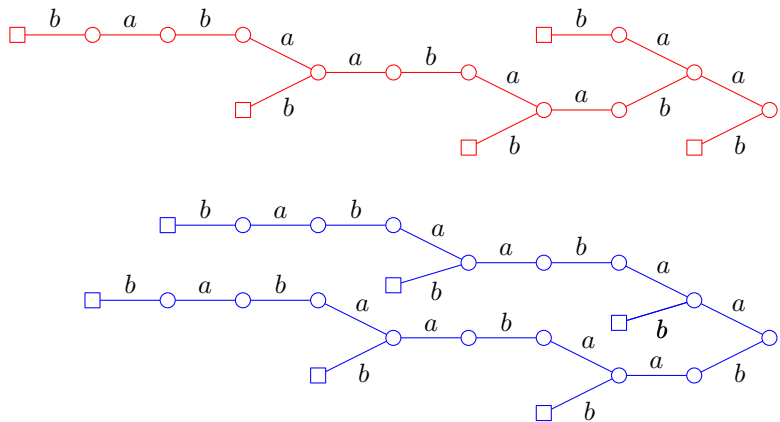
The following property relates the equivalence $\theta_X$ with the
right cosets of $H=\langle X\rangle$.

\begin{proposition}\label{propTheta}
  Let $X$ be a bifix code, let $P$ be the set of proper prefixes of
  $X$ and let $H$ be the subgroup generated by $X$. For any $p,q\in
  P$, $p\equiv q\bmod\theta_X$ implies $Hp=Hq$.
\end{proposition}

\begin{proof}
  Since $p\equiv q\bmod\theta_X$, there is a path from $1\otimes p$ to
  $1\otimes q$ in the incidence graph of $X$ of length $2n$, for some
  $n\ge0$.  If $n=1$, There is a word $s\in A^+$ such that $ps,qs\in
  X$. Then $p,q\in Hs^{-1}$ and thus $Hp=Hq$. The general case follows
  by induction.
\end{proof}

Let $\A=(P,1,1)$ be the literal automaton of $X^*$ (see
Section~\ref{sectionAutomata}).  We show that the equivalence
$\theta_X$ is compatible with the transitions of the automaton $\A$ in
the following sense.

\begin{lemma}\label{lemmaCompatible}
  Let $F$ be a Sturmian set. Let $X\subset F$ be a bifix code and let $P$
  be the set of proper prefixes of $X$.  Let $p,q\in P$ and $a\in
  A$. If $p\equiv q\bmod\theta_X$ and if $p\cdot a,q\cdot
  a\ne\emptyset$ in the literal automaton of $X^*$, then $p\cdot
  a\equiv q\cdot a\bmod\theta_X$.
\end{lemma}

\begin{proof}
Let $G$ be the incidence graph of $X$. 

Let $p,q\in P$ and $a\in A$ be such that $p\equiv q\bmod\theta_X$ and
$p\cdot a,q\cdot a\ne\emptyset$. If $p=1$, then $q=1$ and the
conclusion holds. Thus we may assume that $p\ne1, q\ne1$, and that
$p\ne q$.  Let $(1\otimes u_0,v_1\otimes 1,1\otimes
u_1,\ldots,v_n\otimes 1,1\otimes u_n)$ be a path in $G$ with $p=u_0$,
$u_n=q$.  The corresponding words in $X$ are $u_0v_1,
u_1v_1,u_1v_2,\dots,u_nv_n$. 
We may assume that the words $u_i$ are pairwise distinct,
and that the $v_i$ are pairwise distinct.  Moreover, since $p\cdot
a,q\cdot a\ne\emptyset$ there exist words $v,w$ such that $pav,qaw\in
X$.

The proof is in two steps. In the first step, we assume that  $v_1$
and $v_n$ both start with $a$.  In the second step, we show that this
condition is always fulfilled.

Assume that $v_1$ and $v_n$ begin with $a$. There are two cases.  

Case 1: Assume first that $pa,qa\in P$. Then $p\cdot a=pa$ and $q\cdot
a=qa$.  If all words $v_i$ begin with $a$, then clearly the
equivalence $pa\equiv qa\bmod\theta_X$ holds. Thus assume the
contrary, and let $i>1$ be minimal such that $v_i$ begins with a
letter distinct of $a$ and let $i\le j<n$ be maximal such that $v_j$
begins with a letter distinct of $a$. Then both words $u_{i-1}$ and
$u_j$ are right-special (since $u_{i-1}v_{i-1}$ starts with $u_{i-1}a$
and $u_{i-1}v_i$ starts with $u_{i-1}b$ for some letter $b\ne a$ and
similarly for $u_j$). But since $u_{i-1}$ and $u_j$ are in the same
trace on $P'$ of a connected component of $G$, Lemma~\ref{lemmaSuite} implies that $u_{i-1}=u_j$, that is
$i-1=j$. But this contradicts the inequality $i\le j$.

Case 2: Suppose now that $pa\in X$. This implies that $v_1=a$, since
$pv_1=u_0v_1$ is in $X$ and begins with $pa$.  If $v_n=aw$, then since
$v_1$ and $v_n$ , if they are distinct, are not prefix-comparable by
Lemma~\ref{lemmaSuite}, one $n=1$ and $w=1$.  If $v_n\ne aw$, then
$(v_1\otimes 1,1\otimes u_1,\ldots,v_n\otimes 1,1\otimes u_n,aw\otimes
1)$ is a path from $v_1\otimes 1$ to $aw\otimes 1$ (recall that
$u_naw=qaw\in X$). Lemma~\ref{lemmaSuite} implies that $v_1=a$ and
$aw$, if they are distinct, are not prefix-comparable.  Thus, one has
again $w=1$. In both cases, $qa\in X$ and therefore $p\cdot a=1=q\cdot
a$.

We now show that the assumption that $v_1$ begins with a letter
distinct of $a$ leads to a contradiction (the case where $v_n$ starts
with a letter distinct from $a$ is handled symmetrically). In this
case since $u_0v_1$ is in $X$ and $u_0av=pav\in X$, the word $u_0$ is
right-special.  Let $i$ be the largest integer such that $v_i$ begins
with a letter distinct of $a$ for $1\le i\le n$. If $i<n$, then $u_i$
is right-special. This contradicts Lemma~\ref{lemmaSuite}(iii), since
$u_0$ and $u_i$ are distinct (because $i\ge1$) elements of the trace
on $P'$ of a connected component of $G$.  If $i=n$, then $u_0$ and
$u_n$ are right-special since $u_nv_n\in X$ and $u_naw=qaw\in X$. We
obtain again a contradiction since $u_0$ and $u_n$ are distinct.
\end{proof}

\subsection{Coset automaton}

Let $F$ be a Sturmian set and
let $X\subset F$ be a bifix code.
We introduce a new automaton denoted $\B_X$ or $\B$
for short, and called the \emph{coset automaton}\index{coset
  automaton}\index{automaton!coset} of $X$. 
Let $R$ be the set of classes of $\theta_X$ with the class of $1$
still denoted $1$. 
The coset automaton of $X$  is the automaton
$\B_X=(R,1,1)$ with set of states $R$ and transitions induced by the
transitions of the literal automaton $\A=(P,1,1)$ of $X^*$. Formally,
for $r,s\in R$ and $a\in A$, one has $r\cdot a=s$ in the automaton
$\B$ if there exist $p$ in the class $r$ and $q$ in the class $s$ such
that $p\cdot a=q$ in the automaton $\A$.

Observe first that the definition is consistent since, by
Lemma~\ref{lemmaCompatible}, if $p\cdot a$ and $p'\cdot a$ are
nonempty and $p, p'$ are in the same class $r$, then $p\cdot a$ and
$p'\cdot a$ are in the same class. Since the class $p\cdot
a$ is uniquely defined, the automaton is indeed
deterministic. 

Observe next that if there is a path from $p$ to $p'$ in  the
automaton $\A$ labeled $w$ , then there is a path from the class $r$
of $p$ to the class $r'$ of $p'$ labeled $w$ in $\B_X$.

\begin{figure}[hbt]\label{figureBX}
  \centering
  \begin{picture}(40,10)(0,-5)
    \node(1)(0,0){$1$}
    \node[linecolor=red](2)(20,0){\textcolor{red}{$2$}}
    \node[linecolor=blue](3)(40,0){\textcolor{blue}{$3$}}
    \drawedge[curvedepth=3](1,2){$b$}\drawedge[curvedepth=3](2,1){$b$}
    \drawedge[curvedepth=3](2,3){$a$}\drawedge[curvedepth=3](3,2){$a$}
    \drawloop[loopangle=180](1){$a$}\drawloop[loopangle=0](3){$b$}
  \end{picture}
  \caption{The automaton $\B_X$.}
\end{figure}

\begin{example} For the code $X$ of
  Example~\ref{exampleIncidenceGraph}, the automaton $\B_X$ has three
  states. State $2$ is the red class, that is the class containing
  $b$, and state $3$ is the blue class containing $ba$. The bifix code 
  generating the submonoid  recognized by this automaton is
  $Z=a\cup b(ab^*a)^*b$. Observe that the word $bb$ is in $Z^*$ but it
  is not in $X^*$.
\end{example}

The following result shows that the coset automaton of $X$ is the
Stallings automaton of the subgroup generated by $X$.

\begin{lemma}\label{lemmaBidet} Let $F$ be a Sturmian set, and let 
  $X\subset F$ be a bifix code. The coset automaton $\B_X$ is
  reversible and describes the subgroup generated by $X$.
Moreover $X\subset Z$, where $Z$ is the bifix code
  generating the submonoid recognized by $\B_X$.
\end{lemma}

\begin{proof} Let $\A=(P,1,1)$ be the literal automaton of $X$ and Set
  $\B_X=(R,1,1)$.  Let $r,s\in R$ and $a\in A$ be such that $r\cdot
  a=s\cdot a$ is nonempty. Let $p,q\in P$ be elements of the classes
  $r$ and $s$ respectively, such that $p\cdot a,q\cdot a$ are
  nonempty. Then $pa,qa\in P\cup X$. To show that $\B_X$ is
  reversible, it is enough to show that $p\equiv q\bmod \theta_X$.

  Suppose first that $pa\in X$. Then $r\cdot a=s\cdot a=1$ and thus
  $qa\in X$ since $1$ is isolated mod $\theta_X$. Thus $p\equiv q\bmod
  \theta_X$.

  Suppose next that $pa,qa\in P$. Then there is a path $(1\otimes
  u_0,v_1\otimes 1,\ldots,v_n\otimes 1,1\otimes u_n)$ in the incidence
  graph $G$ of $X$, with $pa=u_0$ and $qa=u_n$. We may assume that the
  nodes of the path are pairwise distinct, except for a possible
  equality $u_0=u_n$.

If all the words $u_i$ end with $a$, then $p\equiv q\bmod\theta_X$.

Otherwise, let $i$ be minimal such that $u_i$ ends with a letter
distinct of $a$ and $j$, with $1\le i\le j<n$ be maximal such that
$u_j$ ends with a letter distinct of $a$. Then $v_i$ and $v_{j+1}$ are
left-special and they are distinct since $j+1>i$. This contradicts
Lemma~\ref{lemmaSuite}(iii) since $v_i$ and $v_{j+1}$ are distinct
elements of the same trace on the set $S'$ of proper nonempty suffixes
of $X$.

Thus the coset automaton is reversible.

Let $Z$ be the bifix code generating the submonoid recognized by
$\B_X$.
To show the inclusion $X\subset Z$, consider a word $x\in X$. There is
a path from $1$ to $1$ labeled $x$ in $\A$, hence also in $\B_X$.  Since
the class of $1$ modulo $\theta_X$ is reduced to $1$, this path in
$\B_X$ does not pass by $1$ except at its ends. Thus $x$ is in $Z$.

Let us finally show that the coset automaton describes the group
$H=\langle X\rangle$. By Proposition~\ref{propGeneratedGroup}, the subgroup 
described
by $\B_X$ is equal to $\langle Z\rangle$. Set $K=\langle Z\rangle$.
Since $X\subset Z$, we have $H\subset K$. To show the converse
inclusion, let us show by induction on the length of $w\in A^*$
that, for $p,q\in P$, there is a path from the class of $p$ to the
class of $q$ in $\B_X$ with label $w$ then $Hpw=Hq$. By
Proposition~\ref{propTheta}, this holds for $w=1$. Next,
assume
that it is true for $w$ and consider $wa$ with $a\in A$. Assume
that there are states $p,q,r\in P$ such that there
there is a path from the class of $p$ to the class of $q$ in $\B_X$ with
label $w$,
and an edge from the class of $q$ to the class of $r$ in $\B_X$ with
the label $a$. By induction
hypothesis, we have $Hpw=Hq$. Next, by definition of $\B_X$, there
is an $s\equiv q\bmod \theta_X$ such that $s\cdot a\equiv
r\bmod\theta_X$.
If $sa\in P$, then $s\cdot a=sa$, and
by Proposition~\ref{propTheta}, we have $Hs=Hq$ and $Hsa=Hr$.
Thus $Hpwa=Hqa=Hsa=Hr$. Otherwise, $sa\in X$ and $s\cdot a=r=1$
because the class of $1$ is a singleton. In this case, $Hsa=H=Hr$.
This property shows that if $z\in Z$, then
$Hz=H$, that is $z\in H$. Thus $Z\subset H$ and finally $H=K$.
\end{proof}

\subsection{Return words}

Let $F$ be a factorial set. For $u\in F$, define
\begin{displaymath}
  \Gamma_F(u)=\{z\in F\mid uz\in A^+u\cap F\}\,,\qquad
  \Gamma'_F(u)=\{z\in F\mid zu\in uA^+\cap F\}
\end{displaymath}
and
\begin{displaymath}
  R_F(u)=\Gamma_F(u)\setminus \Gamma_F(u)A^+\,,
\qquad
  R'_F(u)=\Gamma'_F(u)\setminus A^+\Gamma'_F(u)
\,.
\end{displaymath}
When $F=F(x)$ for an infinite word $x$, the sets $\Gamma_F(u)$ and
$R_F(u)$ are respectively the set of \emph{right return
  words}\index{right return word} to $u$ and \emph{first right return
  words}\index{return word!first right}\index{first right return word}
to $u$ in $x$, and $\Gamma'_F(u)$ and $R'_F(u)$ are respectively the
set of \emph{left return words}\index{left return word} to $u$ and
\emph{first left return words}\index{return word!first
  left}\index{first left return word} to $u$ in $x$.  The relation between $R_F(u)$
and $R'_F(u)$ is simply
\begin{equation}
  uR_F(u)=R'_F(u)u\,.\label{eqAutomo}
\end{equation}
Words in the set $uR_F(u)=R'_F(u)u$ are called \emph{complete return
  words}\index{complete return words}
in~\cite{JustinVuillon2000}. When there is no ambiguity, we will call
the (first) right return words simply the (first) return words,
omitting the `right' specification.

\begin{example}
  Let $F$ be the Fibonacci set. The sets
  $R_F(u)$ and $R'_F(u)$ are given below for the first small words of
  $F$.
  \begin{displaymath}
    \begin{array}{c|c|c|c|c|c|c|c|c|c|c}
      u     &1 &a &b  &aa   &ab &ba &aab  &aba&baa&bab\\ \hline
      \multirow{2}*{$R_F(u)$}%
       &a &a &ab &baa  &ab &ba &aab  &ba &baa &aabab\\ 
       &b &ba&aab&babaa&aab&aba&abaab&aba&babaa&aabaabab\\ \hline
      \multirow{2}*{$R'_F(u)$}%
       &a &a &ba &aab  &ab &ba &aab  &ab &baa&babaa\\ 
       &b &ab&baa&aabab&aba&baa&aabab&aba&baaba&babaabaa\\ 
    \end{array}
  \end{displaymath}
\end{example}

Vuillon has shown in~\cite{Vuillon2001} that $x$ is a Sturmian word if
and only if $R'_F(u)$ has exactly two elements for every factor $u$ of
$x$. Another proof of this result is given by Justin and Vuillon
in~\cite{JustinVuillon2000}. 

In fact, they show in~\cite{JustinVuillon2000} the following
theorem. Since this result is not exactly formulated in
\cite{JustinVuillon2000} as stated here, we show how it follows easily
from their article.

\begin{theorem}\label{propositionGamma}
  Let $F$ be a Sturmian set. For any word $u\in F$, the set $R_F(u)$
  (and the set $R'_F(u)$) is a basis of the free group $A^\circ$.
\end{theorem}
By Equation~\eqref{eqAutomo}, the sets  $R_F(u)$ and $R'_F(u)$ are
conjugate in the free group. Conjugacy by an element $u$ is an automorphism of the
free group. It follows that $R_F(u)$ is a basis if and only if
$R'_F(u)$ is a basis. Thus, it suffices to prove the claim for
$R'_F(u)$. We quote the following result
of~\cite[Theorem 4.4, Corollaries~4.1 and~4.5]{JustinVuillon2000},
with the notations of Section~\ref{subsection-episturmian}.

\begin{proposition}\label{propJustinVuillon}
  Let $s$ be a standard strict episturmian word over $A$, let
  $\Delta=a_0a_1\cdots$ be its directive word, and let $(u_n)$ be its
  sequence of palindrome prefixes. 
  \begin{enumerate}
  \item[\upshape{(i)}] The first left return words to $u_n$ are the
    words $\psi_{a_0\cdots a_{n-1}}(a)$ for $a\in A$.
  \item[\upshape{(ii)}] For each factor $u$ of $s$, there exist a word
    $z$ and an integer $n$ such that the first left return words to
    $u$ are the words $zyz^{-1}$, where $y$ ranges over the first left
    return words to $u_n$.
  \end{enumerate}
\end{proposition}

\begin{proofof}{of Theorem~\ref{propositionGamma}}
  We may assume that $F = F(s)$ for some standard and strict
  episturmian word $s$. By Proposition~\ref{propJustinVuillon}(i), the
  set of first left return words to $u_n$ is the image of the alphabet
  by the endomorphism $\psi_{a_0\cdots a_{n-1}}$. It is easily seen
  that these endomorphisms define automorphisms of the free group. We
  deduce that the set of first left return words to $u_n$ is a basis
  of the free group on $A$.  By
  Proposition~\ref{propJustinVuillon}(ii), the set of first left
  return words to $u$ is a basis, too.  This ends the proof.
\end{proofof}

\subsection{Proof of the main result}
Some preliminary results are needed for the proof of
Theorem~\ref{theoremGroups}.

\begin{proposition}\label{propositionHcapF}
  Let $F$ be a Sturmian set and let $X\subset F$ be a finite
  $F$-maximal bifix code.  Then $\langle X\rangle\cap F=X^*\cap F$.
\end{proposition}

\begin{proof}
  We have $X^*\cap F\subset \langle X\rangle\cap F$.
  To show the converse inclusion, consider the bifix code $Z$
  generating the submonoid re\-cog\-ni\-zed by the coset automaton
  $\B_X$ associated to $X$.

  Let us show that $Z\cap F=X$.  By Lemma~\ref{lemmaBidet}, we have
  $X\subset Z$ and thus $X\subset Z\cap F$. Since $X$ is an
  $F$-maximal bifix code, this implies that $X=Z\cap F$.

  Since any reversible automaton is minimal and since the automaton
  $\B_X$ is reversible by Lemma~\ref{lemmaBidet}, it is equal to the
  minimal automaton of $Z^*$.  Let $K$ be the subgroup generated by
  $Z$. By Proposition~\ref{lemmaExercise612}, we have $K\cap A^*=Z^*$.

  This shows that
  \begin{displaymath}
    \langle X\rangle\cap F\subset K\cap F= K\cap A^*\cap F=Z^*\cap F=
    X^*\cap F\,. 
  \end{displaymath}
  The first inclusion holds because $X\subset Z$ implies $\langle
  X\rangle\subset K$.  The last equality follows from the fact that if
  $z_1\cdots z_n\in F$ with $z_1,\ldots ,z_n\in Z$, then each $z_i$ is
  in $F$ hence in $Z\cap F=X$.  Thus $\langle X\rangle\cap F\subset
  X^*\cap F$, which was to be proved.
\end{proof}

We will use the following consequence of
Proposition~\ref{propositionHcapF}. 

\begin{corollary}\label{corollaryCosets}
  Let $F$ be a Sturmian set and let $X\subset F$ be a finite
  $F$-maximal bifix code. 
  Each right coset of the subgroup $\langle X\rangle$ generated by $X$
  contains at most one right-special proper prefix of $X$.
\end{corollary}

\begin{proof} Set $H=\langle X\rangle$. Let $Q$ be the set of those
  proper prefixes of the words of $X$ which are right-special.

  Let us show that if $p,q\in Q$ belong to the same right coset, then
  $p=q$. We may assume that $p=uq$. Since $Hp=Hq$, one has
  $Huq=Hq$. Consequently, $Hu=H$ and thus $u\in H$.  By
  Proposition~\ref{propositionHcapF}, since $u\in F$, this implies
  that $u\in X^*$ and thus $u=1$ since $p$ is a proper prefix of $X$.
\end{proof}

\begin{proofof}{of Theorem~\ref{theoremGroups}}
  
  Assume first that $X$ is an $F$-maximal bifix code of $F$-degree
  $d$. Let $P$ be the set of proper prefixes of $X$.  Let
  $Q$ be the set of words in $P$ which are right-special. Let $H$ be
  the subgroup generated by $X$.

By Lemma~\ref{lemmadSpecial} there is a right-special
word $u$ such that
$\pars_X(u)=d$. The $d$ suffixes of $u$ which are in $P$ are the
elements of $Q$. By Theorem~\ref{theoremDegree}, the word $u$ is not
an internal factor of $X$.

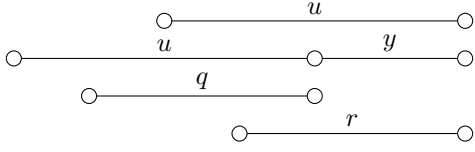
\begin{figure}[hbt]
\centering
\gasset{Nadjust=wh,AHnb=0}
\begin{picture}(50,15)(0,-5)
\node(uh)(20,10){}\node(endh)(60,10){}
\node(u)(0,5){}\node(v)(40,5){}\node(end)(60,5){}
\node(q)(10,0){}\node(qend)(40,0){}\node(r)(30,-5){}\node(endr)(60,-5){}
\drawedge(uh,endh){$u$}
\drawedge(u,v){$u$}\drawedge(v,end){$y$}\drawedge(q,qend){$q$}
\drawedge(r,endr){$r$}
\end{picture}
\caption{A word $y\in R_F(u)$.}\label{figGamma}
\end{figure}

Let
\begin{displaymath}
  V=\{v\in A^\circ\mid Qv\subset HQ\}\,.
\end{displaymath}
Any $v\in V$ defines a permutation of $Q$. Indeed, suppose that for
$p,q\in Q$, one has $pv,qv\in Hr$ for some $r\in Q$. Then $rv^{-1}$ is
in $Hp\cap Hq$.  This forces $Hp=Hq$ and thus $p=q$ by
Corollary~\ref{corollaryCosets}.

The set $V$ is a subgroup of $A^\circ$. Indeed, $1\in V$. Next, let $v\in V$. Then
for any $q\in Q$, since $v$ defines a permutation of $Q$, there
is a $p\in Q$ such that $pv\in
Hq$. Then $q v^{-1} \in Hp$. This shows that $v^{-1}\in V$.
Next, if $v,w\in V$, then $Qvw\subset HQw\subset HQ$ and thus $vw\in
V$.

We show that the set $R_F(u)$ is contained in $V$. Indeed, let $q\in
Q$ and $y\in R_F(u)$. Since $q$ is a suffix of $u$, $qy$ is a suffix
of $uy$, and since $uy$ is in $F$ (by definition of $R_F(u)$),
also $qy$ is in $F$.  The fact that $X$ is $F$-maximal 
implies that
there is a word $r\in P$ such that $qy\in X^*r$. We verify that the
word $r$ is a suffix of $u$.  Since $y\in R_F(u)$, there is a word
$y'$ such that $uy=y'u$. Consequently, $r$ is a suffix of $y'u$, and
in fact the word $r$ is a suffix of $u$. Indeed, one has $|r|\le |u|$
since otherwise $u$ is in $I(X)$ and this is not the case. Thus we
have $r\in Q$ (see Figure~\ref{figGamma}). Since $X^*\subset H$ and
$r\in Q$, we have $qy\in HQ$. Thus $y\in V$.

By Theorem~\ref{propositionGamma}, the group generated by $R_F(u)$ is
$A^\circ$.  Since $R_F(u)\subset V$, and since $V$ is a subgroup of
$A^\circ$, we have $V=A^\circ$. Thus $Qw\subset HQ$ for any $w\in
A^\circ$.  Since $1\in Q$, we have in particular $w\in HQ$.  Thus
$A^\circ=HQ$. Since $\Card(Q)=d$, and since the right cosets $Hq$ for
$q\in Q$ are pairwise disjoint, this shows that $H$ is a subgroup of
index $d$. By Theorem~\ref{theoremBifixd+1} and in view of Schreier's
Formula, $X$ is a basis of $H$.


Assume conversely that the bifix code $X\subset F$ is a basis of the
group $H=\langle X\rangle$ and that $\langle X\rangle$ has index $d$. Since $X$ is a
basis, by Schreier's Formula, we have $\Card(X)= (k-1)d+1$, where
$k=\Card(A)$. The case $k=1$ is straightforward; thus we assume $k\ge2$. By
Theorem~\ref{theoremCompletion}, there is a finite $F$-maximal bifix
code $Y$ containing $X$. Let $e$ be the $F$-degree of $Y$. By the
first part of the proof, $Y$ is a basis of a subgroup $K$ of index $e$
of $A^\circ$.  In particular, it has $(k-1)e+1$ elements. Since
$X\subset Y$, we have $(k-1)d+1\le (k-1)e+1$ and thus $d\le e$. On the
other hand, since $H$ is included in $K$, $d$ is a multiple of $e$ and
thus $e\le d$. We conclude that $d=e$ and thus that $X=Y$.
\end{proofof}

\begin{example}\label{exampleCodeGiuseppina}
  Let $F$ be the Fibonacci set. Let $X\subset F$ be the
  bifix code shown on
  Figure~\ref{figureCodeGiuseppina}. The right-special proper prefixes
  of the words of $X$ are the four suffixes of $aba$ and are indicated in
  black on the figure.
\begin{figure}[hbt]
\centering
\gasset{AHnb=0,Nadjust=wh}
\begin{picture}(60,35)(0,3)
\node[fillgray=0](1)(0,20){}
\node[fillgray=0](a)(10,30){}\node(b)(10,10){}
\node[Nmr=0](aa)(20,35){}\node(ab)(20,25){}\node[fillgray=0](ba)(20,10){}
\node[fillgray=0](aba)(30,25){}\node(baa)(30,15){}\node(bab)(30,5){}
\node(abaa)(40,30){}\node[Nmr=0](abab)(40,20){}
\node[Nmr=0](baab)(40,15){}\node[Nmr=0](baba)(40,5){}
\node(abaab)(50,30){}\node[Nmr=0](abaaba)(60,30){}
\drawedge(1,a){$a$}\drawedge[ELside=r](1,b){$b$}
\drawedge(a,aa){$a$}\drawedge(a,ab){$b$}\drawedge(b,ba){$a$}
\drawedge(ab,aba){$a$}\drawedge(ba,baa){$a$}\drawedge[ELside=r](ba,bab){$b$}
\drawedge(aba,abaa){$a$}\drawedge[ELside=r](aba,abab){$b$}
\drawedge[ELside=r](baa,baab){$b$}\drawedge(bab,baba){$a$}
\drawedge(abaa,abaab){$b$}\drawedge(abaab,abaaba){$a$}
\end{picture}
\caption{An $F$-maximal bifix code of $F$-degree 4.}\label{figureCodeGiuseppina}
\end{figure}
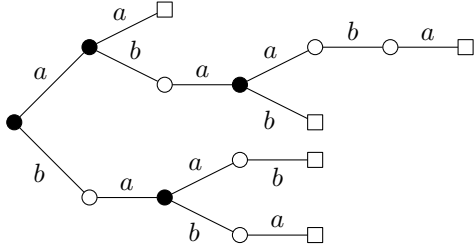
The states of the coset automaton are the sets $\{1\}$, $\{a, bab,
abaab\}$, $\{aba, b, baa\}$ and $\{ba,ab,abaa\}$. The code $X$ has
$F$-degree~$4$. Each state is represented by its right-special factor in
Figure~\ref{figureGroup}.
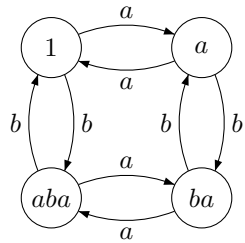
\begin{figure}[hbt]
\centering
\begin{picture}(30,30)(0,-3)
\node(1)(0,20){$1$}\node(a)(20,20){$a$}
\node(aba)(0,0){$aba$}\node(ba)(20,0){$ba$}
\drawedge[curvedepth=3](1,a){$a$}\drawedge[curvedepth=3](a,1){$a$}
\drawedge[curvedepth=3](1,aba){$b$}\drawedge[curvedepth=3](aba,1){$b$}
\drawedge[curvedepth=3](a,ba){$b$}\drawedge[curvedepth=3](ba,a){$b$}
\drawedge[curvedepth=3](aba,ba){$a$}\drawedge[curvedepth=3](ba,aba){$a$}
\end{picture}
\caption{The associated coset automaton.}\label{figureGroup}
\end{figure}
\end{example}

We end this section with a combinatorial consequence of
Theorem~\ref{theoremGroups}.

\begin{proposition}\label{theoremSigmaL}
  Let $F$ be a Sturmian set on an alphabet with $k$ letters and let
  $X\subset F$ be a finite $F$-maximal bifix code of $F$-degree
  $d$. Let $P$ (resp. $S$) be the set of proper prefixes
  (resp. suffixes) of  $X$. Then
\begin{displaymath}
\sum_{x\in X}|x|=\Card(P)+\Card(S)+(k-2)d.
\end{displaymath}
\end{proposition}

We will use the following proposition, of independent interest. 

\begin{proposition}\label{lemmaCosets}
  Let $F$ be a Sturmian set and let $X\subset F$ be a finite
  $F$-maximal bifix code of $F$-degree $d$. The coset automaton $\B_X$
  is a group automaton with $d$ states. Each state of $\B_X$ other
  than $1$ is an $F$-maximal suffix code.
\end{proposition}

\begin{proof}
  Let $\A=(P,1,1)$ be the literal automaton recognizing $X^*$ and let
  $\B=(R,1,1)$ be the coset automaton of $X$.  By
  Lemma~\ref{lemmaBidet}, the automaton $\B$ is reversible and
  describes the subgroup $H$ generated by $X$.

  By Theorem~\ref{theoremGroups}, the subgroup $H$ has index $d$ in
  $A^\circ$. Since $\B$ is reversible, it is minimal.
  Propositions~\ref{propGroupCode} and~\ref{propStallings} show that
  $\B$ is a group automaton. Its number of states is $d$ since a group
  automaton which describes a subgroup of index $d$  has $d$ states by
  Proposition~\ref{propositionGroupAutomaton}.

Finally, consider $r\in R$ and let $X_r=\{p\in P\mid 1\cdot p=r\}$.
Let us show that any $w\in F$ is suffix-comparable with an element of
$X_r$. We may assume that $w$ is longer than any word of $X$. Since
$\B_X$ is a group automaton, there is an $u\in R$ such that $u\cdot
w=r$. Since $w$ is longer than any word of $X$, the path from $u$ to
$r$ labeled $w$ passes through state~$1$.  Thus $w$ has a parse
$(s,x,p)$ such that $1\cdot p=r$ and thus $w$ has a suffix in
$X_r$. This shows that $X_r$ is an $F$-maximal suffix code.
\end{proof}

\noindent
Note that the fact that the set $P$ of nonempty proper prefixes of $X$
is a disjoint union of $d-1$ $F$-maximal suffix codes is also a
consequence of the dual statement of
Theorem~\ref{theoremDisjointUnion}.\medskip

\begin{proofof}{of Proposition~\ref{theoremSigmaL}}
  Let $H$ be the subgroup generated by $X$.  By
  Theorem~\ref{theoremGroups}, the set $X$ is a basis of $H$ and the
  index of $H$ is equal to $d=d_F(X)$.  Let $G$ be the incidence graph
  of $X$.   Let $E$ be
  the set of edges of $G$. One has
  \begin{displaymath}
    \Card(E)=\sum_{x\in X}(|x|-1)=\sum_{x\in X}|x|-\Card(X)=\sum_{x\in
      X}|x|-(k-1)d-1\,. 
  \end{displaymath}
  By Proposition~\ref{lemmaCosets}, the classes of $\theta_X$ are the
  set $\{1\}$ and $d-1$ $F$-maximal suffix codes denoted $P_i$, for
  $i=1,\ldots,d-1$.  Each of the latter is the trace on $P\setminus1$
  of a connected component $C_i$ of $G$. Let $G_i$ be the subgraph of $G$
  induced by its connected component $C_i$.  By
  Lemma~\ref{lemmaSuite}, $G_i$ is a tree.

  Similarly, let $S_i$ be the trace on $S\setminus1$ of the connected
  component $C_i$.  Let $E_i$ be the set of edges of $G_i$. Since
  $G_i$ is a tree, we have $\Card(E_i)=\Card(P_i)+\Card(S_i)-1$ for
  $i=1,\ldots, d-1$.  Finally
  \begin{align*}
    \Card(E)&=\sum_{i=1}^{d-1}\Card(E_i)=\sum_{i=1}^{d-1}(\Card(P_i)+\Card(S_i)-1)\\
    &=\Card(P\setminus1)+\Card(S\setminus1)-(d-1)\,,
  \end{align*}
  whence the result. 
\end{proofof}


\section{Syntactic groups}\label{sectionSyntacticGroups}

Let $F$ be a recurrent subset of $A^*$.
In this section, we introduce the notion of $F$-group of a bifix code
$X\subset F$ of finite $F$-degree. It is a permutation group of degree
$d_F(X)$.  We investigate the relation between this group and the
notion of group of a maximal bifix code
(Theorem~\ref{theoremGroupCodes}).  We use Theorem~\ref{theoremGroups}
to prove a new result on the syntactic groups of bifix codes: any
transitive permutation group $G$ of degree $d$ and with $k$ generators
is a syntactic group of a bifix code with $(k-1)d+1$ elements
(Theorem~\ref{newTheorem}).

\subsection{Preliminaries}

We first recall the basic terminology on groups in monoids
(see~\cite{BerstelPerrinReutenauer2009} for a more detailed
exposition). We are mainly concerned with monoids of maps from a set
into itself. The maps considered in this section are partial maps.

Let $M$ be a monoid.
A \emph{group in}\index{group!in a monoid} $M$ is
a subsemigroup of $M$ which is isomorphic to a group. Note that the
neutral element of a group contained in $M$ needs not be equal to the
neutral element of $M$. 

A group in $M$ is \emph{maximal}\index{group!maximal}\index{maximal
  group} if it not included in another group in~$M$.

\begin{proposition}
  Let $G$ be a group in a monoid $M$ of partial maps from a set $Q$
  into itself. All elements of $G$ have the same image $I$. The
  restriction of the elements of $G$ to $I$ is a faithful
  representation of $G$ as a permutation group on $I$.
\end{proposition}

\begin{proof}
  Two elements $g,h\in G$ have the same image. Indeed, let $k$ be the
  inverse of $g$ in $G$. Then $h=hkg$ and thus the image of $h$ is
  contained in the image of $g$. The converse inclusion is shown
  analogously.  Then $G$ is a permutation group on the common image
  $I$ of its elements.  Indeed, let $e$ be the neutral element of
  $G$. Then for any $p\in I$, let $q\in Q$ be such that $qe=p$. Then
  $pe=qe^2=qe=p$. This shows that $e$ is the identity on $I$. Next,
  for any $g\in G$ the inverse $k$ of $g$ is such that $gk=kg=e$. Thus
  $g$ is a permutation on $I$.

  Let $g,g'\in G$ be such that they have the same restriction to
  $I$. Then for each $p\in Q$, $p(eg)=(pe)g=(pe)g'=p(eg')$ since
  $pe\in I$.  Since $eg=g$ and $eg'=g'$, we obtain $g=g'$. This shows
  that the representation of $G$ by permutations on $I$ is faithful.
\end{proof}

Let $G$ be a group in a monoid of maps from $Q$ into itself as above.
The \emph{canonical}\index{canonical
  representation}\index{representation!canonical} representation of
$G$ by permutations is the restriction of the maps in $G$ to their common
image.

A \emph{syntactic group}\index{syntactic group}\index{group!syntactic}
of a prefix code $X$ is the canonical representation by permutations
of a maximal group in the monoid of transitions of the minimal
automaton $\A(X^*)$ of $X^*$.

Let $X$ be a prefix code and let $\A=\A(X^*)$. A syntactic group $G$
of $X$ is called \emph{special}\index{special syntactic group} if
$\varphi_\A^{-1}(G)$ is a cyclic submonoid of $A^*$.  In particular a
special syntactic group is cyclic.

The \emph{degree}\index{degree of a permutation
  group}\index{permutation group!degree} of a permutation group $G$ on
a set $R$ is the cardinality of $R$. Recall that the group $G$ is
\emph{transitive}\index{transitive permutation
  group}\index{permutation group!transitive} if for any $r,s\in R$
there is some $g\in G$ such that $rg=s$.

A permutation group $G$ on a set $R$ and a permutation group $H$ on a
set $S$ are \emph{equivalent}\index{equivalent permutation
  group}\index{permutation group!equivalent} if there exists a
bijection $\beta:R\to S$ and an isomorphism $\sigma:G\to H$ such that,
for all $g\in G$ and $r\in R$, one has
\begin{displaymath}
  \beta(rg)=\beta(r)\sigma(g)\,,
\end{displaymath}
in other terms, if the diagram of Figure~\ref{commdiagram} is commutative
for all $g\in G$.

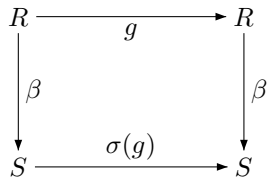
\begin{figure}[hbt]
\centering
\gasset{Nframe=n,Nadjust=wh}
\begin{picture}(30,25)
\node(1)(0,20){$R$}\node(2)(30,20){$R$}
\node(3)(0,0){$S$}\node(4)(30,0){$S$}
\drawedge[ELside=r](1,2){$g$}\drawedge(3,4){$\sigma(g)$}
\drawedge(1,3){$\beta$}\drawedge(2,4){$\beta$}
\end{picture}
\caption{Equivalent permutation groups.}\label{commdiagram}
\end{figure}

Let us recall the notation concerning Green relations in a monoid $M$
(see~\cite{BerstelPerrinReutenauer2009}).
We denote by $\RR$ the equivalence in $M$ defined by $m\RR n$ if
$m,n$ generate the same right ideal, i.e. if $mM=nM$. We denote
by $R(m)$ the $\RR$-class of $m$. 

Symmetrically,
we denote by $\LL$ the equivalence defined by $m\LL n$ if $m,n$
generate the same left ideal, i.e. if $Mm=Mn$. We denote by $L(m)$
the $\LL$-class of $m$.

It is well known that the equivalences $\LL$ and $\RR$ commute.
We denote by $\D$ the equivalence $\LL\RR=\RR\LL$. Finally, we denote
by $\H$ the equivalence $\LL\cap \RR$.

A $\D$-class $D$ is \emph{regular}\index{D-class,
  regular@$\D$-class!regular} if it contains an idempotent. In this
case, there is at least an idempotent in each $\LL$-class of $D$ and
in each
$\RR$-class of $D$.  The following statement is known as \emph{Clifford and
  Miller's Lemma}\index{Clifford and Miller's Lemma}.  For $m,n\in M$,
one has $mn\in R(m)\cap L(n)$ if and only if $R(n)\cap L(m)$ contains
an idempotent.

Assume that $M$ is a monoid of maps from a finite set $Q$ into itself.

If $m,n\in M$ are $\LL$-equivalent, then
they have the same image. If they are $\RR$-equivalent, then they have the same
nuclear equivalence (the nuclear equivalence of a partial map $m$ from
$Q$ into itself is the partial equivalence, for which $p,q\in Q$ are
equivalent if $m$ is defined on $p$ and $q$ and $pm=qm$).

If $m,n\in M$ are $\H$-equivalent, they have the
same
image and the same nuclear equivalence. The converse is not true
but it holds in the following important particular case.

\begin{proposition}\label{propIdempotent}
  Let $M$ be a monoid of maps from a finite set $Q$ into itself.  Let
  $e\in M$ be an idempotent. An element $m$ of $M$ is in the
  $\H$-class of $e$ if and only if it has the same nuclear equivalence
  and the same image as $e$.
\end{proposition}

\begin{proof}
  If $m$ and $e$ are $\H$-equivalent, they have the same nuclear
  equivalence $\rho$ and the same image $I$. Conversely, we have
  $me=m$ since $e$ is the identity on its image $I$. For any $p\in Q$,
  $pe^2=pe$ implies that $p$ and $pe$ are in the same class of $\rho$.
  This implies that $pem=pm$.  Thus $em=m$.

Finally, the restriction of $m$ to $I$ is a permutation. Indeed,
$pm=qm$ for $p,q\in I$ implies $pe=qe$ which forces $p=q$.
Let $k>0$ be such that the restriction of $m^k$ to $I$ is the
identity.
Then $m^k$ and $e$ are two idempotents with the same nuclear
equivalence and the same image. This implies that they are equal.
Thus $m$ and $e$ are in the same $\H$-class.
\end{proof}

Let $F$ be a recurrent set and let $X\subset F$ be a bifix code of
finite $F$-degree $d$. Let $\A=(Q,1,1)$ be a simple automaton
recognizing $X^*$. We set $\varphi=\varphi_\A$ and we denote by
$M$ the transition monoid $\varphi(A^*)$ of $\A$.


For a word $w$, we denote by $\Im(w)$ the \emph{image}\index{image of
  a word}\index{word!image} of $w$ with respect to $\A$, that is the
set $\Im(w)=\{p\cdot w \mid p\in Q\}$. The
\emph{rank}\index{rank!word}\index{word!rank} of $w$ (with respect to
the automaton $\A$) is the number $\rank(w)=\Card(\Im(w))$.  Then
$\Im(w)$ is also the image of the map $\varphi(w)$ (recall that the
action of $M$ is on the right of the elements of $Q$), and the rank of
$w$ is also the rank of $\varphi(w)$. Clearly $\rank(uwv)\le\rank(w)$
for all $u,w,v$.

\begin{proposition}\label{propDclass}
  The set of elements of $\varphi(F)$ of rank $d$ is included in a
  regular $\D$-class of $M$.
\end{proposition}

We use the following lemmas.



\begin{lemma}\label{lemmaRank1}
  A word $w\in F$ which has $d$ parses with respect to $X$ has rank
  $d$ with respect to $\A$. Moreover, $\Im(w)$ is the set of
  states $1\cdot p$ for all $p$ such that there is a parse $(s,x,p)$
  of $w$. For all $q\in\Im(w)$, there is a unique proper prefix $p$ of
  $P$ which is a suffix of $w$, and such that $q=1\cdot p$.
\end{lemma}

\begin{proof}
  Consider first two states $q,r\in Q$ and suppose that $q\cdot
  w=r$. Since $\A$ is simple, it is trim. Consequently there exist two
  words $u,v$ such that $1\cdot u=q$ and $r\cdot v=1$. It follows that
  $uwv\in X^*$. Since $w$ has $d$ parses, by
  Theorem~\ref{theoremDegree} it is not an internal factor of a word
  in $X$. Thus there is a parse $(s,x,p)$ of $w$ such that $us,pv\in
  X^*$. Then $r=1\cdot p$. The relation $r\to (s,x,p)$ is a
  function. Indeed, let us  show that if $(s,x,p)$ and $(s',x',p')$ are two
  distinct parses of $w$, then $1\cdot p\ne 1\cdot p'$. Assume the
  contrary. Then we have $pv,p'v\in X^*$ for the same word $v$.  Since
  $p,p'$ are suffixes of $w$, they are suffix-comparable and thus
  $p=p'$ since $X$ is bifix. This is impossible if the parses are
  distinct. Of course, the function $r\mapsto (s,x,p)$ is injective
  since $\A$ is deterministic.

  Conversely, let $(s,x,p)$ be a parse of $w$. Since $X$ is an
  $F$-maximal bifix code, there exist by Theorem~\ref{theoremEquivMax}
  words $u,v$ such that $us,pv\in X^*$. Thus we have $1\cdot us=1\cdot
  x=1\cdot pv=1$.  Consequently $(1\cdot u)\cdot w=1\cdot usxp=1\cdot
  xp=1\cdot p$.  This shows that $1\cdot p\in \Im(w)$.

\end{proof}

\begin{lemma}\label{lemmaRank2}
  Let $u\in F$ be a word. If $\rank(u)=d$, then $\rank(uv)= d$ for all
  $v$ such that $uv\in F$.
\end{lemma}

\begin{proof}
  Since $X$ is $F$-thin, there exists $w\in F$ which is not a
  factor of a word in $X$. This word $w$ has $d$ parses.  Assume
  $uv\in F$. Since $F$ is recurrent, there exists a word $t$ such that
  $uvtw\in F$. Then $uvtw$ also has $d$ parses. By
  Lemma~\ref{lemmaRank1}, this implies that the rank of $uvtw$ is
  $d$. Since $d=\rank(uvtw)\le\rank(uv)\le\rank(u)=d$, one has
  $\rank(uv)=d$.
\end{proof}

\medskip

\begin{proofof}{of Proposition~\ref{propDclass}}
  Let $u,v\in F$ be two words of rank $d$. Set $m=\varphi(u)$ and
  $n=\varphi(v)$.  Let $w$ be such that $uwv\in F$.  We show first
  that $m\RR \varphi(uwv)$ and $n\LL \varphi(uwv)$.

  For this, let $t$ be such that $uwvtu\in F$.  Set $z=wvtu$.  By
  Lemma~\ref{lemmaRank2}, the rank of $uz$ is $d$. Since
  $\Im(uz)\subset\Im(z)\subset \Im(u)$, this implies that the images
  are equal. Consequently, the restriction of $\varphi(z)$ to $\Im(u)$
  is a permutation. Since $\Im(u)$ is finite, there is an integer
  $\ell\ge 1$ such that $\varphi(z)^\ell$ is the identity on
  $\Im(u)$. Set $e=\varphi(z)^\ell$ and $s=tuz^{\ell-1}$. Then,
  since $e$ is the identity on $\Im(u)$, one has $m=me$. Thus
  $m=\varphi(uwv)\varphi(s)$, and since $\varphi(uwv)=m\varphi(wv)$,
  it follows that $m$ and $\varphi(uwv)$ are
  $\RR$-equivalent. 

  Similarly $n$ and $\varphi(uwv)$ are $\LL$-equivalent. Indeed, set
  $z'=tuwv$. Then $\Im(vz')\subset\Im(z')\subset\Im(v)$.  Since $vz'$
  is a factor of $z^2$ and $z$ has rank $d$, it follows that
$d=\rank(z^2)\le\rank(vz')\le\rank(v)=d$. Therefore, $vz'$  has rank $d$ and
  consequently the images 
 $\Im(vz')$, $\Im(z')$ and $\Im(v)$
are equal. There is an integer $\ell'\ge 1$
  such that $\varphi(z')^{\ell'}$ is the identity on $\Im(v)$. Set
  $e'=\varphi(z')^{\ell'}$.  Then
  $n=ne'=n\varphi(z')^{\ell'-1}\varphi(tuwv)=nq\varphi(uwv)$, with
  $q=\varphi(z')^{\ell'-1}\varphi(t)$. Since
  $\varphi(uwv)=\varphi(uw)n$, one has $n\LL\varphi(uwv)$.  Thus $m,n$
  are $\D$-equivalent, and $\varphi(uwv)\in R(m)\cap L(n)$.

  Set $p=\varphi(wv)$. Then $p=\varphi(w)n$ and, with the previous
  notation, $n=ne'=nq\varphi(u)p$, so $L(n)=L(p)$.  Thus
  $mp=\varphi(uwv)\in R(m)\cap L(p)$, and by Clifford and Miller's
  Lemma, $R(p)\cap L(m)$ contains an idempotent. Thus the $\D$-class
  of $m$, $p$ and $n$ is regular.
\end{proofof}

\subsection{Group of a bifix code}

Let $M$ be a monoid.
The $\H$-class of an idempotent $e$ is denoted $H(e)$.
It is the maximal group contained
in $M$ and containing $e$.

All groups $H(e)$ for $e$ idempotent in a regular $\D$-class $D$ are
isomorphic.  The \emph{structure
  group}\index{group!structure}\index{structure group} (or
Sch\"utzenberger group) of $D$ is any one of them.  When $M$ is a
monoid of maps from a set $Q$ into itself, the canonical
representations of the groups $H(e)$ are equivalent permutation
groups. See \cite[Proposition~9.1.9]{BerstelPerrinReutenauer2009}.  We
then also consider the structure group as a permutation group.

Let $F$ be a recurrent set and let $X\subset F$ be a bifix code of
finite $F$-degree~$d$.  Let $\A=(Q,1,1)$ be a simple automaton
recognizing $X^*$. Set $\varphi=\varphi_A$.  The
structure group of the $\D$-class of elements of rank $d$ of
$\varphi(F)$ is a permutation group of degree $d$.  By Proposition
9.5.1 in~\cite{BerstelPerrinReutenauer2009}, this group does not
depend on the choice of the simple automaton $\A$ recognizing
$X^*$. It is called the $F$-\emph{group}\index{F-group@$F$-group} of
the code $X$ and denoted $G_F(X)$.

When $F=A^*$, the group $G_F(X)$ is the group $G(X)$ of the code $X$
defined in~\cite{BerstelPerrinReutenauer2009}. Indeed, in this case,
the $\D$-class of elements of rank $d$ coincides with the minimal
ideal of the monoid $\varphi(A^*)$.

The following example shows that the $F$-group of an $F$-maximal
bifix code is not always transitive.

\begin{example}
  Let $X=\{ab,ba\}$ and let $F=F((ab)^*)$. Then $X$ is an $F$-maximal
  bifix code of $F$-degree $2$. It can be verified easily that the
  syntactic monoid of $X^*$ contains only trivial subgroups (see also
  Exercises~7.1.1, 7.2.1 in~\cite{BerstelPerrinReutenauer2009}). Thus
  $G_F(X)$ is reduced to the identity.
\end{example}

\begin{figure}[hbt]
  \centering
  \gasset{AHnb=0,Nadjust=wh}
  \begin{picture}(60,35)(0,3)
    \node(1)(0,20){$1$}\node(a)(10,30){$2$}\node(b)(10,10){$7$}
    \node(aa)(20,35){$1$}\node(ab)(20,25){$3$}\node(ba)(20,10){$8$}
    \node(aba)(30,25){$4$}\node(baa)(30,15){$9$}\node(bab)(30,5){$6$}
    \node(abaa)(40,30){$5$}\node(abab)(40,20){$1$}
    \node(baab)(40,15){$1$}\node(baba)(40,5){$1$}
    \node(abaab)(50,30){$6$}\node(abaaba)(60,30){$1$}
    \drawedge(1,a){$a$}\drawedge[ELside=r](1,b){$b$}
    \drawedge(a,aa){$a$}\drawedge[ELside=r](a,ab){$b$}\drawedge(b,ba){$a$}
    \drawedge(ab,aba){$a$}\drawedge(ba,baa){$a$}\drawedge[ELside=r](ba,bab){$b$}
    \drawedge(aba,abaa){$a$}\drawedge[ELside=r](aba,abab){$b$}
    \drawedge[ELside=r](baa,baab){$b$}\drawedge(bab,baba){$a$}
    \drawedge(abaa,abaab){$b$}\drawedge(abaab,abaaba){$a$}
  \end{picture}
  \caption{An $F$-maximal bifix code of $F$-degree
    4.}\label{figureCodeGiuseppina2}
\end{figure}
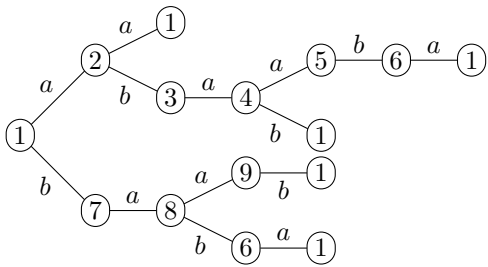

\begin{example} We consider again the code of
  Example~\ref{exampleCodeGiuseppina}.  The minimal automaton of $X^*$
  is represented on Figure~\ref{figureCodeGiuseppina2}. 

\begin{figure}
\begin{displaymath}
\def\rb{\hspace{2pt}\raisebox{0.8ex}{*}}\def\vh{\vphantom{\biggl(}}
    \begin{array}%
    {r|@{}l@{}c|@{}l@{}c|@{}l@{}c|}%
    \multicolumn{1}{r}{}&\multicolumn{2}{c}{1,2,4,8}&\multicolumn{2}{c}{1,2,5,9}
    &\multicolumn{2}{c}{1,3,6,7}\\
    \cline{2-7}
    1|2,6|3|7&\vh  & & &a^2&\vh\rb&  \\
    \cline{2-7}
    1|2|4,9|5,8&\vh\rb&ba&\vh\rb& &\vh& b\\
    \cline{2-7}
    1|2,6|3,8|4,7&\vh\rb&aba &\vh& &\vh\rb&ab\\
    \cline{2-7}
    \end{array}
\end{displaymath}
\caption{The $\D$-class of rank $4$.}\label{figDclass}
\end{figure}

We have represented on Figure~\ref{figDclass} the $\D$-class of
elements of rank $4$ meeting $\varphi(F)$. It is composed of three
$\LL$-classes and three $\RR$-classes. Each $\LL$-class is represented
by a column and each $\RR$-class by a row. On top of each column, we
have indicated the common image of the its elements. On the left of
each row, we have indicated the common nuclear equivalence of its
elements (recall that two elements are equivalent for the nuclear
equivalence if and only if they have the same image). The $\H$-classes
containing an idempotent are indicated by a star. Each $\H$-class has
four elements, and five of them are groups (this happens when the
image is a system of representatives of the nuclear equivalence). For
instance, the five classes in the nuclear equivalence of $\varphi(ba)$
are $\{1\}$, $\{2\}$, $\{4,9\}$, $\{5\}$ and $\{8\}$, and 
the $\H$-class of (the image of) $ba$ is composed of the
following elements:
\begin{displaymath}
  \begin{array}{ll} 
\text{\upshape{Word}}&\text{\upshape{Permutation}}\\\hline
ba& (18)(24)\\
baaba&(12)(48)\\
baba&(1)\\
babaaba&(14)(28)
  \end{array}
\end{displaymath}
The structure group of this $\D$-class is the Abelian group
$\Z/2\Z\times\Z/2\Z$. It is the $F$-group of the code.
\end{example}


The aim of this section is to prove the following theorem.

\begin{theorem}\label{newTheorem}
  Any transitive permutation group of degree $d$ which can be
  generated by $k$ elements is a syntactic group of a bifix code with
  $(k-1)d+1$ elements.
\end{theorem}

Theorem~\ref{newTheorem} was known before in particular cases.
In~\cite{Perrin1981} it is shown that any transitive permutation
group is a syntactic group of a finite bifix code. The bound
$d+1$ on the cardinality of the bifix code is proved
for the case of a group generated by
a $d$-cycle and another permutation. In \cite{Rindone1987}, it is
proved that for an Abelian group of rank $2$ and order $d$ there
exists a bifix code $X$ such that $\Card(X)-1=d$ . The proof is based
on the fact that the Cayley graph of an Abelian group contains a
Hamiltonian cycle.  

Let us call \emph{minimal rank} of a group $G$ the minimal cardinality
of a generating set for $G$.  Theorem~\ref{newTheorem} is related to
the following conjecture~\cite{Perrin1981}.

{\sl Let $X$ be a finite bifix code and let $G$ be a transitive
  permutation group of degree $d$ and minimal rank $k$. If $G$ is a
  syntactic group of $X$, then $\Card(X)\ge (k-1)d+1$.}

Theorem~\ref{newTheorem} shows that the lower bound is sharp.

The following result, which is  from~\cite{PerrinRindone2003}, shows
that the conjecture holds for $k=2$.

\begin{theorem}\label{theoremPR}
  Let $G$ be a permutation group of degree $d$. If $G$ is a nonspecial
  syntactic group of a finite prefix code $X$, then $\Card(X)\ge d+1$.
\end{theorem}


Theorem~\ref{theoremPR} is clearly not true for special syntactic groups
since $\Z/n\Z$ is a syntactic group of $X=a^n$ for any $n\ge 1$.

Theorem~\ref{newTheorem} is a consequence of the following theorem
which can be viewed as a complement to Theorem~\ref{proposition1}.
The proof itself makes use of  Theorem~\ref{theoremGroups}.

\begin{theorem}\label{theoremGroupCodes}
  Let $Z\subset A^*$ be a group code of degree $d$. Let $F$ be a
  Sturmian set. The set $X=Z\cap F$ is an $F$-maximal bifix code of
  $F$-degree $d$ and $G_F(X)=G(Z)$.
\end{theorem}

\begin{proof}
The fact that $X$ is an $F$-maximal bifix code of $F$-degree $d$
results from Corollary~\ref{newCorollary}.

Let us show that $G_F(X)=G(Z)$.
Let $\B=(R,1,1)$ be the minimal automaton of $Z^*$. Set
$\psi=\varphi_\B$ and $G=\psi(A^*)$.
Thus
$G$ is a permutation group equivalent to $G(Z)$.

Let $\A=(Q,1,1)$ be the minimal automaton of $X^*$. Set
$\varphi=\varphi_\A$. Denote by $\Im(w)$ the image of $\varphi(w)$
with respect to $\A$.
Thus $\Im(w)=\{t\in Q\mid s\cdot w =t\text{ for some } s\in Q\}$.

 Let $u\in F$ be a word with $d$ parses with respect to $X$. Let
 $I=\Im(u)$.
By Lemma~\ref{lemmaRank1}, the word $u$ has rank $d$ and thus
$\Card(I)=d$.

Let $Y=R_F(u)$ be the set of first return words to $u$.  By
Theorem~\ref{propositionGamma}, the set $Y$ is a basis of the free
group $A^\circ$.  For any $y\in Y$, the restriction of $\varphi(y)$ to
$I$ is a permutation of $I$. Indeed, $uy\in A^+u$ implies
$\Im(uy)\subset I$. Since $uy\in F$, the set $\Im(uy)$ has $d$
elements by Lemma~\ref{lemmaRank2}.  Thus $\Im(uy)= I$. Since
$\Im(u)=I$, this proves the claim.

Let $e$ be an idempotent in $\varphi(Y^+)$.  The restriction of $e$ to
$I$ is the identity.  Any long enough element of $\varphi^{-1}(e)\cap
Y^*$ has $u$ as a suffix. Thus the image of $e$ is $I$. Moreover,
since $\varphi(u)e=\varphi(u)$ and $e\in\varphi(A^*u)$, $e$
and $\varphi(u)$ belong to the same $\LL$-class and thus to the
$\D$-class. Thus $e$ belongs to the $\D$-class of $\varphi(A^*)$ 
which contains the elements of rank $d$ in $\varphi(F)$.

Let $G'$ be the maximal group contained in $\varphi(A^*)$
which contains $e$. It is a permutation group on $I$ which is
equivalent to $G_F(X)$.

For $y\in Y^*$, let $\chi(y)$ be the restriction of $\varphi(y)$
to the set $I$.

For any $y\in Y^*$, $e\varphi(y)e$ has the same nuclear equivalence
and the same image as $e$. By Proposition~\ref{propIdempotent}
it implies that they are in the same $\H$-class. Thus $e\varphi(y)e$
is in $G'$.

Since $e\varphi(y)e$ and $\varphi(y)$
have the same restriction to $I$ and since $e\varphi(y)e$ belongs
to the $\H$-class of $e$, $\chi$ is a morphism from $Y^*$ into 
the permutation group $G'$.
Since $Y$ generates $A^\circ$, this morphism is surjective. 
Indeed, if $\varphi(w)\in G'$,
let
$y_1,\ldots,y_n\in Y$ be such that $w=y_1^{\varepsilon_1}\cdots
y_n^{\varepsilon_n}$
with $\varepsilon_i=\pm 1$. Then $\chi(w)=\chi(y_1)^{\varepsilon_1}\cdots
\chi(y_n)^{\varepsilon_n}$. Since $G'$ is a finite group
$\chi(y)^{-1}\in\chi(Y^*)$.
Thus $\chi(w)\in\chi(Y^*)$.

Let us show that $G$ and $G'$ are equivalent as permutation
groups. 





For this, let us define a bijection $\beta:I\to R$ as follows. Let $P$
be the set of proper prefixes of the words of $X$ and let $S$ be the
set of elements of $P$ which are suffixes of $u$.  For $i\in I$, there
is a unique $q\in S$ such that $i=1\cdot q$ by
Lemma~\ref{lemmaRank1}. Set $\beta(i)=1\psi(q)$. We show that $\beta$
is injective. Let $q,t\in S$ be such that $1\psi(q)=1\psi(t)$. Assume
that $|q|\le |t|$.  Since $q,t$ are suffix-comparable, we have
$t=vq$. Since $1\psi(t)=1\psi(v)\psi(q)=1\psi(q)$ and since $\psi(q)$
is a permutation, we have $1\psi(v)=1$ and thus $v\in Z^*$.  Since $v$
is in $F$ and since $Z^*\cap F\subset X^*$, this implies $v\in X^*$ and thus $v=1$.  This shows that
$q=t$ and thus that $\beta$ is injective.  Since
$\Card(R)=\Card(I)=d$, we have shown that $\beta$ is a bijection.

Let us verify that for any $i,j\in I$ and $y\in Y^*$, we have
\begin{equation}
  i\varphi(y)=j \quad \iff \quad\beta(i)\psi(y)=\beta(j).\label{eqdiagram}
\end{equation}

Let us first prove~\eqref{eqdiagram} for $y\in Y$.  For this, let
$q,t\in S$ be such that $i=1\cdot q$, $j=1\cdot t$.  The states
$q,t$ exist by Lemma~\ref{lemmaRank1}.  Then
\begin{displaymath}
  i\varphi(y)=j \ \iff
  \ 1\varphi(qy)=1\varphi(t)\ \iff \
  qy\in X^*t.
\end{displaymath}
The last equivalence holds because $1\cdot qy=1\cdot v$ for the word
$v\in P$ such that $qy\in X^*v$.  But since $uy\in A^*u$, $v$ is a
suffix of $u$ and thus $v\in S$. This forces $t=v$.

Since $qy\in F$, we have 
\begin{displaymath}
qy\in X^*t \ \iff \ qy\in Z^*t
\end{displaymath}
and thus, we obtain
\begin{displaymath}
i\varphi(y)=j \ \iff \ qy\in Z^*t \ \iff \ \beta(i)\psi(y)=\beta(j).
\end{displaymath}
This proves \eqref{eqdiagram} for $y\in Y$. Next, let us show that if
$y,z\in Y^*$ satisfy~\eqref{eqdiagram} for all $i,j\in I$, the same is
true for $yz$. Assume first that for $i,j\in I$, one has
$i\varphi(yz)=j$.  Since the restrictions of $\varphi(y),\varphi(z)$
to $I$ are permutations, there is a unique $k\in I$ such that
$i\varphi(y)=k$ and $k\varphi(z)=j$.  Then, since $y,z$ satisfy
\eqref{eqdiagram}, we have $\beta(i)\psi(y)=\beta(k)$ and
$\beta(k)\psi(z)=\beta(j)$. Thus $\beta(i)\psi(yz)=\beta(j)$.
Conversely, assume that $\beta(i)\psi(yz)=\beta(j)$. Since $\beta$ is
a bijection from $I$ onto $R$, there is a unique $k\in I$ such that
$\beta(k)=\beta(i)\psi(y)$. Then $\beta(k)\psi(z)=\beta(j)$. By
\eqref{eqdiagram}, we have $i\varphi(y)=k$ and $k\varphi(z)=j$ whence
$i\varphi(yz)=j$.  This proves that $yz$ satisfies \eqref{eqdiagram}.

Equation~\eqref{eqdiagram} shows that we may define a morphism $\alpha$
from $G'$ to $G$ by $\alpha(g)=\psi(y)$ for $y\in Y^*$ such that
$\chi(y)=g$. This map is injective. Indeed, if $\alpha(g)=\alpha(g')$,
let
$y,y'\in Y^*$ be such that $\chi(y)=g$ and $\chi(y')=g'$. Then, 
$\alpha(g)=\psi(y)$ and $\alpha(g')=\psi(y')$ imply that $\psi(y)=\psi(y')$.
By \eqref{eqdiagram}, $\psi(y)=\psi(y')$ implies that
$\chi(y)=\chi(y')$
and thus $g=g'$. Since $Y$ generates the free group $A^\circ$, the map
is surjective. Indeed, for any $a\in A$ we have
$a=y_1^{\epsilon_1}\cdots y_n^{\epsilon_n}$ with $y_i\in Y$ and
$\epsilon_i=\pm 1$.
Thus 
$\psi(a)=\psi(y_1)^{\epsilon_1}\cdots\psi(y_n)^{\epsilon_n}=\alpha(g_1^{\epsilon_1}\cdots
g_n^{\epsilon_n})$ with $\chi(y_i)=g_i$.

Finally, the commutative diagrams of Figure~\ref{diagram} show that
the pair $(\alpha,\beta)$ is an equivalence of permutation groups.

\begin{figure}[hbt]
\centering
\gasset{Nframe=n,Nadjust=wh}
\begin{picture}(80,25)
\node(1)(0,20){$1\cdot q$}\node(2)(30,20){$1\cdot qy$}
\node(3)(0,0){$1\psi(q)$}\node(4)(30,0){$1\psi(qy)$}
\drawedge(1,2){$y$}\drawedge(3,4){$\psi(y)$}
\drawedge(1,3){}\drawedge(2,4){}

\node(I1)(50,20){$I$}\node(I2)(80,20){$I$}
\node(R1)(50,0){$R$}\node(R2)(80,0){$R$}
\drawedge(I1,I2){$g$}\drawedge(R1,R2){$\alpha(g)$}
\drawedge(I1,R1){$\beta$}\drawedge(I2,R2){$\beta$}
\end{picture}
\caption{The equivalence of $G$ and $G'$.}\label{diagram}
\end{figure}
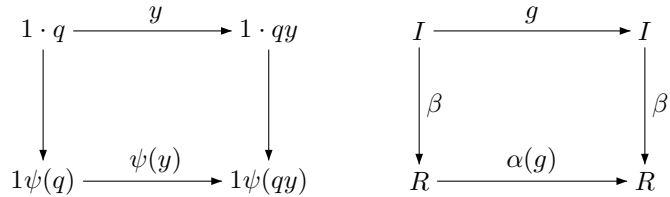

\end{proof}


\begin{example} We illustrate the proof of
  Theorem~\ref{theoremGroupCodes}. Let $Z$ be the group code of
  degree~$4$ recognized by the automaton of Figure~\ref{figureGroupX}.
  It is the automaton of Figure~\ref{figureGroup} with more convenient
  labels for the states.
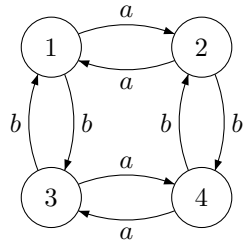
\begin{figure}[hbt]
\centering
\begin{picture}(30,30)(0,-3)
\node(1)(0,20){$1$}\node(a)(20,20){$2$}
\node(aba)(0,0){$3$}\node(ba)(20,0){$4$}
\drawedge[curvedepth=3](1,a){$a$}\drawedge[curvedepth=3](a,1){$a$}
\drawedge[curvedepth=3](1,aba){$b$}\drawedge[curvedepth=3](aba,1){$b$}
\drawedge[curvedepth=3](a,ba){$b$}\drawedge[curvedepth=3](ba,a){$b$}
\drawedge[curvedepth=3](aba,ba){$a$}\drawedge[curvedepth=3](ba,aba){$a$}
\end{picture}
\caption{A group automaton.}\label{figureGroupX}
\end{figure}
  It is clear the $G(Z)$ is $\Z/2\Z\times \Z/2\Z$.  Let $F$ be the
  Fibonacci set.  The code $X=Z\cap F$ is the code of 
  Example~\ref{exampleCodeGiuseppina}.  The minimal automaton of $X^*$
  is represented on Figure~\ref{figureCodeGiuseppina2}.  Let us chose
  $u=aba$. It has rank~$4$, and $\Im(u)=\{1,2,4,8\}$. One gets
  $Y=\{ba,aba\}$. Next $\chi(ba)=(18)(24)$ and $\chi(aba)=(14)(28)$.
  The function $\beta$ maps $1,2,4,8$ to $1,2,3,4$ respectively.
\end{example}



The following example shows that Theorem~\ref{theoremGroupCodes}
does not hold for the set of factors of an episturmian word which is
not strict.

\begin{example}\label{exNonStrictGroup}
Let $F$ and $X$ be as in Example~\ref{exampleNotStrict}.
The bifix code $X$ has $F$-degree $8$.
Let $\A$ be the minimal automaton of $X^*$ represented on
Figure~\ref{figureExNonStrict}. The image of $\varphi(bc)$ is the
set $I=\{1,4,6,13,14,21,25,32\}$. The submonoid $U=\{u\in A^*\mid
I\cdot u=I\}$ is generated by $acbc$ and $acacbc$. The restrictions to
$I$ of $\varphi(acbc)$ and $\varphi(acacbc)$ are
\begin{displaymath}
(1\ 14)(25\ 6)(4\ 21)(32\ 13),\quad (1\ 6)(14\ 25)(4\ 13)(21\ 32).
\end{displaymath}
These permutations generate a group which has two orbits:
$\{1,6,14,25\}$ and $\{4,13,21,32\}$. The restriction 
to each orbit is isomorphic to $(\Z/2\Z)^2$.
 Thus 
the $F$-group of $X$ is $(\Z/2\Z)^2$. However $X=Z\cap F$ where
$Z$ is a group code such that $G(Z)=(\Z/2\Z)^3$.
\end{example}


\begin{proofof}{of Theorem~\ref{newTheorem}}
  Let $G$ be a transitive permutation group of degree $d$ and let $Z$
  be a group code on an alphabet $A$ with $k$ letters such that
  $G(Z)=G$. Let $F$ be a Sturmian set on the alphabet $A$ and let
  $X=Z\cap F$. Then, by Theorem~\ref{theoremGroupCodes}, $G_F(X)= G$
  and, by Theorem~\ref{theoremBifixd+1}, $X$ has $(k-1)d+1$ elements.
\end{proofof}


\paragraph{Acknowledgments} We wish to thank Mike Boyle, Aldo De Luca, 
Thierry Monteil, Patrice S\'e\'ebold,
Martine Queff\'elec and Gw\'ena\"el Richomme for their help in the preparation of this manuscript.

\bibliography{bifixesSturm}
\bibliographystyle{plain}
\printindex

\end{document}